\documentclass[leqno,a4paper,12pt]{report}
\usepackage{amssymb,amsmath,amsthm}
\usepackage[alphabetic,lite]{amsrefs}

\usepackage[all]{xy}
\usepackage{color}

\numberwithin{equation}{section}

\usepackage{enumerate}
\usepackage{mathrsfs}
\usepackage{color}
\usepackage[colorlinks=true, pdfstartview=FitV, linkcolor=blue, citecolor=blue, urlcolor=blue]{hyperref}

\newcommand{\nc}{\newcommand}
\nc{\on}{\operatorname}

\newcommand{\spa}{\vspace{0.3ex}\noindent}
\newcommand{\spaa}{\vspace{0.5ex}\noindent}

\newtheorem{theorem}{Theorem}[section]
\newtheorem{proposition}[theorem]{Proposition}
\newtheorem{lemma}[theorem]{Lemma}
\newtheorem{corollary}[theorem]{Corollary}

\theoremstyle{definition}

\newtheorem{definition}[theorem]{Definition}
\newtheorem{notation}[theorem]{Notation}
\newtheorem{example}[theorem]{Example}

\newtheorem{remark}[theorem]{Remark}

\newtheorem{conjecture}[theorem]{Conjecture}

\nc{\RR}{\mathrm{R}}
\nc{\LL}{\mathrm{L}}
\newcommand{\RC}{{\rm C}}

\newcommand{\C}{{\mathbb{C}}}
\newcommand{\N}{{\mathbb{N}}}
\newcommand{\R}{{\mathbb{R}}}

\newcommand{\Z}{{\mathbb{Z}}}

\newcommand{\BBV}{\mathbb{V}}

\newcommand{\Fil}{\operatorname{F}}

\def\phi{{\varphi}}
\def\epsilon{\varepsilon}

\newcommand{\cor}{{\bf k}}

\def\sha{\mathscr{A}}

\def\shb{\mathscr{B}}
\def\shhb{{\mathscr{B}}}
\def\shc{\mathscr{C}}
\def\shhc{{\mathscr{C}}}
\def\shd{\mathscr{D}}
\def\shhd{{\mathscr{D}}}

\def\shf{\mathscr{F}}

\def\shh{\mathscr{H}}
\def\shi{\mathscr{I}}
\def\shj{\mathscr{J}}

\def\shl{\mathscr{L}}
\def\shm{\mathscr{M}}

\def\sho{\mathscr{O}}

\def\shs{\mathscr{S}}
\def\sht{\mathscr{T}}
\def\shu{\mathscr{U}}

\renewcommand{\ker}{\operatorname{Ker}}

\DeclareMathOperator{\im}{Im}

\newcommand{\into}{\hookrightarrow}
\newcommand{\vvert}{\vert\mspace{-2mu}\vert}

\renewcommand{\to}[1][]{\xrightarrow[]{#1}}

\newcommand{\isoto}[1][]{\xrightarrow[#1]%
{{\raisebox{-.6ex}[0ex][-.6ex]{$\mspace{1mu}\sim\mspace{2mu}$}}}}

\newcommand{\tto}{\rightrightarrows}

\newcommand{\muHom}[1][]{\mathrm{Hom}^\mu_{\raise1.5ex\hbox to.1em{}#1}}
\newcommand{\Hom}[1][]{\mathrm{Hom}_{\raise1.5ex\hbox to.1em{}#1}}
\newcommand{\RHom}[1][]{\RR\mathrm{Hom}_{\raise1.5ex\hbox to.1em{}#1}}
\newcommand{\Ext}[2][]{\mathrm{Ext}_{\raise1.5ex\hbox to.1em{}#1}^{#2}}
\renewcommand{\hom}[1][]{{\mathscr{H}\mspace{-4mu}om}_{\raise1.5ex\hbox to.1em{}#1}}
\newcommand{\rhom}[1][]{{\RR\mathscr{H}\mspace{-3mu}om}_{\raise1.5ex\hbox to.1em{}#1}}
\newcommand{\rhomc}[1][]
{{\mathscr{H}\mspace{-3mu}om}^*_{\raise1.5ex\hbox to.1em{}#1}}

\newcommand{\ext}[2][]{{\mathscr{E}xt}_{\raise1.5ex\hbox to.1em{}#1}^{#2}}
\newcommand{\Tor}[2][]{\mathrm{Tor}^{\raise1.5ex\hbox to.1em{}#1}_{#2}}
\newcommand{\tens}[1][]{\mathbin{\otimes_{\raise1.5ex\hbox to-.1em{}{#1}}}}

\newcommand{\Endo}[1][]{\mathrm{End}_{\raise1.5ex\hbox to.1em{}#1}}

\newcommand{\Aut}[1][]{\mathrm{Aut}_{\raise1.5ex\hbox to.1em{}#1}}
\newcommand{\sect}{\Gamma}
\newcommand{\rsect}{\mathrm{R}\Gamma}

\newcommand{\Cov}{\mathrm{Cov}}

\newcommand{\oim}[1]{{#1}_*}

\newcommand{\eim}[1]{{#1}_!}
\newcommand{\roim}[1]{\RR{#1}_*}

\newcommand{\reim}[1]{\RR{#1}_!}
\newcommand{\opb}[1]{#1^{-1}}

\newcommand{\opbsal}[1]{#1^{-1}_{\rm sal}}
\newcommand{\oimsal}[1]{#1_{\rm sal*}}
\newcommand{\opbdag}[1]{#1^{\dag}}
\newcommand{\opbddag}[1]{#1^{\ddag}}

\newcommand{\epb}[1]{#1^{!}}

\newcommand{\oeim}[1]{{#1}_{*!}}

\newcommand{\md[1]}{{\rm Mod}({#1})}

\newcommand{\mdrc[1]}{{\rm Mod_{{\Rc}}}({#1})}

\newcommand{\eqdot}{\mathbin{:=}}

\newcommand{\cl}{\colon}
\newcommand{\scbul}{{\,\raise.4ex\hbox{$\scriptscriptstyle\bullet$}\,}}

\newcommand{\tw}[1]{\widetilde{#1}}

\newcommand{\ol}{\overline}
\newcommand{\bl}{\bigl(}
\newcommand{\br}{\bigr)}
\newcommand{\lp}{{\rm(}}
\newcommand{\rp}{{\rm)}}

\newcommand{\SP}{\rm SP}

\newcommand{\Cc}{{\C\text{-c}}}
\newcommand{\Rc}{{\R\text{-c}}}

\newcommand{\ba}{\begin{array}}
\newcommand{\ea}{\end{array}}

\newcommand{\bnum}{\begin{enumerate}[{\rm(i)}]}
\newcommand{\enum}{\end{enumerate}}
\newcommand{\banum}{\begin{enumerate}[{\rm(a)}]}
\newcommand{\eanum}{\end{enumerate}}

\newcommand{\eq}{\begin{eqnarray}}
\newcommand{\eneq}{\end{eqnarray}}
\newcommand{\eqn}{\begin{eqnarray*}}
\newcommand{\eneqn}{\end{eqnarray*}}

\nc{\Proof}{\begin{proof}}
\nc{\QED}{\end{proof}}

\def\rop{{\rm op}}

\def\Op{{\rm Op}}

\def\holreg{{\rm hol reg}}

\def\sa{{\rm sa}}
\def\sb{{\rm sb}}
\def\sal{{\rm sal}}
\def\sbl{{\rm sbl}}
\def\Msa{{M_{\sa}}}
\def\Msb{{M_{\sb}}}
\def\Msal{{M_{\sal}}}
\def\Msbl{{M_{\sbl}}}
\def\Nsal{{N_{\sal}}}
\def\Nsa{{N_{\sa}}}
\def\Nsb{{N_{\sb}}}
\def\Nsal{{N_{\sal}}}
\def\Nsbl{{N_{\sbl}}}
\def\Nsalf{{N^f_{\sal}}}

\def\Nf{{N^f}}
\def\Lh{{L^h}}
\def\Lg{{L^g}}
\def\Lsa{{L_{\sa}}}

\def\Xsa{{X_{\sa}}}
\def\Xsal{{X_{\sal}}}

\def\olXsa{{\ol X_{\sa}}}
\def\olXsal{{\ol X_{\sal}}}

\def\XRsal{{X^\R_{\sal}}}

\def\Usa{U_{\sa}}

\def\rhosa{\rho_{\rm sa}}
\def\rhosal{\rho_{\sal}}

\def\rhosl{\rho_{\rm sl}}

\def\fsbl{f_{\sbl}}

\DeclareMathOperator{\id}{id}

\newcommand{\Der}[1][]{\mathsf{D}^{#1}}
\newcommand{\Derb}{\Der[\mathrm{b}]}

\newcommand{\RD}{\mathrm{D}}
\newcommand{\rb}{{\mathrm b}}
\newcommand{\ub}{{\rm ub}}

\newcommand{\mon}{\Lambda}

\newcommand{\Psh}{\operatorname{PSh}}

\newcommand{\Fct}{\operatorname{Fct}}

\newcommand{\for}{\mathit{for}}

\newcommand{\Int}{{\rm Int}}

\newcommand{\Sol}{{\rm Sol}}

\newcommand{\dT}{{\dot{T}}}

\newcommand{\indlim}[1][]{\mathop{\varinjlim}\limits_{#1}}
\newcommand{\sindlim}[1][]{\smash{\mathop{\varinjlim}\limits_{#1}}\,}

\newcommand{\prolim}[1][]{\mathop{\varprojlim}\limits_{#1}}

\newcommand{\inddlim}[1][]{\mathop{``{\varinjlim}"}\limits_{#1}}

\nc{\eps}{\varepsilon}
\nc{\hs}{\hspace*}
\nc{\nn}{\nonumber}
\nc{\tM}{\widetilde{M}}
\nc{\h}{\mathbf{h}}
\nc{\tf}{\tilde{f}}
\nc{\codim}{\on{codim}}
\nc{\lh}{\mathscr{H}}
\nc{\bwr}{\mbox{\large{$\wr$}}}
\nc{\dTi}{\dT^{*,\mathrm{in}}}
\nc{\Cd}{\mathrm{C}}
\nc{\BX}[1][]{\shb_{#1\vert X}}
\nc{\BXsa}[1][]{\shb_{#1\vert \Xsa}}

\newcommand{\Dbt}{{\cal D} b^\tp}
\newcommand{\Db}{{\cal D} b}

\nc{\Fwhat}{\mathrm{F}_{\star}}
\nc{\Fi}{\mathrm{F}}

\nc{\Ftwo}{\mathrm{F}_{2}}
\nc{\Finf}{\mathrm{F}_{{\small \mspace{-4mu}\infty}}}
\nc{\Finfsa}{\mathrm{F}_{{\small \mspace{-4mu}\infty,\sa}}}
\nc{\RHFinf}{\mathrm{RHF}_{{\small \mspace{-4mu}\infty}}}
\nc{\RHFinfsa}{\mathrm{RHF}_{{\small \mspace{-4mu}\infty,\sa}}}

\newcommand{\tp}{{\rm tp}}

\newcommand{\Cinft}{{\shc}^{\infty,\tp}}
\newcommand{\Cinftt}{{\shc}^{\infty,{\tp\, st}}}
\newcommand{\Cinf}{\shc^\infty}
\newcommand{\Cin}[1][]{{\shc}^{\infty,{#1}}}

\newcommand{\Ot}{{\sho}^{\tp}}
\newcommand{\Ott}{{\sho}^{\tp\, st}}

\newcommand{\Ga}[1][]{{\shc}^{\infty,{\rm gev}(#1)}}
\newcommand{\Gb}[1][]{{\shc}^{\infty,{\rm gev}\{#1\}}}
\newcommand{\Gc}{{\shc}^{\infty,{\rm gev\, st}}}
\newcommand{\GG}{{\shc}^{\infty,{\rm gev}}}

\newcommand{\OGa}[1][]{{\sho}^{{\rm gev}(#1)}}
\newcommand{\OGb}[1][]{{\sho}^{{\rm gev}\{#1\}}}
\newcommand{\OG}{{\sho}^{\rm gev}}
\newcommand{\OGc}{{\sho}^{\rm gev\, st}}

\newenvironment{rouge}
{\relax\color{red}}
{\hspace*{.3ex}\relax}
\newcommand{\ber}{\begin{rouge}}
\newcommand{\er}{\end{rouge}}

\newenvironment{bleu}
{\relax\color{blue}}
{\hspace*{.3ex}\relax}
\newcommand{\beb}{\begin{bleu}}
\newcommand{\eb}{\end{bleu}}

\begin{document}
\author{St{\'e}phane  Guillermou  and Pierre Schapira}
\title{Construction of sheaves on the subanalytic site}
\maketitle
\tableofcontents

\chapter*{Introduction}
Let $M$ be a  real analytic manifold. The Grothendieck subanalytic topology on $M$, denoted $\Msa$, and the morphism of sites
 $\rhosa\cl M\to \Msa$, were introduced in~\cite{KS01}. 
Recall that the objects of the site  $\Msa$ are the relatively compact subanalytic 
open subsets of $M$ and the coverings are, roughly speaking,  the finite coverings. 
In loc.\ cit.\ the authors use this topology to construct new sheaves which would have no meaning on the usual topology, such as the sheaf 
$\Cinft_\Msa$ of $\Cinf$-functions with temperate growth  and the sheaf 
$\Dbt_\Msa$ of temperate distributions. 
On a complex manifold $X$, using the Dolbeault complexes, they constructed the sheaf $\Ot_\Xsa$ (in the derived sense) of holomorphic functions with temperate growth.
The last sheaf  is implicitly used in the solution of  the Riemann-Hilbert problem by Kashiwara~\cites{Ka80,Ka84} and is also extremely important in the study of  irregular holonomic $\shd$-modules (see~\cite{KS03}*{\S~7}).

\smallskip
In this paper, we shall modify the preceding construction in order to obtain sheaves 
of $\Cinf$-functions with a given  growth at the boundary. For example, functions 
whose growth at the boundary is bounded by a given  power of the distance (temperate growth of order $s\geq0$), or by an exponential of a given power 
of the distance (Gevrey growth of order $s>1$), as well as their holomorphic counterparts. 
For that purpose, we have to refine the subanalytic topology and we introduce what we call the linear subanalytic topology, denoted 
$\Msal$. 

Let us describe the contents of this paper with some details.

\smallskip
In \textbf{Chapter~\ref{ShvSa1}} we construct the linear subanalytic topology on $M$. Denoting by $\Op_\Msa$ the category of open relatively compact subanalytic subsets of $M$, the  presite underlying the site $\Msal$ is the same as for $\Msa$, namely $\Op_\Msa$, 
but the coverings are the linear coverings. Roughly speaking, a finite family $\{U_i\}_{i\in I}$ is a linear covering of their union $U$  if there is a constant $C$ such that the distance of any $x\in M$ to
$M\setminus U$ is bounded by $C$-times the maximum of the distance
of $x$ to $M\setminus U_i$ ($i\in I$). (See Definition~\ref{def:regsit1}.) In this chapter, we also prove some technical results on linear coverings that we shall need in the course of the paper. 

{\bf Chapter~\ref{ShvSa2}}. Let $\cor$ be a commutative unital Noetherian ring with finite global dimension. One easily shows that a
presheaf $F$ of $\cor$-modules on $\Msal$ is a sheaf as soon as, for any open sets $U_1$ and $U_2$ 
such that $\{U_1,U_2\}$ is a linear covering of $U_1\cup U_2$,
the Mayer-Vietoris sequence 
\eq\label{MV0}
&&0\to F(U_1\cup U_2)\to F(U_1)\oplus F(U_2)\to F(U_1\cap U_2)
\eneq
 is exact. Moreover, if for any such a covering, the sequence 
 \eq\label{MV00}
&&0\to F(U_1\cup U_2)\to F(U_1)\oplus F(U_2)\to F(U_1\cap U_2)\to 0
\eneq
is exact, then the sheaf $F$ is $\sect$-acyclic, that is,  $\rsect(U;F)$  is concentrated in degree $0$ 
for all $U\in\Op_\Msa$. 

There is a natural morphism of sites 
$\rhosal\cl\Msa\to\Msal$ and we shall prove the two results  below (see
Theorems~\ref{th:rightadj} and~\ref{th:Lipbnd}):\\
(1) the functor $\roim{\rhosal}\cl\RD^+(\cor_\Msa)\to \RD^+(\cor_\Msal)$ admits a right adjoint $\epb{\rhosal}$,\\
(2) if  $U$ has a Lipschitz boundary, then the object $\roim{\rhosal}\cor_U$
is concentrated in degree $0$.

Therefore,  if a presheaf $F$ on $\Msa$ has the property that the Mayer-Vietoris
sequences~\eqref{MV00} are exact, it follows 
that $\rsect(U;\epb{\rhosal}F)$ is concentrated in degree $0$ and is isomorphic
to $F(U)$ for any $U$ with Lipschitz boundary. In other words, to a presheaf on $\Msa$ satisfying a natural condition, we are able to associate an object of the derived category of sheaves on $\Msa$ which has the same sections as $F$ on any Lipschitz open set. This construction is in particular used by Gilles Lebeau~\cite{Le14} 
who obtains for $s\leq0$ the ``Sobolev sheaves $\shh^s_\Msa$'', objects of $\RD^+(\C_\Msa)$ with the property that if $U\in\Op_\Msa$ has a Lipschitz boundary, then $\rsect(U;\shh^s_\Msa)$ is concentrated in degree $0$ and coincides with the classical Sobolev space $H^s(U)$. 

The fact that Sobolev sheaves are objects of derived categories and are not concentrated in degree $0$ shows that when dealing with spaces of functions or distributions defined on open subsets which are not regular (more precisely, which have not a Lipschitz boundary), it is natural to replace the notion of a space by that of a complex of spaces.

\smallskip
In {\bf Chapter~\ref{ShvSa3}}, we briefly study the natural operations on the linear subanalytic sites. 
The main difficulty is that a morphism $f\cl M\to N$ of real analytic manifolds does not induce a morphism
of the linear subanalytic sites. This forces us to treat  separately the direct or inverse images of sheaves for closed embeddings and for submersive maps. 

\smallskip
In {\bf Chapter~\ref{ShvSa4}} we construct some sheaves on $\Msal$.
We construct the sheaf $\Cin[s]_\Msal$   of
$\Cinf$-functions with growth of  order $s\geq0$ at the boundary and the sheaves 
$\Ga[s]_\Msal$ and  $\Gb[s]_\Msal$ of $\Cinf$-functions 
 with Gevrey growth of type $s>1$ at the boundary. 
 By using a refined cut-off lemma (which follows from a refined partition of unity due to H\"ormander~\cite{Ho83}),  we prove that these sheaves are
 $\sect$-acyclic. Applying the functor $\epb{\rhosal}$, we get 
new sheaves (in the derived sense) on $\Msa$ whose sections on open sets with
Lipschitz boundaries are concentrated in degree $0$. 
Then, on a complex manifold $X$, by considering the Dolbeault complexes of 
the sheaves of $\Cinf$-functions  considered  above, we obtain new sheaves of holomorphic functions  with various growth.

As already mentioned,  Sobolev sheaves are treated in a separate paper by G.~Lebeau in~\cite{Le14}. 
 
Finally, in {\bf Chapter~\ref{ShvSa5}}, we apply these results to endow the sheaf
$\Ot_\Xsa$  with a filtration (in the derived sense) that we call the $L^\infty$-filtration. 

Denote by $\Fi\shd_\Msa$ the sheaf $\shd_\Msa\eqdot\eim{\rhosa}\shd_M$ of differential operators
on $\Msa$, endowed with its natural filtration and denote by $\Fi\shd_\Msal$ the sheaf 
$\shd_\Msal\eqdot\oim{\rhosal}\shd_\Msa$ endowed with its natural filtration. 
For $\sht=M,\Msa,\Msal$, the category $\md[\Fi\shd_\sht]$ of filtered $\shd$-modules on $\sht$ is quasi-abelian 
in the sense of~\cite{Sn99} and its derived category $\RD^+(\Fi\shd_\sht)$ is well-defined.
We shall use here the recent results of~\cite{SSn13} which give an easy description of these derived categories
and we construct a right adjoint $\epb{\rhosal}$ to the derived functor 
$\roim{\rhosal}\cl \RD^+(\Fi\shd_\Msa)\to \RD^+(\Fi\shd_\Msal)$. 

By considering  the sheaves $\Cin[s]_\Msal$ ($s\geq0$) we obtain the 
filtered sheaf   $\Finf\Cinft_\Msal$.
Then, on a complex manifold $X$, by considering the Dolbeault complex of 
this filtered sheaf, we obtain the filtration $\Finf \Ot_\Xsa$ on the sheaf $\Ot_\Xsa$.

Recall now the Riemann-Hilbert correspondence. Let $\shm$ be a regular 
holonomic $\shd_X$-module and let $G\eqdot\rhom[\shd_X](\shm,\sho_X)$ be the perverse sheaf of its holomorphic solutions. Kashiwara's theorem of~\cite{Ka84} may be formulated by saying that the natural morphism 
$\shm\to\opb{\rhosa}\rhom(G,\Ot_\Xsa)$ is an isomorphism. 
Replacing the sheaf $\Ot_\Xsa$ with its filtered version 
$\Finf\Ot_\Xsa$, we define the filtered Riemann-Hilbert functors $\RHFinfsa$ and 
$\RHFinf$ by the formulas 
\eqn
\RHFinfsa\cl\RD_{\holreg}^+(\shd_X)&\to&\RD^+(\Fi\shd_\Xsa),\\
\hspace{5ex}\shm&\mapsto&\Fi\rhom(\Sol(\shm), \Finf\Ot_\Xsa),\\
\RHFinf= \opb{\rhosa}\RHFinfsa\cl \RD_{\holreg}^+(\shd_X)&\to&\RD^+(\Fi\shd_X)
\eneqn
and we prove that  the composition
\eqn
&& \Derb_{\holreg}(\shd_X)\to[ \RHFinf]\RD^+(\Fi\shd_X)\to[\for] \RD^+(\shd_X)
\eneqn 
is isomorphic to the identity functor.
In other words, any regular holonomic $\shd_X$-module $\shm$ can be
{\em functorially} endowed  with a filtration $\Finf\shm$, in the derived sense. 

We also briefly introduce an $L^2$-filtration better suited to apply H\"or\-man\-der's theory (see~\cite{Ho65}) 
and present some open problems.

\vspace{1.ex}\noindent
{\bf Acknowledgments}\\
An important part of this paper has been written during two stays of the  authors at
the Research Institute for 
Mathematical Sciences at Kyoto University in 2011 and 2012 and we wish to
thank this institute for its 
hospitality. During our stays we had, as usual, extremely enlightening
discussions with Masaki Kashiwara and we warmly thank him here. 

We have also been very much stimulated by the interest of Gilles Lebeau for
sheafifying the classical Sobolev spaces
and it is a pleasure to thank him here. 

Finally  Theorem~\ref{th:ParuTh} plays an essential role in the whole paper and we are extremely grateful to  Adam Parusinski who has given a proof of this result. 

\chapter{Subanalytic topologies}\label{ShvSa1}

\section{Linear coverings}

\subsubsection*{Notations and conventions}
We shall mainly follow the notations of~\cite{KS90,KS01} and~\cite{KS06}. 

In this paper, unless otherwise specified, a manifold means a real analytic manifold.
We shall freely use the theory of subanalytic sets, due to Gabrielov and Hironaka, after the pioneering work of Lojasiewicz. A short presentation of this theory may be found in~\cite{BM88}.

For a subset $A$ in a topological space $X$, $\ol A$ denotes its closure,
$\Int\, A$ its interior and $\partial A$ its boundary,
$\partial A=\ol A\setminus \Int\, A$.

Recall that given two metric spaces $(X,d_X)$ and $(Y,d_Y)$, a  function $f\cl X\to Y$ is Lipschitz if there exists a constant $C\geq0$ such that 
$d_Y(f(x),f(x'))\leq C\cdot d_X(x,x')$ for all $x,x'\in X$.

\eq\label{eq:distonM}
&&\left\{\parbox{60ex}{
All along this paper, if $M$ is a real analytic manifold, we choose a distance $d_M$ on $M$ such that, for any $x\in M$ and any local
chart $(U,\phi\cl U \hookrightarrow \R^n)$ around $x$, there exists a
neighborhood of $x$ over which $d_M$ is Lipschitz equivalent to the pull-back
of the Euclidean distance by $\phi$. If there is no risk of confusion, we write $d$ instead of $d_M$.
}\right.
\eneq
In the following, we will adopt the convention
\eq\label{eq:dist-au-vide}
d(x,\emptyset) = D_M + 1, \qquad \text{for all $x\in M$},
\eneq
where $D_M = \sup \{d(y,z);\; y,z\in M\}$.  In this way we avoid distinguishing
the special case where $M = \bigcup_{i\in I}U_i$ in~\eqref{eq:resit} below
(which can happen if $M$ is compact). 

\subsubsection*{The site $\Msa$}

The subanalytic topology was introduced in~\cite{KS01}.

Let $M$ be a real analytic manifold and denote by $\Op_{\Msa}$ the category of relatively compact 
subanalytic open subsets of 
$M$, the morphisms being the inclusion morphisms. Recall that one endows $\Op_{\Msa}$ with a
 Grothendieck topology by saying that a 
family $\{U_i\}_{i\in I}$ of objects of $\Op_{\Msa}$  is a covering of $U\in\Op_{\Msa}$ if $U_i\subset U$ for
 all $i\in I$ and there exists a finite subset $J\subset I$ such that
 $\bigcup_{j\in J}U_j=U$.
It follows from the theory of subanalytic sets that in this situation there exist a constant $C>0$ and a positive integer $N$ such that  
\eq\label{eq:resitsuba}
&&d(x,M\setminus U)^N\leq C\cdot(\max_{j\in J}  d(x,M \setminus U_j)).
\eneq

One shall be aware that if $U$ is an open subset of $M$, we may endow it with the subanalytic topology $\Usa$,
but this topology does not coincide in general with the topology induced by $M$.

We denote by $\rhosa\cl M\to\Msa$ (or simply $\rho$) the natural morphism of
sites.

\subsubsection*{The site $\Msal$}

\begin{definition}\label{def:regsit1}
Let $\{U_i\}_{i\in I}$ be a finite family in $\Op_{\Msa}$.
We say that this family is $1$-regularly situated 
if there is a constant $C$ such that for any $x\in M$
\eq\label{eq:resit}
&&d(x,M\setminus\bigcup_{i\in I}U_i)\leq C\cdot
\max_{i\in I}  d(x,M \setminus U_i).
\eneq
\end{definition}
Of course, this definition does not depend on the choice of the distance $d$. 

When $M = \R^n$ and $U \subset \R^n$ we have $d(x,M\setminus U) =
d(x,\partial U)$, for all $x\in U$.  In general we have the following comparison
result.
\begin{lemma}\label{lem:distbord}
Let $U\in \Op_{\Msa}$ be such that $\partial U$ is non empty (that is, $U$ is
not a union of connected components of $M$).  Then there exists $C>0$ such that
for all $x\in U$ we have
$$
d(x,M\setminus U) \leq d(x,\partial U) \leq C\, d(x,M\setminus U).
$$
\end{lemma}
\begin{proof}
The first inequality is clear and we prove the second one.  If it is false,
there exist $x_n \in U$, $n\in \N$, such that $d(x_n,\partial U)/ d(x_n,
M\setminus U) \to[n\to \infty] \infty$.  Since $\ol{U}$ is compact, up to taking
a subsequence we may assume that $x_n$ converges to a point $x \in \ol U$.  We
see easily that $x\in \partial U$.  We take a chart around $x$ as
in~\eqref{eq:distonM}.  Since $d_{\R^n}(y,\partial U) = d_{\R^n}(y, M\setminus
U)$ for $y$ in the chart near $x$, we can not have $d(x_n,\partial U)/ d(x_n,
M\setminus U) \to[n\to \infty] \infty$, which proves the result.
\end{proof}

\begin{example}
Let $U_1,U_2 \in \Op_{\Msa}$ be two disjoint open sets. We prove that
$\{U_1,U_2\}$ is $1$-regularly situated.  We set $U = U_1 \cup U_2$.  We argue
as in the proof of Lemma~\ref{lem:distbord} and assume by contradiction that
there exists a sequence $x_n \in U$, $n\in \N$, such that $d(x_n, M\setminus U)
/ \max_{i=1,2}\{d(x_n,M\setminus U_i)\}$ converges to $\infty$.  We may as well
assume $x_n \in U_1$ for all $n$.  Up to taking a subsequence we may assume that
$x_n$ converges to a point $x \in \ol{U_1}$.  We see that $x\in \partial U_1$.
We take a chart around $x$ as in~\eqref{eq:distonM}.  Then, for $n\gg 0$,
$d_{\R^d}(x_n, M\setminus U_1)$ is realized by a point $y_n \in \partial U_1$.
Since $U_2 \cap \partial U_1 = \emptyset$ we have in fact $y_n\in M\setminus U$.
Hence $d_{\R^d}(x_n,M\setminus U_1) = d_{\R^d}(x_n, M\setminus U)$.  Since $d$
is Lipschitz equivalent to $d_{\R^d}$, the quotient $d(x_n, M\setminus U) /
\max_{i=1,2}\{d(x_n,M\setminus U_i)\}$ remains bounded and we have a
contradiction.
\end{example}

\begin{example}
On $\R^2$ with coordinates $(x_1,x_2)$ consider the open sets:
\eqn
&&U_1=\{(x_1,x_2);\; x_2>-x_1^2,\, x_1>0\},\\
&&U_2=\{(x_1,x_2);\; x_2<x_1^2,\, x_1>0\},\\
&&U_3=\{(x_1,x_2);\; x_1>-x_2^2,\, x_2>0\}.
\eneqn
Then $\{U_1,U_2\}$ is not $1$-regularly situated. Indeed, set $W\eqdot U_1\cup U_2=\{x_1>0\}$. 
Then, if $x=(x_1,0), x_1>0$, 
$d(x,\R^2\setminus W)=x_1$ and $d(x,\R^2\setminus U_i)$ ($i=1,2$) is less that $x_1^2$. \\
On the other hand $\{U_1,U_3\}$ is $1$-regularly situated.
Indeed,  
\eqn
&&d(x,\R^2\setminus (U_1\cup U_3))
\leq \sqrt 2  \max(d(x,\R^2\setminus U_1),d(x,\R^2\setminus U_3)).
\eneqn
\end{example}

\begin{definition}\label{def:linearcov}
A linear covering of $U$ is a small family $\{U_i\}_{i\in I}$ of objects of $\Op_{\Msa}$ such that 
$U_i\subset U$ for all $i\in I$ and 
\eq\label{eq:lcoverings} 
&&\left\{\parbox{60ex}{
there exists a finite subset $I_0\subset I$ such that the  family $\{U_i\}_{i\in I_0}$ 
is $1$-regularly situated and $\bigcup_{i\in I_0}U_i=U$.
}\right.\eneq
\end{definition}
Let $\{U_i\}_{i\in I}$ and $\{V_j\}_{j\in J}$ be two families of objects of $\Op_\Msa$. Recall that one says that  
$\{U_i\}_{i\in I}$ is a refinement of $\{V_j\}_{j\in J}$ if for any $i\in I$,
there exists $j\in J$
 with $U_i\subset V_j$. 

\begin{proposition}\label{pro:linearcov=top}
The family of linear coverings satisfies the axioms
of Grothendieck topologies below \lp see~{\rm \cite[\S~16.1]{KS06}}\rp.
\\
{\rm COV1} $\{U\}$ is a covering of $U$, for any $U\in \Op_{\Msa}$. 
\\
{\rm COV2} If a covering $\{U_i\}_{i\in I}$  of $U$ is a refinement of a family 
$\{V_j\}_{j\in J}$ in $\Op_{\Msa}$ with $V_j\subset U$ for all $j\in J$, then $\{V_j\}_{j\in J}$ is a covering of $U$.  
\\
{\rm COV3} If $V\subset U$ are in $\Op_{\Msa}$ and $\{U_i\}_{i\in I}$ is a covering
of $U$, then $\{V\cap U_i\}_{i\in I}$ is a covering of $V$.  
\\
{\rm COV4} If $\{U_i\}_{i\in I}$ is a covering of $U$ and 
$\{V_j\}_{j\in J}$ is a small family in $\Op_\Msa$ with $V_j\subset U$ such that
$\{U_i \cap V_j\}_{j\in J}$ is a covering of $U_i$ for all $i\in I$, then
$\{V_j\}_{j\in J}$ is a covering of $U$.
\end{proposition}

\begin{proof} 
We shall use the obvious fact stating that for two subsets $A\subset B$ in $M$, 
we have $d(x,M\setminus A)\leq d(x, M\setminus B)$.

\medskip\noindent
COV1 is trivial. 

\vspace{0.2ex}\noindent
COV2 Let $I_0\subset I$ be as in~\eqref{eq:lcoverings}.
Let $\sigma\cl I\to J$ be such that $U_i \subset V_{\sigma(i)}$, for all $i\in I$.
Then, for all $x\in U_i$ we have
$d(x,M\setminus  U_i) \leq d(x,M\setminus  V_{\sigma(i)})$.
It follows that $\sigma(I_0)$ satisfies~\eqref{eq:lcoverings} with respect
to $\{V_j\}_{j\in J}$.

\vspace{0.2ex}\noindent
COV3 Let $I_0\subset I$ be as in~\eqref{eq:lcoverings} and let $C$ be the
constant in~\eqref{eq:resit}.  Let $x$ be a given point in $V\cap U$.
We have $d(x,M\setminus (V\cap U)) \leq d(x,M\setminus U)$. We distinguish two
cases.
\\
(a) We assume that $d(x,M\setminus (V\cap U_i)) = d(x,M\setminus U_i)$, for all
$i\in I_0$. Then we clearly have
$d(x,M\setminus (V\cap U)) \leq C \max_{i\in I_0} d(x,M\setminus (V\cap U_i))$
and $I_0$ satisfies~\eqref{eq:lcoverings} with respect to $\{V\cap U_i\}_{i\in
  I}$.

\vspace{0.1ex}\noindent
(b) We assume $d(x,M\setminus  (V\cap U_{i_0})) < d(x,M\setminus  U_{i_0})$
for some $i_0\in I_0$.
We choose $y\in M\setminus (V\cap U_{i_0})$ such that
$d(x,y) =d(x,M\setminus (V\cap U_{i_0}))$. Then we have
$d(x,y) < d(x,M\setminus U_{i_0})$. We deduce that $y\in U_{i_0}$ and then that
$y\in M\setminus V$.  Hence $y \in M\setminus (V\cap U)$ and $d(x, M\setminus
(V\cap U)) \leq d(x,y)$. Then
\begin{align*}
d(x,M\setminus  (V\cap U)) &\leq d(x,M\setminus  (V\cap U_{i_0})) \\
&\leq \max_{i\in I_0}  d(x,M\setminus  (V\cap U_i)).
\end{align*}
We obtain~\eqref{eq:resit} for the family $\{V\cap U_i\}_{i\in I_0}$ with $C=1$.

\vspace{0.2ex}\noindent
COV4 Let $I_0\subset I$ be as in~\eqref{eq:lcoverings}  and let $C$ be the
constant in~\eqref{eq:resit}.  
For each $i\in I_0$ let $J_i \subset J$ satisfy~\eqref{eq:lcoverings} with
respect to $U_i$ for the family $\{U_{i} \cap V_j\}_{j\in J}$ and let $C_i$ be
the corresponding constant.  We set $J_0 = \bigcup_{i\in I_0} J_i$ and
$B = \max \{ C \cdot C_i$; $i\in I_0\}$. Then we have
\begin{align*}
d(x,M\setminus  U)  & \leq  C \max_{i\in I_0}  d(x,M\setminus  U_i)  \\
&\leq  C \max_{i\in I_0}   (C_i \max_{ j\in J_i}  d(x,M\setminus  (U_i\cap V_j))  )  \\
&\leq  B \max_{i\in I_0}   \; \max_{ j\in J_i}  d(x,M\setminus  V_j)   \\
& \leq B \max_{j\in J_0} d(x,M\setminus  V_j)  ,
\end{align*}
which proves that $J_0$ satisfies~\eqref{eq:lcoverings} with respect to
$\{V_j\}_{j\in J}$.
\end{proof}

As a particular case of COV4, we get 
\begin{corollary}
If $\{U_i\}_{i\in I}$ is a linear
covering of $U\in\Op_{\Msa}$ and $I=\bigsqcup_{\alpha\in A}I_\alpha$ is a
partition of $I$, then setting
$U_\alpha\eqdot\bigcup_{i\in I_\alpha}U_i$, $\{U_\alpha\}_{\alpha\in A}$ is a
linear covering of $U$.
\end{corollary}
The notion of a linear covering is of local nature (in the usual topology). More
precisely, we have:

\begin{proposition}\label{pro:covolV}
Let $V\in\Op_\Msa$ and let $\{U_i\}_{i\in I}$ be a finite covering of $\ol V$
in $\Msa$. Then $\{V\cap U_i\}_{i\in I}$ is a linear covering of $V$.
\end{proposition}
\begin{proof}
Set $U=\bigcup_iU_i$ and let $W\in\Op_\Msa$ be a neighborhood of the boundary
$\partial U$ such that $V\cap W=\emptyset$. Let us prove that the family
$\{W, \{U_i\}_{i\in I}\}$ is a linear covering of $W\cup U$.
We set $f(x) = \max\{ d(x,M\setminus W), d(x,M\setminus U_i), i\in I\}$ and
$Z = \{x\in M$; $d(x,M \setminus (W\cup U)) \geq d(x,U)\}$.  Then $Z$ is a
compact subset of $W\cup U$.  Hence there exists $\varepsilon>0$ such that
$f(x)>\varepsilon$ for all $x \in Z$.

We also see that $\ol U \subset Z$. Hence $f(x) = d(x,M\setminus W)$ for
$x\not\in Z$. Moreover, for a given $x\not\in Z$ we have
$d(x,M\setminus W) \leq d(x,M \setminus (W\cup U)) < d(x,U)$ by definition of
$Z$. Hence a given $y\in M\setminus W$ realizing $d(x,M\setminus W)$ can not
belong to $U$ and we obtain $d(x,M\setminus (W\cup U)) = d(x,M\setminus W)$.
Finally $d(x,M\setminus (W\cup U)) = f(x)$ for $x\not\in Z$.

Now we deduce that $d(x,M\setminus (W\cup U)) \leq C f(x)$ for some $C>0$ and
for all $x\in M$, that is, $\{W, \{U_i\}_{i\in I}\}$ is a linear covering of
$W\cup U$.

Taking the intersection with $V$ we obtain by COV3 that
$\{V \cap U_i\}_{i\in I}$ is a linear covering of $V$.
\end{proof}

\begin{corollary}\label{cor:1covlocal}
Let $\{U_i\}_{i\in I}$ and $\{B_j\}_{j\in J}$ be two finite families in
$\Op_{\Msa}$. We set $U=\bigcup_i U_i$ and we assume that
$\ol U \subset \bigcup_j B_j$.
Then $\{U_i\}_{i\in I}$ is a linear covering of $U$ if and only if
$\{U_i\cap B_j\}_{i\in I}$ is a linear covering of $U\cap B_j$ for all
$j\in J$.
\end{corollary}
\begin{proof}
(i) Assume that $\{U_i\}_i$ is a linear covering of $U$. Applying COV3 to
$B_j\cap U\subset U$ we get that the family $\{U_i\cap B_j\}_{i\in I}$ is a
linear covering of $U\cap B_j$ for all $j\in J$.

\vspace{0.3ex}\noindent
(ii) Assume that the family $\{U_i\cap B_j\}_{i\in I}$ is a linear covering of
$U\cap B_j$ for all $j\in J$.  By Proposition~\ref{pro:covolV} the family
$\{U \cap B_j\}_{j\in J}$ is a linear covering of $U$.
Hence the result follows from COV4.
\end{proof}

\begin{definition}\label{def:linsasite}
\banum
\item
The  linear subanalytic site  $\Msal$ is the presite $\Msa$ endowed with the
Grothendieck topology for which the coverings are the linear coverings given by
Definition~\ref{def:linearcov}.
\item
We denote by $\rhosal\cl \Msa\to\Msal$ and by $\rhosl\cl M\to\Msal$ the natural morphisms of sites. 
\enum
\end{definition}
The morphisms of sites constructed above are summarized by the diagram
\eqn
&&\xymatrix{
M\ar[r]^-{\rhosa}\ar[rd]_-{\rhosl}&\Msa\ar[d]^-{\rhosal}\\
&\Msal.
}\eneqn
\begin{remark}\label{rem:bilip_homeo} 
  Let $M$ and $N$ be two real analytic manifolds and let $f\cl M \to N$ be a
  topological isomorphism such that both $f$ and $\opb{f}$ are subanalytic
  Lipschitz maps.  Then $\opb{f}\cl \Op_\Msa \to \Op_{N_{\rm sa}}$ induces an
  isomorphism of sites $N_{\rm sal} \isoto \Msal$.
\end{remark}

\section{Regular coverings}

We shall also use the following:
\begin{definition}\label{def:regularcov}
Let $U\in\Op_\Msa$. A regular covering of $U$ is a sequence
$\{U_i\}_{i\in [1,N]}$ with $1\leq N\in\N$ such that
$U = \bigcup_{i\in [1,N]} U_i$ and, for all $1\leq k\leq N$,
$\{U_i\}_{i\in [1,k]}$ is a linear covering of $\bigcup_{1\leq i\leq k}U_i$.
\end{definition}

We will use the following recipe to turn an arbitrary covering into a linear
covering by a slight enlargement of the open subsets.  For an open subset $U$
of $M$, an arbitrary subset $V \subset U$ and $\varepsilon>0$ we set 
\eq\label{eq:enlarge_to_covA}
V^{\varepsilon,U} = \{ x\in M; \; d(x,V) < \varepsilon\, d(x,M\setminus U) \}.
\eneq
Then $V^{\varepsilon,U}$ is an open subset of $U$.
If the distance $d$ is a subanalytic function on $M\times M$,
$U \in\Op_{\Msa}$ and $V$ is a subanalytic subset, then $V^{\varepsilon,U}$
also belongs to $\Op_{\Msa}$.
We see easily that $(U\cap \ol V) \subset V^{\varepsilon,U} \subset U$.

\begin{lemma}\label{lem:enlarge_to_cov}
We assume that the distance $d$ is a subanalytic function on $M\times M$. 
Let $U\in \Op_{\Msa}$ and let $V \subset U$ be a subanalytic subset.
Let $0<\varepsilon$ and $0<\delta<1$.
We set $\varepsilon'= \frac{\varepsilon + \delta}{1-\delta}$. Then
\bnum
\item for any $x\in V^{\varepsilon,U}$ and $y\in M$ such that
 $$
d(x,y) < \delta\, d(x,M\setminus U) \text{ or }
d(x,y) < \delta\, d(y,M\setminus U),
$$
 we have $d(y,V) < \varepsilon' d(y,M\setminus U)$, that is, $y\in V^{\varepsilon',U}$,
\item for any $x\in V^{\varepsilon,U}$ we have $d(x,M\setminus V^{\varepsilon',U})
 \geq \delta\, d(x,M\setminus U)$,
\item $\{U\setminus \ol V,V^{\varepsilon',U}\}$ is a linear  covering of $U$.
\enum
\end{lemma}
We remark that any $\varepsilon'>0$ can be written
$\varepsilon'= \frac{\varepsilon + \delta}{1-\delta}$ with
$\varepsilon, \delta$ as in the lemma.
\begin{proof}
(i) The triangular inequality
$d(x,M\setminus U) \leq d(x,y) + d(y,M\setminus U)$ implies
$$
\begin{cases}
  d(x,M\setminus U) < (1-\delta)^{-1} d(y,M\setminus U), & \text{if }
d(x,y) < \delta \, d(x,M\setminus U), \\
d(x,M\setminus U) < (1+\delta) d(y,M\setminus U), & \text{if }
d(x,y) < \delta \, d(y,M\setminus U).
\end{cases}
$$
Since $1+\delta < (1-\delta)^{-1}$ we obtain in both cases
\eq\label{eq:points_proches}
d(x,M\setminus U) < (1-\delta)^{-1} d(y,M\setminus U).
\eneq
In particular we have in both cases $d(x,y) < \delta (1-\delta)^{-1} d(y,M\setminus U)$.
Now the definition of $V^{\varepsilon,U}$ implies
\begin{align*}
d(y,V) &\leq d(x,y) + d(x,V) \\
&< \delta (1-\delta)^{-1}  d(y,M\setminus U) +  \varepsilon\, d(x,M\setminus U) \\
& < (\varepsilon + \delta)(1-\delta)^{-1} d(y,M\setminus U),
\end{align*}
where the last inequality follows from~\eqref{eq:points_proches}.

\medskip\noindent
(ii) By (i), if a point $y\in M$ does not belong to $ V^{\varepsilon',U}$,
we have $d(x,y) \geq \delta\, d(x,M\setminus U)$.
This gives~(ii).

\medskip\noindent
(iii) Since $d$ is subanalytic, the open subset $ V^{\varepsilon',U}$ is 
subanalytic. We also see easily that
$U = (U\setminus \ol V) \cup V^{\varepsilon',U}$.
Now let $x\in M$.

\vspace{0.2ex}\noindent
(a) If $x\not\in V^{\varepsilon,U}$, then~\eqref{eq:enlarge_to_covA} gives
$d(x,V) \geq \varepsilon\, d(x,M\setminus U)$.
Since $d(x, M \setminus (U\setminus \ol V)) = \min \{ d(x,M\setminus U),
d(x,V)\}$, we deduce 
$d(x, M \setminus (U\setminus \ol V))
 \geq \min\{\varepsilon,1\} d(x,M\setminus U)$.
 
\vspace{0.2ex}\noindent
(b) If $x\in V^{\varepsilon,U}$, then~(ii) gives $d(x,M\setminus V^{\varepsilon',U})
 \geq \delta\, d(x,M\setminus U)$.
 
\vspace{0.2ex}\noindent
We obtain in both cases
$$
\max \{ d(x, M \setminus (U\setminus \ol V)) , 
d(x,M\setminus V^{\varepsilon',U}) \}
\geq C d(x,M\setminus U),
$$
where $C = \min \{ \delta, \varepsilon \}$.
This proves~(iii).
\end{proof}

Lemma~\ref{lem:prelim_Horm2} below will be used later to obtain subsets
satisfying the hypothesis of Lemma~\ref{le:Ho0}. We will prove it 
by using Lemma~\ref{lem:enlarge_to_cov} as follows.
Let $U_1,U_2\in\Op_\Msa$ and let $U=U_1\cup U_2$.  For $\varepsilon>0$ we set,
using Notation~\eqref{eq:enlarge_to_covA},
\eq
\label{eq:defUeps1}
&& U^\varepsilon_1 = (U_1 \setminus U_2)^{\varepsilon,U_1} 
 = \{x\in U_1;\; d(x,U_1\setminus U_2) 
 < \varepsilon\, d(x,M\setminus U_1) \},  \\
\label{eq:defUeps2}
&& U^\varepsilon_2 = (U_2 \setminus U_1)^{\varepsilon,U_2}
 = \{x\in U_2;\; d(x,U_2\setminus U_1) 
 < \varepsilon\, d(x,M\setminus U_2) \} .
\eneq

\begin{lemma}\label{lem:prelim_Horm}
\bnum
\item For $i=1,2$ and for any $\varepsilon>0$, the pair
$\{U^\varepsilon_i, U_1\cap U_2\}$ is a linear covering of $U_i$.
\item For any $\varepsilon, \varepsilon'>0$ such that
 $\varepsilon\varepsilon' < 1$, we have
$\ol{U^\varepsilon_1} \cap \ol{U^{\varepsilon'}_2} \cap U =\emptyset$.
\item Let $\varepsilon >0$, $0<\delta<1$ and set
$\varepsilon' = \frac{\varepsilon+\delta}{1-\delta}$,
$\varepsilon'' = \frac{\varepsilon'+\delta}{1-\delta}$. We assume
$\varepsilon \varepsilon'' < 1$. Then, for any $x\in M$,
$$
\begin{cases}
  d(x, U^\varepsilon_2) \geq \delta \, d(x,M\setminus U_1) 
& \text{if } x\in U^{\varepsilon'}_1,\\
d(x, U^{\varepsilon}_1) \geq \delta \, d(x,M\setminus U_1) 
 & \text{if } x\not\in U^{\varepsilon'}_1 .
\end{cases}
$$
\enum
\end{lemma}
\begin{proof}
(i) By symmetry we can assume $i=1$. By Lemma~\ref{lem:enlarge_to_cov},
the pair $\{U_1\setminus \ol{(U_1 \setminus U_2)}, U^\varepsilon_1\}$ is a linear
covering of $U_1$. Since $U_2$ is open we have
$U_1\setminus \ol{(U_1 \setminus U_2)} = U_1\cap U_2$ and~(i) follows.

\medskip\noindent
(ii) We have
\eqn
&&\ol{U^\varepsilon_1} \cap U \subset\{x\in U;\; d(x,U_1\setminus U_2) \leq \varepsilon\, d(x,M\setminus U_1) \},\\
&&\ol{U^{\varepsilon'}_2} \cap U \subset\{x\in U;\; d(x,U_2\setminus U_1) \leq \varepsilon' \, d(x,M\setminus U_2) \}.
\eneqn
We remark that $d(x,M\setminus U_2) \leq d(x,U_1\setminus U_2)$ and
$d(x,M\setminus U_1) \leq d(x,U_2\setminus U_1)$ for any $x\in M$.
Let $x\in \ol{U^\varepsilon_1} \cap \ol{U^{\varepsilon'}_2}\cap U$ and set
$d_1= d(x,U_2\setminus U_1)$, $d_2 = d(x,U_1\setminus U_2)$.
We deduce $d_i \leq \varepsilon\varepsilon' d_i$, for $i=1,2$.
Since $\varepsilon\varepsilon' <1$ we obtain $d_1=d_2=0$.
Hence $x\not\in U_1$ and $x\not\in U_2$.
Since $U=U_1\cup U_2$, this proves~(ii).

\medskip\noindent
(iii) By Lemma~\ref{lem:enlarge_to_cov}~(ii), we have
$d(x, M\setminus U^{\varepsilon''}_1) \geq \delta \, d(x,M\setminus U_1)$
for any $x\in U^{\varepsilon'}_1$.
By~(ii) we have $U^{\varepsilon}_2 \subset M\setminus U^{\varepsilon''}_1$
and the first inequality follows.

By Lemma~\ref{lem:enlarge_to_cov}~(i), if $x\not\in U^{\varepsilon'}_1$ and
$z\in U^{\varepsilon}_1$, then $d(x,z) \geq \delta\, d(x, M\setminus U_1)$.
This gives the second inequality.
\end{proof}

\begin{lemma}\label{lem:prelim_Horm2}
Let $U_1, U_2 \in \Op_{\Msa}$ and set $U=U_1 \cup U_2$.
We assume that $\{U_1,U_2\}$ is a linear covering of $U$.
Then there exist $U'_i \subset U_i$, $i=1,2$, and $C>0$ such that
\bnum
\item $\{U'_i, U_1\cap U_2\}$ is a linear covering of $U_i$ \lp$i=1,2$\rp,
\item $\ol{U'_1} \cap \ol{U'_2} \cap U = \emptyset$,
\item setting $Z_i = (M\setminus U) \cup \ol{U'_i}$, we have
$Z_1\cap Z_2 = M\setminus U$ and
$$
d(x, Z_1\cap Z_2) \leq C(d(x,Z_1)+d(x,Z_2)),
\quad\text{ for any $x\in M$.}
$$
\enum
\end{lemma}
\begin{proof}
We set $\varepsilon = \delta = 1/3$, 
$\varepsilon' = \frac{\varepsilon+\delta}{1-\delta} = 1$ and
$\varepsilon'' = \frac{\varepsilon'+\delta}{1-\delta} = 2$.
Using the notations~\eqref{eq:defUeps1} and~\eqref{eq:defUeps2}
we set $U'_i = U^\varepsilon_i$, $i=1,2$.

\medskip\noindent
(i) and (ii) are given by Lemma~\ref{lem:prelim_Horm}~(i) and (ii).

\medskip\noindent
(iii) The equality $Z_1\cap Z_2 = M\setminus U$ follows from~(ii).
Let $C'$ be the constant in~\eqref{eq:resit} for the family $\{U_1,U_2\}$.
We set $C_1 = \max\{1,\delta^{-1} C'\}$.
Let $x\in M$ and let $x_i \in Z_i$ be such that $d(x,x_i) = d(x,Z_i)$.
By the definition of $Z_1$, if $x_1 \not\in \ol{U'_1}$, then
$x_1 \in M\setminus U$. Hence $d(x,Z_1)= d(x,M\setminus U)$ and the
inequality in~(iii) is clear.

Hence we can assume $x_1 \in  \ol{U'_1}$ and also $x_2 \in  \ol{U'_2}$ by
symmetry. Then we have
$d(x,Z_1)+d(x,Z_2) = d(x, U^\varepsilon_1) + d(x, U^{\varepsilon}_2)$.
Since $\varepsilon \varepsilon'' = 2/3 < 1$, Lemma~\ref{lem:prelim_Horm}~(iii)
gives
$d(x, U^\varepsilon_1) + d(x, U^{\varepsilon}_2) \geq \delta \, d(x,M\setminus U_1)$.
The same holds with $M\setminus U_1$ replaced by $M\setminus U_2$
and~\eqref{eq:resit} gives
$$
d(x, U^\varepsilon_1) + d(x, U^{\varepsilon}_2)
 \geq \delta \, \max_{i=1,2} \{ d(x,M\setminus U_i) \}
\geq  C_1^{-1} d(x,M\setminus U),
$$
so that~(iii) holds with $C=C_1$.
\end{proof}

\begin{lemma}\label{lem:shrink_to_cov}
We assume that the distance $d$ is a subanalytic function on $M\times M$.
Let $\{U_i\}_{i= 1}^N$ be a $1$-regularly situated family in $\Op_{\Msa}$
and let $C\geq 1$ be a constant satisfying~\eqref{eq:resit}.
We choose $D>C$ and $1>\varepsilon>0$ such that $\varepsilon D < 1-\varepsilon$.
We define $U^0_i, V_i, U'_i \in \Op_{\Msa}$ inductively
on $i$ by $U^0_1 = V_1 = U'_1 =  U_1$ and
\begin{align*}
U^0_i & = \{ x\in U_i ; \; d(x, M \setminus (U_i \cup V_{i-1})) 
< D \, d(x,M \setminus U_i)  \} , \\
V_i &= V_{i-1} \cup U^0_i,  \\
U'_i &= (U^0_i)^{\varepsilon,V_i}  \quad
\text{\lp using the notation~\eqref{eq:enlarge_to_covA}\rp.}
\end{align*}
Then $V_N = \bigcup_{i= 1}^N U_i$ and, for all $k = 1,\ldots,N$, we have
$U'_k \subset U_k$, $V_k = \bigcup_{i= 1}^k U'_i$ and
$\{U'_i\}_{i= 1}^k$ is a $1$-regularly situated family in $\Op_{\Msa}$.
\end{lemma}
\begin{proof}
(i) Let us prove that $U'_k \subset U_k$.
Let $x\in U'_k$ and let us show that $x\in U_k$.
By~\eqref{eq:enlarge_to_covA} we have
$x\in V_k$ and there exists $y\in U^0_k$
such that $d(x,y) < \varepsilon\, d(x,M\setminus V_k)$.
We deduce $d(x,y) < \varepsilon (d(x,y) + d(y,M\setminus V_k))$ and then
\eq\label{eq:shrink_to_cov1}
&& d(x,y) < (\varepsilon/(1-\varepsilon)) \, d(y,M\setminus V_k).
\eneq
On the other hand we have $U^0_k \subset U_k$, hence
$V_k \subset U_k \cup V_{k-1}$. Since $y\in U^0_k$ we deduce
\eq\label{eq:shrink_to_cov2}
&& d(y, M \setminus V_k ) \leq d(y, M \setminus (U_k \cup V_{k-1})) 
< D \, d(y,M \setminus U_k) .
\eneq
The inequalities~\eqref{eq:shrink_to_cov1}, \eqref{eq:shrink_to_cov2} and the
hypothesis on $D$ and $\varepsilon$ give $d(x,y) < d(y,M\setminus U_k)$. Hence
$x\in U_k$.

\medskip\noindent
(ii) We have $V_i = V_{i-1} \cup U^0_i$. Hence Lemma~\ref{lem:enlarge_to_cov}
implies that $\{V_{i-1}, U'_i\}$ is a covering of $V_i$ in $\Msa$. 
 Let us prove the last part of the lemma by induction on $k$. We immediately obtain that
$V_k = \bigcup_{i= 1}^k U'_i$. 
Moreover,  $\{V_{k-1}, U'_k\}$ being a covering of $V_k$, 
we get by using COV4 that, for all $k = 1,\ldots,N$,
$\{U'_i\}_{i= 1}^k$ is a $1$-regularly situated family in $\Op_{\Msa}$.

\medskip\noindent
(iii) It remains to prove that $V_N = \bigcup_{i= 1}^N U_i$.
It is clear that $V_k \subset \bigcup_{i= 1}^N U_i$, for all $k = 1,\ldots,N$.
Let $x\in \bigcup_{i= 1}^N U_i$. Since $\{U_i\}_{i= 1}^N$ is $1$-regularly
situated, there exists $i_0$ such that
$d(x,M\setminus \bigcup_{i= 1}^N U_i)\leq C \, d(x,M \setminus U_{i_0})$.
In particular $x \in U_{i_0}$ and moreover
$d(x, M \setminus (U_{i_0} \cup V_{i_0-1}))
\leq C \, d(x,M \setminus U_{i_0})< D \, d(x,M \setminus U_{i_0})$.
Therefore $x\in U^0_{i_0}$. By definition $U^0_{i_0} \subset V_{i_0} \subset V_N$.
Hence $x\in V_N$ and we obtain $V_N = \bigcup_{i= 1}^N U_i$.
\end{proof}

In particular, we have proved:

\begin{proposition}\label{pro:exregcov}
Let $U\in\Op_\Msa$. Then for any linear covering $\{U_i\}_{i\in I}$ of $U$  
there exists a refinement which is  a regular covering of $U$.
\end{proposition}

 \chapter{Sheaves on subanalytic topologies}\label{ShvSa2}

\section{Sheaves}

\subsubsection*{Usual notations}
We shall mainly follow the notations of~\cite{KS90,KS01} and~\cite{KS06}.

In this paper, we denote by $\cor$  a commutative unital {\em Noetherian} ring with finite global dimension.
Unless otherwise specified, a manifold means a real analytic manifold.

If $\shc$ is an additive category, we denote by $\RC(\shc)$ the additive category of complexes in $\shc$.
For $*=+,-,\rb$ we also consider the full additive subcategory $\RC^*(\shc)$
of $\RC(\shc)$ consisting of complexes bounded from below 
(resp. from above, resp. bounded) and $\RC^\ub(\shc)$ means $\RC(\shc)$ (``$\ub$'' stands for ``unbounded''). 
If  $\shc$ is an abelian category, we denote by $\RD(\shc)$  its  
derived category and similarly with $\RD^*(\shc)$ for  $*=+,-,\rb,\ub$.

For a site $\sht$, we denote by $\Psh(\cor_\sht)$ and $\md[\cor_\sht]$ the abelian categories of presheaves and sheaves of $\cor$-modules on $\sht$. We denote  by $\iota\cl\md[\cor_\sht]\to\Psh(\cor_\sht)$ the forgetful functor
and by $(\scbul)^a$ its left adjoint, the functor which associates a sheaf to a presheaf. Note that in practice we shall often not write 
 $\iota$.
Recall that $\md[\cor_\sht]$  is a Grothendieck category and, in particular, has enough injectives.
We write $\RD^*(\cor_\sht)$ instead of $\RD^*(\md[\cor_\sht])$ ($*=+,-,\rb,\ub$).

For a site $\sht$, we will often use the following well-known fact. For any $F\in \RD(\cor_\sht)$
and any $i\in \Z$, the cohomology sheaf $H^i(F)$ is the sheaf associated with
the presheaf $U\mapsto H^i(U;F)$. In particular, if $H^i(U;F) = 0$ for all
$U\in \sht$, then $H^i(F) \simeq  0$.

For an object $U$ of $\sht$, recall that there is a sheaf naturally attached to
$U$ (see {\em e.g.}~\cite[\S~17.6]{KS06}). We shall denote it here by
$\cor_{U\sht}$ or simply $\cor_U$ if there is no risk of confusion. This is the
sheaf associated with the presheaf (see loc.\ cit.\ Lemma~17.6.11):
\eqn
&&V\mapsto\oplus_{V\to U}\cor.
\eneqn
The functor ``associated sheaf'' is exact. If follows that, if $V\to U$ is a
monomorphism in $\sht$, then the natural morphism $\cor_{V\sht}\to\cor_{U\sht}$
also is a monomorphism.

\subsubsection*{Sheaves on $M$ and $\Msa$}

We shall mainly use the subanalytic topology introduced in~\cite{KS01}. In loc.\ cit.,
sheaves on the subanalytic topology  are studied in the more general framework of indsheaves.
We  refer to~\cite{Pr08}  for a direct and more elementary treatment of subanalytic sheaves.

Recall that  $\rhosa\cl M\to\Msa$ denotes the natural morphism of sites. 
The functor $\oim{\rhosa}$ is left exact and  its  left adjoint $\opb{\rhosa}$ is exact. Hence, we 
have the pairs of adjoint functors
\eq\label{eq:fctrho0}
\xymatrix{
\md[\cor_M]\ar@<.5ex>[r]^{\oim{\rhosa}}&\md[\cor_\Msa]\ar@<.5ex>[l]^{\opb{\rhosa}},
}
\quad
\xymatrix{
\Derb(\cor_M)\ar@<.5ex>[r]^{\roim{\rhosa}}&\Derb(\cor_\Msa).\ar@<.5ex>[l]^{\opb{\rhosa}}
}
\eneq
The functor $\oim{\rhosa}$  is fully faithful and
$\opb{\rhosa}\oim{\rhosa}\simeq\id$. Moreover,
$\opb{\rhosa}\roim{\rhosa}\simeq\id$  and $\roim{\rhosa}$ in~\eqref{eq:fctrho0} is fully faithful.

The functor $\opb{\rhosa}$ also admits a left adjoint functor
$\eim{\rhosa}$. For $F\in\md[\cor_M]$, $\eim{\rhosa}F$ is the sheaf on $\Msa$
associated with the presheaf $U\mapsto F(\ol U)$.  The functor $\eim{\rhosa}$
is exact, fully faithful and commutes with tensor products.

\begin{proposition}\label{pro:sectrhoUF00}
Let  $U\in\Op_\Msa$ and let $F\in\md[\cor_M]$. Then
\eqn
&&\rsect(U;\roim{\rhosa}F)\simeq\rsect(U;F).
\eneqn
\end{proposition}
\begin{proof}
This follows from $\rsect(U;G)\simeq \RHom(\cor_U,G)$ for $G\in\md[\cor_\sht]$
($\sht=M$ or $\sht=\Msa)$ and by adjunction since
$\opb{\rhosa}\cor_{U\Msa}\simeq\cor_{UM}$.
\end{proof}
Also note that the functor $\oim{\rhosa}$ admitting an exact left adjoint functor, it sends injective objects of 
$\md[\cor_M]$ to injective objects of $\md[\cor_\Msa]$.

One denotes by $\mdrc[\cor_M]$ the category
of $\R$-constructible sheaves on $M$. One denotes by 
$\Derb_\Rc(\cor_M)$ the full triangulated subcategory of $\Derb(\cor_M)$
consisting of objects with $\R$-constructible cohomologies.

Recall that  $\oim{\rhosa}$ is  exact when restricted to the subcategory $\mdrc[\cor_M]$. 
Hence we shall consider this last category both as a full subcategory of $\md[\cor_M]$ and a full subcategory of $\md[\cor_\Msa]$.

For $U\in\Op_\Msa$  we have the sheaf $\cor_{U\Msa}\simeq\oim{\rhosa}\cor_{UM}$ on $\Msa$ 
that we simply denote by $\cor_U$. 

\subsubsection*{Sheaves on $M$ and $\Msal$}

Recall Definition~\ref{def:linsasite}.
 The functor $\oim{\rhosal}$ is left exact and  its  left adjoint $\opb{\rhosal}$ is exact 
since the presites underlying the sites $\Msa$ and $\Msal$ are the same (see~\cite[Th.~17.5.2]{KS06}). Hence, we 
have the pairs of adjoint functors
\eq\label{eq:fctrho00}
\xymatrix{
\md[\cor_\Msa]\ar@<.5ex>[r]^{\oim{\rhosal}}&\md[\cor_\Msal]\ar@<.5ex>[l]^{\opb{\rhosal}},
}
\quad
\xymatrix{
\RD^+(\cor_\Msa)\ar@<.5ex>[r]^{\roim{\rhosal}}&\RD^+(\cor_\Msal). \ar@<.5ex>[l]^{\opb{\rhosal}}
}
\eneq
\begin{lemma}\label{le:rhoff}
The functor $\oim{\rhosal}$ in~\eqref{eq:fctrho00} is fully faithful and
$\opb{\rhosal}\oim{\rhosal}\simeq\id$. Moreover,
$\opb{\rhosal}\roim{\rhosal}\simeq\id$  and $\roim{\rhosal}$ in~\eqref{eq:fctrho00} is fully faithful.
\end{lemma}
\begin{proof}
(i) By its definition, $\opb{\rhosal}\oim{\rhosal}F$ is the sheaf associated
with the presheaf
$U\mapsto (\oim{\rhosal}F)(U)\simeq F(U)$ and this presheaf is already a sheaf.

\vspace{2ex}\noindent
(ii) Since $\opb{\rhosal}$ is exact, $\opb{\rhosal}\roim{\rhosal}$ is the derived
functor of $\opb{\rhosal}\oim{\rhosal}$.
\end{proof}

In the sequel, if $K$ is a compact subset of $M$, we set for a sheaf $F$
on $\Msa$ or $\Msal$: 
\eqn
&&\sect(K;F)\eqdot\indlim[K\subset U]\sect(U;F),\quad U\in\Op_\Msa.
\eneqn

\begin{lemma}\label{le:secttK}
Let $F\in\md[\cor_\Msal]$. For $K$ compact in $M$,  we have the natural isomorphisms
\eqn
&&\sect(K;F)\isoto \sect(K;\opb{\rhosal}F)\isoto \sect(K;\opb{\rhosl}F).
\eneqn
\end{lemma}
\begin{proof}
The first isomorphism follows from Proposition~\ref{pro:covolV}.
The second one from~\cite[Prop.~6.6.2]{KS01} since $\opb{\rhosl}\simeq\opb{\rhosa}\opb{\rhosal}$.
\end{proof}
The next result is analogue to~\cite[Prop.~6.6.2]{KS01}. 

\begin{proposition}\label{pro:662}
Let $F\in\md[\cor_\Msal]$. For $U$ open in $M$,  we have the natural isomorphism
\eqn
&&\sect(U;\opb{\rhosl}F)\simeq\prolim[V\subset\subset U] \sect(V;F),\, V\in\Op_\Msa.
\eneqn
\end{proposition}
\begin{proof}
We have the chain of isomorphisms, the second one following from 
Lemma~\ref{le:secttK}: 
\eqn
&\sect(U;\opb{\rhosl}F)
 \simeq \prolim[V\subset\subset U] \sect(\ol V;\opb{\rhosl}F)
\simeq \prolim[V\subset\subset U] \sect(\ol V;F)
\simeq \prolim[V\subset\subset U] \sect(V;F).
\eneqn
\end{proof}
The next result is analogue to~\cite[Prop.~6.6.3,~6.6.4]{KS01}. Since the proof of loc.\ cit.\ 
extends to our situation with the help of Proposition~\ref{pro:662}, we do not repeat it.

\begin{proposition}\label{pro:663}
The functor $\opb{\rhosl}$ admits a left adjoint that we denote by $\eim{\rhosl}$.
For $F\in\md[\cor_M]$,   $\eim{\rhosl}F$ is the sheaf on $\Msal$ associated with the presheaf 
$U\mapsto F(\ol U)$. The functor $\eim{\rhosl}$ is
exact and fully faithful. 
\end{proposition}

\subsubsection*{Sheaves on $\Msa$ and $\Msal$}

\begin{proposition}\label{pro:corUsasal}
Let $U\in\Op_\Msa$. Then we have $\oim{\rhosal}\cor_{U\Msa}\simeq\cor_{U\Msal}$ 
and $\opb{\rhosal}\cor_{U\Msal}\simeq\cor_{U\Msa}$.
\end{proposition}
\begin{proof}
The proof of~\cite[Prop.~6.3.1]{KS01} gives the first isomorphism  without any
changes other than notational. 
The second isomorphism follows by Lemma~\ref{le:rhoff}.
\end{proof}

\begin{proposition}\label{pro:sectrhoUF}
Let  $U\in\Op_\Msa$ and let $F\in\md[\cor_\Msa]$. Then
\eqn
&&\rsect(U;\roim{\rhosal}F)\simeq\rsect(U;F).
\eneqn
\end{proposition}
The proof goes as for Proposition~\ref{pro:sectrhoUF00}.

In the sequel we shall simply denote by $\cor_U$ the sheaf $\cor_{U\sht}$ for $\sht=\Msa$ or $\sht=\Msal$. 

\begin{proposition}\label{pro:MV0}
Let $\sht$ be either the site $\Msa$ or the site $\Msal$. Then a presheaf $F$
is a sheaf if and only if it satisfies:
\bnum
\item
$F(\emptyset)=0$,
\item
for any $U_1,U_2\in\Op_\Msa$ such that $\{U_1,U_2\}$ is a covering of
$U_1\cup U_2$,  the sequence
$0\to F(U_1\cup U_2)\to F(U_1)\oplus F(U_2)\to F(U_1\cap U_2)$ is exact.
\enum
\end{proposition}
Of course, if $\sht=\Msa$, $\{U_1,U_2\}$ is always a covering of $U_1\cup U_2$.
\begin{proof}
In the case of the site $\Msa$ this is Proposition~6.4.1 of~\cite{KS01}.
Let $F$ be a presheaf on $\Msal$ such that (i) and (ii) are satisfied and let
us prove that $F$ is a sheaf. Let $U\in \Op_\Msa$ and let
$\{U_i\}_{i\in I}$ be a linear covering of $U$.
By Proposition~\ref{pro:exregcov} we can find a finite refinement
$\{V_j\}_{j\in J}$ of $\{U_i\}_{i\in I}$ which is a regular covering of $U$.
We choose $\sigma\cl J\to I$ such that $V_j \subset U_{\sigma(j)}$
for all $j\in J$ and we consider the commutative diagram
\begin{equation}\label{eq:diag_raf_recouv}
\vcenter{\xymatrix{
0 \ar[r] & F(U) \ar[r]^-u\ar@{=}[d] & \bigoplus_{i\in I} F(U_i) \ar[r]^-v\ar[d]^a
 & \bigoplus_{i,j \in I} F(U_{ij}) \ar[d]^b  \\
0 \ar[r] & F(U) \ar[r]
& \bigoplus_{k\in J} F(V_k) \ar[r] & \bigoplus_{k,l \in J} F(V_{kl}), }}
\end{equation}
where $a$ and $b$ are defined as follows.
For $s = \{s_i\}_{i\in I} \in \bigoplus_{i\in I} F(U_i)$, we set
$a(s) = \{t_k\}_{k\in J} \in \bigoplus_{k\in J} F(V_k)$ where
$t_k = s_{\sigma(k)}|_{V_k}$. In the same way we set
$b( \{s_{ij}\}_{i,j \in I}) = \{s_{\sigma(k) \sigma(l)}|_{V_{kl}}\}_{k,l \in J}$.
The proof of~\cite[Prop.~6.4.1]{KS01} applies to a regular covering in $\Msal$
and we deduce that the bottom row of the diagram~\eqref{eq:diag_raf_recouv} is
exact.
It follows immediately that $\ker u = 0$. This proves that $F$ is a separated
presheaf.

It remains to prove that $\ker v = \im u$. Let 
$s = \{s_i\}_{i\in I} \in \bigoplus_{i\in I} F(U_i)$ be such that $v(s) = 0$.
By the exactness of the bottom row we can find $t\in F(U)$ such that
$a(u(t) - s) =0$.
Let us check that $t|_{U_i} = s_i$ for any given $i\in I$.  The
family $\{U_i\cap V_k\}_{k\in J}$ is a covering of $U_i$ in $\Msal$.
Since $F$ is separated it is enough to see that
$t|_{U_i \cap V_k} = s_i|_{U_i\cap V_k}$ for all $k\in J$. Setting $W = U_i \cap V_k$,
we have
$$
t|_W = s_{\sigma(k)}|_W = (s_{\sigma(k)}|_{U_i \cap U_{\sigma(k)}})|_W
= (s_i|_{U_i \cap U_{\sigma(k)}})|_W = s_i|_W,
$$
where the first equality follows from $a(u(t) - s) =0$ and the third one
from $v(s) = 0$.
\end{proof} 

\begin{lemma}\label{le:sectUindlim}
Let $\sht$ be either the site $\Msa$ or the site $\Msal$. 
Let $U\in\Op_\Msa$ and let $\{F_i\}_{i\in I}$ be an inductive system in
$\md[\cor_\sht]$ indexed by a small filtrant category $I$. Then
\eq\label{eq:rsectoplus}
&&  \indlim[i] \sect(U;F_i)\isoto\sect(U;\indlim[i]F_i).
\eneq
\end{lemma}
This kind of results is well-known from the specialists (see {\em e.g.}~\cite{KS01,EP10}) 
but for the reader's convenience, we give a proof.

\begin{proof}
For a covering $\shs=\{U_j\}_j$ of $U$ set 
\eqn
&&\sect(\shs;F)\eqdot\ker\bl \prod_iF(U_i)\tto\prod_{ij}F(U_i\cap U_j)\br.
\eneqn
Denote by $\inddlim$ the inductive limit in the category of presheaves and
recall that $\indlim[i]F_i$ is the sheaf associated with $\inddlim[i]F_i$. The
presheaf $\inddlim[i]F_i$ is separated.  Denote by $\Cov(U)$ the family of
coverings of $U$ in $\sht$ ordered as follows. For $\shs_1$ and $\shs_2$ in
$\Cov(U)$, $\shs_1\preceq\shs_2$ if $\shs_1$ is a refinement of $\shs_2$. Then
$\Cov(U)^\rop$ is filtrant and
\eqn
\sect(U;\indlim[i]F_i)
&\simeq&\indlim[\shs\in \Cov(U)]\sect(\shs;\inddlim[i]F_i)\\
&\simeq&\indlim[\shs]\indlim[i]\sect(\shs;F_i)\\
&\simeq&\indlim[i]\indlim[\shs]\sect(\shs;F_i)\simeq \indlim[i]\sect(U;F_i).
\eneqn
Here, the second isomorphism follows from the fact that we may assume that the
covering $\shs$ is finite.
\end{proof}
\begin{example}\label{exa:germcusp}
Let $M=\R^2$ endowed with coordinates $x=(x_1,x_2)$. For $\varepsilon, A>0$ we
define the subanalytic open subset
\eq\label{eq:defUAeps}
& U_{A,\varepsilon} = \{ x; \;  0 < x_1 < \varepsilon, \,
-A x_1^2 < x_2 < A x_1^2\}.
\eneq
We define a presheaf $F$ on $\Msal$ by setting, for any $V\in \Op_\Msa$,
$$
F(V) =\begin{cases}
\cor & \text{if for any $A>0$, there exists $\varepsilon>0$ such that 
   $U_{A,\varepsilon}  \subset V$,} \\
0 & \text{otherwise.}
\end{cases}
$$
The restriction map $F(V) \to F(V')$, for $V'\subset V$, is $\id_\cor$ if
$F(V') = \cor$. We prove that $F$ is sheaf in~(iii) below after the
preliminary remarks~(i) and~(ii).

\medskip\noindent
(i) For given $A, \varepsilon_0>0$ we have
$d( (\varepsilon,0), M\setminus U_{A,\varepsilon_0})
\geq (A/4) \varepsilon^2 $, for any $\varepsilon > 0$ small enough.
In particular, if $F(V) = \cor$, then
\begin{equation}\label{eq:germcusp1}
d((\varepsilon,0), M\setminus V) / \varepsilon^2 \to +\infty
\quad \text{when $\varepsilon\to 0$.}
\end{equation}

\noindent
(ii) Let us assume that there exist $A>0$ and a sequence $\{\varepsilon_n\}$,
$n\in\N$, such that $\varepsilon_n > 0$, $\varepsilon_n\to 0$ when
$n\to \infty$ and $V$ contains the closed balls
$B((\varepsilon_n,0), A \varepsilon_n^2)$ for all $n\in \N$.  Then there exists
$\varepsilon>0$ such that $V$ contains $\ol{U_{A,\varepsilon}} \setminus \{0\}$.

Before we prove this claim we translate the conclusion in terms of sheaf
theory (in the usual site $\R^2$). Let $p\cl \R^2 \to \R$ be the projection
$(x_1,x_2) \mapsto x_1$.  Then, for $x_1>0$, the set
$\opb{p}(x_1) \cap V \cap \ol{U_{A,\varepsilon}}$ is a finite disjoint union of
intervals, say $I_1,\ldots,I_N$.
If $\opb{p}(x_1) \cap V$ contains $\opb{p}(x_1) \cap \ol{U_{A,\varepsilon}}$,
then $N=1$, $I_1$ is closed and $\rsect(\R;\cor_{I_1}) = \cor$.  In the other
case none of these $I_1,\ldots,I_N$ is closed and $H^0(\R; \cor_{I_j})=0$, for
all $j=1,\ldots,N$.
By the base change formula we deduce that $V$ contains
$\ol{U_{A,\varepsilon}} \setminus \{0\}$ if and only if
$\roim{p}(\cor_{V\cap \ol{U_{A,\varepsilon}}}) |_{]0,\varepsilon]}
\simeq \cor_{]0,\varepsilon]}$.

We remark that, for $\varepsilon < 1$, we have
$\roim{p}(\cor_{V\cap \ol{U_{A,\varepsilon}}}) |_{]0,\varepsilon]}
\simeq \roim{p}(\cor_{V\cap \ol{U_{A,1}}}) |_{]0,\varepsilon]}$.
The sheaf $\roim{p}(\cor_{V\cap \ol{U_{A,1}}})$ is  constructible. Hence it is
constant on $]0,\varepsilon]$ for $\varepsilon>0$ small enough. Since
$(\roim{p}(\cor_{V\cap \ol{U_{A,1}}}))_{\varepsilon_n} \simeq \cor$ by
hypothesis, the conclusion follows.

\medskip\noindent
(iii) Now we check that $F$ is a sheaf on $\Msal$ with the criterion
of Proposition~\ref{pro:MV0}. Let $U,U_1,U_2\in\Op_\Msa$ such that
$\{U_1,U_2\}$ is a covering of $U$.

\smallskip\noindent
(iii-a) Let us prove that $F(U) \to F(U_1) \oplus F(U_2)$ is injective.
So we assume that $F(U) =\cor$ (otherwise this is obvious) and we prove
that $F(U_1) = \cor$ or $F(U_2) = \cor$.
Let $A>0$. By~\eqref{eq:germcusp1} and~\eqref{eq:resit} there exists
$\varepsilon_0>0$ such that
$$
\max \{ d((\varepsilon,0), M\setminus U_1),d((\varepsilon,0), M\setminus U_2) \}
\geq A \varepsilon^2 ,
\quad \text{for all $\varepsilon \in ]0,\varepsilon_0[$.}
$$
Hence, for any integer $n\geq 1$, the ball $B((1/n,0), A/n^2)$ is included
in $U_1$ or $U_2$. One of $U_1$ or $U_2$ must contain infinitely many such
balls. By~(ii) we deduce that it contains $U_{A,\varepsilon_A}$, for some
$\varepsilon_A>0$.  When $A$ runs over $\N$ we deduce that one of $U_1$ or
$U_2$ contains infinitely many sets of the type $U_{A,\varepsilon_A}$,
$A\in\N$.  Hence $F(U_1) = \cor$ or $F(U_2) = \cor$.

\smallskip\noindent
(iii-b) Now we prove that the kernel of $F(U_1) \oplus F(U_2) \to F(U_{12})$ is
$F(U)$.  We see easily that the only case where this kernel could be bigger
than $F(U)$ is $F(U_1) = F(U_2) = \cor$ and $F(U_{12}) = 0$. In this case,
for any $A>0$, there exist $\varepsilon_1, \varepsilon_2>0$ such that
$U_{A,\varepsilon_1} \subset U_1$ and $U_{A,\varepsilon_2} \subset U_2$.  This
gives $U_{A,\min\{\varepsilon_1, \varepsilon_2\}} \subset U_{12}$ which
contradicts $F(U_{12}) = 0$.

\medskip\noindent
(iv) By the definition of $F$ we have a natural morphism
$u\cl F\to \oim{\rhosal} \cor_{\{0\}}$ which is surjective.  We can see that
$\opb{\rhosal}(u)$ is an isomorphism. We define $N \in \md[\cor_\Msal]$ by the
exact sequence
\begin{equation}\label{eq:defN_germcusp}
  0 \to N \to F \to \oim{\rhosal} \cor_{\{0\}} \to 0 .
\end{equation}
Then $\opb{\rhosal}N \simeq 0$ but $N\not=0$. More precisely, for
$V\in \Op_\Msa$, we have $N(V)=0$ if $0\in V$ and $N(V) \isoto F(V)$ if
$0\not\in V$.
\end{example}

\section{$\sect$-acyclic sheaves}
\subsubsection*{\v{C}ech  complexes}
In this subsection, $\sht$ denotes either the site $\Msa$ or the site $\Msal$.

For a finite set $I$ and a family of open subsets $\{U_i\}_{i\in I}$ we set for $\emptyset\not=J\subset I$,
\eqn
&&U_J\eqdot\bigcap_{j\in J}U_j.
\eneqn

\begin{lemma}\label{le:Cechexact}
Let $\sht$ be either the site $\Msa$ or the site $\Msal$. Let 
 $\{U_1,U_2\}$ be a covering of $U_1\cup U_2$. Then the sequence 
\eq\label{eq:MV1}
&&0\to \cor_{U_{12}}\to\cor_{U_1}\oplus\cor_{U_2}\to\cor_{U_1\cup U_2}\to 0
\eneq
is exact.
\end{lemma}
\begin{proof}
The result is well-known for the site $\Msa$ and the functor $\oim{\rhosal}$ being left exact, 
it remains to show that $\cor_{U_1}\oplus\cor_{U_2}\to\cor_{U_1\cup U_2}$ is an epimorphism.
This follows from the fact that for any $F\in\md[\cor_\Msal]$, the map 
$\Hom[\cor_\Msal](\cor_{U_1\cup U_2},F)\to \Hom[\cor_\Msal](\cor_{U_1}\oplus\cor_{U_2},F)$ is a monomorphism.
\end{proof}
Consider now a finite family  $\{U_i\}_{i\in I}$ of objects of $\Op_\Msa$ and let $N\eqdot\vert I\vert$.
We choose a bijection $I=[1,N]$.
Then we have the \v{C}ech complex  in $\md[\cor_\sht]$ in which the term corresponding to $\vert J\vert=1$ is in degree $0$.
\eq\label{eq:Cech0}
&&\cor^\scbul_\shu\eqdot 
0\to \bigoplus_{J\subset I,\vert J\vert=N}\cor_{U_J}\tens e_J\to[d]\cdots 
\to[d] \bigoplus_{J\subset I,\vert J\vert=1}\cor_{U_J}\tens e_J\to0.
\eneq
Recall that $\{e_J\}_{\vert J\vert=k}$ is a basis of $\bigwedge^k\Z^N$ and the differential is defined as usual 
by sending  $\cor_{U_J}\tens e_J$ to $\bigoplus_{i\in I}\cor_{U_{J{\setminus \{i\}}}}\tens e_{i}\lfloor e_J$ using the natural morphism 
$\cor_{U_J}\to \cor_{U_{J{\setminus  \{i\}}}}$.

\begin{proposition}\label{pro:Cechexact}
Let $\sht$ be either the site $\Msa$ or the site $\Msal$. Let $U\in\Op_\Msa$ and let 
$\shu\eqdot\{U_i\}_i\in I$ be a finite covering of $U$ in $\sht$ \lp a regular covering in case $\sht=\Msal$\rp.
Then the natural morphism $\cor^\scbul_\shu\to\cor_U$ is a quasi-isomorphism. 
\end{proposition}
\begin{proof}
Recall that $N=\vert I\vert$. We may assume $I=[1,N]$. 
For $N=2$ this is nothing but Lemma~\ref{le:Cechexact}. We argue by induction and assume the result is proved 
for $N-1$. 
Denote by $\shu'$ the covering of $U'\eqdot \bigcup_{1\leq i\leq N-1}U_i$ by the family $\{U_i\}_{i\in [1,\dots, N-1]}$. 
Consider the subcomplex $F_1$ of $\cor^\scbul_\shu$ given by
\eq\label{eq:Cech00}
&&F_1\eqdot 
0\to \bigoplus_{N\in J\subset I,\vert J\vert=N}\cor_{U_J}\tens e_J\to[d]\cdots 
\to[d] \bigoplus_{N\in J\subset I,\vert J\vert=1}\cor_{U_J}\tens e_J\to0.
\eneq
Note that $F_1$ is isomorphic to the complex $\cor^\scbul_{\shu'\cap U_N}\to\cor_{U_N}$ where $\cor_{U_N}$ is in degree $0$
and we shall represent $F_1$ by this last complex.  
By~\cite[Th.~12.4.3]{KS06}, there is a natural morphism of complexes
\eq\label{eq:conecech1243}
&& u \cl \cor^\scbul_{\shu'}\,[-1] \to  \bl\cor^\scbul_{\shu'\cap U_N}
\to \cor_{U_N}\br 
\eneq
such that $\cor^\scbul_\shu$ is isomorphic to the mapping cone of $u$.
Hence, writing the long exact sequence associated with the mapping cone of $u$, 
we are reduced, by the induction hypothesis, to prove that the morphism
\eqn
&&\cor_{U'\cap U_N}\to\cor_{U'}\oplus\cor_{U_N}
\eneqn
is a monomorphism and its cokernel is isomorphic to $\cor_U$. 
Since $\{U',U_N\}$ is a covering of $U$, this follows from  Lemma~\ref{le:Cechexact}.
\end{proof}

\subsubsection*{Acyclic sheaves}
In this subsection, $\sht$ denotes either the site $\Msa$ or the site $\Msal$. 
In the literature, one often encounters sheaves which are $\sect(U;\scbul)$-acyclic for a given $U\in\sht$ 
but the next definition does not seem to be frequently used. 

\begin{definition}\label{def:acyclic}
Let $F\in\md[\cor_\sht]$. We say that $F$ is $\sect$-acyclic if we
have  $H^k(U;F)\simeq0$ for all $k>0$ and all $U\in\sht$. 
\end{definition}
We shall give criteria in order that a sheaf $F$ on the site $\sht$ be $\sect$-acyclic. 

Let $U\in\Op_\Msa$ and let 
$\shu\eqdot\{U_i\}_i\in I$ be a finite covering of $U$ in $\sht$ \lp a regular 
covering in case $\sht=\Msal$\rp. We denote by 
$C^\scbul(\shu;F)$ the associated \v{C}ech complex:
\eq\label{eq:Cech1}
&&C^\scbul(\shu;F)\eqdot \Hom[\cor_{\Msal}](\cor^\scbul_\shu,F).
\eneq
One can write more explicitly this complex as the complex:
\eq\label{eq:Cech11}
&&\
0\to \bigoplus_{J\subset I,\vert J\vert=1}F(U_J)\tens e_J\to[d]\cdots 
\to[d] \bigoplus_{J\subset I,\vert J\vert=N}F(U_J)\tens e_J\to0
\eneq
where the differential $d$ is obtained by sending $F(U_J)\tens e_J$ to 
$\bigoplus_{i\in I}F(U_J\cap U_{i})\tens e_{i}\wedge e_J$.

\begin{proposition}\label{pro:MV1}
Let $\sht$ be either the site $\Msa$ or the site $\Msal$ and let
$F\in\md[\cor_\sht]$.  The conditions below are equivalent.
\bnum
\item For any $\{U_1,U_2\}$ which is a covering of $U_1\cup U_2$, the sequence
  $0\to F(U_1\cup U_2)\to F(U_1)\oplus F(U_2) \to F(U_1\cap U_2) \to 0$ is exact.
\item The sheaf $F$ is $\sect$-acyclic.
\item
For any exact sequence in $\md[\cor_\sht]$
\eq\label{eq:Cech2}
&&G^\scbul \eqdot 0 \to \bigoplus_{i_0\in A_0}\cor_{U_{i_0}}
\to\cdots\to  \bigoplus_{i_N\in A_N}\cor_{U_{i_N}}\to0
\eneq
the sequence $\Hom[\cor_\sht](G^\scbul,F)$  is exact.
\item For any finite covering $\shu$ of $U$ \lp regular covering in case
  $\sht=\Msal$\rp, the morphism $F(U)\to C^\scbul(\shu;F)$ is a
  quasi-isomorphism.
\enum
\end{proposition}
\begin{proof}
(i)$\Rightarrow$(ii)~(a) Let $U\in\Op_\Msa$.  Let us first show that for any
exact sequence of sheaves $0\to F\to[\phi] F' \to[\psi] F''\to0$ and any
$U\in\Op_\Msa$, the sequence $0\to F(U)\to F'(U)\to F''(U)\to0$ is exact.  Let
$s''\in F''(U)$. By the exactness of the sequence of sheaves, there exists a
finite covering $U=\bigcup_{i=1}^NU_i$ and $s'_i\in F'(U_i)$ such that
$\psi(s'_i)= s''\vert_{U_i}$.  In case $\sht=\Msal$, we may assume that the
covering is regular by Proposition~\ref {pro:exregcov}.  For $k=1,\ldots,N$, we
set $V_k = \bigcup_{i=1}^kU_i$.  Let us prove by induction on $k$ that there
exists $t'_k \in F'(V_k)$ such that $\psi(t'_k)= s''\vert_{V_k}$.
Starting with $t'_1 = s'_1$ we assume that we have found $t'_k$.
Since our covering is regular, $\{V_k,U_{k+1}\}$ is a covering of $V_{k+1}$.
We set for short $W = V_k\cap U_{k+1}$.
We have $\psi(t'_k|_W) = \psi(s'_{k+1}|_W)$.
Hence there exists $s \in F(W)$ such that $\phi(s) = t'_k|_W -s'_{k+1}|_W$.
By hypothesis~(i) there exists $s_V \in F(V_k)$ and $s_U\in F(U_{k+1})$
such that $s = s_V|_W - s_U|_W$. Setting $t'_V = t'_k - \phi(s_V)$
and $s'_U = s'_{k+1} - \phi(s_U)$ we obtain $t'_U|_W = s'_V|_W$ and we can glue
$t'_U|_W$ and $s'_V|_W$ into $t'_{k+1} \in F'(V_{k+1})$.
We check easily that $\psi(t'_{k+1})= s''\vert_{V_{k+1}}$ and the induction
proceeds. 

\vspace{2ex}\noindent
(i)$\Rightarrow$(ii)~(b) 
 Denote by $\shj$ the full additive subcategory of $\md[\cor_\sht]$ consisting of sheaves 
satisfying the condition (i). We shall show that the category $\shj$  is $\sect(U;\scbul)$-injective for all $U\in\Op_\Msa$.
The category $\shj$ contains the injective sheaves.
By the first part of the proof, it thus remains to show that,
for any short exact sequence of sheaves
$F^\scbul\eqdot 0\to F'\to F\to F''\to 0$,
if  both $F'$ and $F$ belong to $\shj$, then $F''$ belongs to $\shj$. 

Let $U_1,U_2$ as in (i) and  denote by $\cor_{\shu}^\scbul$ the exact sequence  
$0\to\cor_{U_1\cap U_2}\to\cor_{U_1}\oplus\cor_{U_2}\to\cor_{U_1\cup U_2}\to 0$. 
Consider the double complex $\Hom[\cor_\sht](\cor_{\shu}^\scbul,F^\scbul)$.
By the preceding result all rows and columns except at most one
(either one row or one column depending how one writes the double complex) are exact. It follows that
 the double complex is exact.

\vspace{2ex}\noindent
(ii)$\Rightarrow$(iii) Consider an injective resolution $I^\scbul$ of $F$, that is, a complex $I^\scbul$ 
of injective sheaves such that the sequence $I^{\scbul,+}\eqdot 0\to F\to
I^\scbul$ is exact. The hypothesis implies that
$\sect(W;I^{\scbul,+})$ remains exact for all $W\in\Op_\Msa$. Then the argument goes as in the proof of 
(i)$\Rightarrow$(ii)~(b).
Recall that  $G^\scbul$ denotes the complex 
of~\eqref{eq:Cech2} and consider the double complex $\Hom[\cor_\sht](G^\scbul,I^{\scbul,+})$. Then all its rows 
and columns except one (either one row or one column depending how one writes the double complex) 
will be exact. It follows that all rows and columns are exact.

\vspace{2ex}\noindent
(iii)$\Rightarrow$(iv) follows from Proposition~\ref{pro:Cechexact}.

\vspace{2ex}\noindent
(iv)$\Rightarrow$(i) is obvious.
\end{proof}

\begin{corollary}\label{cor:limit-Gam-acyc}
Let $\sht$ be either the site $\Msa$ or the site $\Msal$.
A small filtrant inductive limit of $\sect$-acyclic sheaves is $\sect$-acyclic.
\end{corollary}
\begin{proof}
Since small filtrant inductive limits are exact in $\md[\cor]$, the family of sheaves satisfying 
condition~(i) of Proposition~\ref{pro:MV1} is stable by such limits by Lemma~\ref{le:sectUindlim}.
\end{proof}

\begin{definition}\label{def:flabby}
Let $\sht$ be either the site $\Msa$ or the site $\Msal$. 
One says that $F\in\md[\cor_\sht]$ is flabby if for any $U$ and $V$ in $\Op_\Msa$ with
 $V\subset U$, the natural morphism $F(U)\to F(V)$ is surjective. 
\end{definition}

\begin{lemma}\label{le:flabby}
Let $\sht$ be either the site $\Msa$ or the site $\Msal$.
\bnum
\item
Injective sheaves are flabby.
\item
Flabby sheaves are $\sect$-acyclic. 
\item
The category of flabby sheaves is stable by small filtrant inductive limits.
\enum
\end{lemma}
\begin{proof}
(i) Let $F$ be an injective sheaf and  let  $U$ and $V$ in $\Op_\Msa$ with $V\subset U$. 
Recall that the sequence $0\to\cor_V\to\cor_U$ is exact. Applying the functor 
$\Hom[\cor_\sht](\scbul,F)$ we get the result.

\vspace{2ex}\noindent
(ii) If $F\in\md[\cor_\sht]$ is flabby then it satisfies condition~(i) of Proposition~\ref{pro:MV1}. 

\vspace{2ex}\noindent
(iii) The proof of Corollary~\ref{cor:limit-Gam-acyc}
also works in this case. 
\end{proof}

\section{The functor $\epb{\rhosal}$}

In this section we make an essential use of the Brown representability
theorem (see for example~\cite[Th~14.3.1]{KS06}).

\subsubsection*{Direct sums in derived categories}
In this subsection, we state and prove some elementary results that we shall
need, some of them being well-known from the specialists.

\begin{lemma}\label{le:oplustau}
Let $\shc$ be a Grothendieck category and let $d\in\Z$. Then the cohomology 
functor $H^d$ and the truncation functors  
$\tau^{\leq d}$ and $\tau^{\geq d}$  commute with small direct sums in 
$\RD(\shc)$. In other words, if $\{F_i\}_{i\in I}$ is a small family of objects of $\RD(\shc)$, then
\eq\label{eq:oplustau}
&&\bigoplus_i\tau^{\leq d}F_i\isoto\tau^{\leq d}(\bigoplus_iF_i)
\eneq
and similarly with $\tau^{\geq d}$ and  $H^d$.
\end{lemma}
\begin{proof}
(i) The case of $H^d$ follows from~\cite[Prop.~10.2.8, Prop.~14.1.1]{KS06}. 

\vspace{0.2ex}\noindent
(ii) The morphism in~\eqref{eq:oplustau} is well-defined and it is enough to check that it 
induces an isomorphism on the cohomology. This follows from (i) since for any object $Y\in\RD(\shc)$, 
$H^j(\tau^{\leq d}Y)$ is either $0$ or $H^j(Y)$.  
\end{proof}

\begin{lemma}\label{le:comp1}
Let $\shc$ and $\shc'$ be two Grothendieck categories and let
$\rho\cl\shc\to\shc'$ be a left exact functor. Let $I$ be a small
category. Assume 
\bnum
\item
$I$ is either filtrant or discrete,
\item
$\rho$  commutes with inductive limits indexed by $I$,
\item
inductive limits indexed by $I$  
of injective objects in $\shc$ are acyclic for the functor $\rho$.
\enum
Then for all $j\in\Z$, the functor 
$R^j\rho\cl \shc\to\shc'$ commutes with inductive limits indexed by $I$.
\end{lemma}
\begin{proof}
Let $\alpha\cl I \to\shc$ be a functor. 
Denote by $\shi$ the full additive subcategory of  $\shc$ consisting of injective objects.
It follows for example from~\cite[Cor.~9.6.6]{KS06} that there exists a 
functor $\psi\cl I\to\shi$ and 
a morphism of functors $\alpha\to \psi$ such that for each $i\in I$,
$\alpha(i)\to\psi(i)$ is a monomorphism.
Therefore one can construct a functor $\Psi\cl I\to \RC^+(\shi)$ 
and a morphism of functor 
 $\alpha\to \Psi$ such that for each $i\in I$, $\alpha(i)\to\Psi(i)$ is a quasi-isomorphism.
 Set   $X_i=\alpha(i)$ and  $G_i^\scbul=\Psi(i)$. 
We get a qis $X_i\to G_i^\scbul$, hence a qis 
\eqn
&&\indlim[i]X_i\to\indlim[i]G_i^\scbul.
\eneqn
On the other hand, we have
\eqn
\indlim[i] R^j\rho(X_i)&\simeq& \indlim[i] H^j(\rho(G^\scbul_i))\\
&\simeq& H^j\rho(\indlim[i]G^\scbul_i)
\eneqn
where the second isomorphism follows from the fact that $H^j$ commutes with
direct sums and with filtrant inductive limits.
Then the result follows from hypothesis~(iii).
\end{proof}

\begin{lemma}\label{le:comp20}
We make the same hypothesis as in Lemma~\ref{le:comp1}.
Let $-\infty<a\leq b<\infty$, let $I$ be a small set and let $X_i\in\RD^{[a,b]}(\shc)$.  
Then
\eq\label{eq:Crsectoplus2}
&& \bigoplus_i R\rho(X_i)\isoto R\rho(\bigoplus_iX_i).
\eneq
\end{lemma}
\begin{proof}
The morphism in~\eqref{eq:Crsectoplus2} is well-defined and we have to prove it is an isomorphism.
If $b=a$, the result follows from Lemma~\ref{le:comp1}. The general case 
is deduced by induction on $b-a$ by considering the distinguished triangles
\eqn
&&  H^a(X_i)\,[-a]\to X_i\to \tau^{\geq a+1}X_i\to[+1].
\eneqn
\end{proof}

\begin{proposition}\label{pro:rhooplus}
Let $\shc$ and $\shc'$ be two Grothendieck categories and let
$\rho\cl\shc\to\shc'$ be a left exact functor. Assume that 
\banum
\item
$\rho$ has finite cohomological dimension,
\item
$\rho$ commutes with small direct sums,
\item
small direct sums of injective objects in $\shc$ are acyclic for
the functor $\rho$.
\eanum
Then 
\bnum
\item
the functor $R\rho\cl \RD(\shc)\to\RD(\shc')$ commutes with small direct sums,
\item
the functor  $R\rho\cl \RD(\shc)\to\RD(\shc')$ admits a right adjoint
 $\epb{\rho}\cl\RD(\shc')\to\RD(\shc)$,
\item
the functor $\epb{\rho}$ induces a functor $\epb{\rho}\cl\RD^+(\shc')\to\RD^+(\shc)$.
\enum
\end{proposition}
\begin{proof}
(i) Let $\{X_i\}_{i\in I}$ be a family of objects of $\RD(\shc)$.
It is enough to check that  the natural morphism in $\RD(\shc')$
\eq\label{eq:comp1}
&&\bigoplus_{i\in I}R\rho (X_i)\to R\rho(\bigoplus_{i\in I}X_i)
\eneq
 induces an isomorphism on the cohomology groups. 
 Assume that $\rho$ has cohomological dimension $\leq d$.
For $X\in \RD(\shc)$ and for $j\in\Z$, we have 
\eqn
&&\tau^{\geq j}R\rho (X)\simeq \tau^{\geq j}R\rho(\tau^{\geq j-d-1}X).
\eneqn
The functor $\rho$ being left exact we get for $k\geq j$:
\eq\label{eq:CHktruncated}
&& H^kR\rho (X)\simeq H^kR\rho(\tau^{\leq k}\tau^{\geq j-d-1}X).
\eneq
We have the sequence of isomorphisms:
\eqn
H^kR\rho(\bigoplus_iX_i)&\simeq& H^kR\rho(\tau^{\leq k}\tau^{\geq j-d-1}\bigoplus_i X_i)\\
&\simeq& H^kR\rho(\bigoplus_i\tau^{\leq k}\tau^{\geq j-d-1}X_i)\\
&\simeq& \bigoplus_iH^kR\rho(\tau^{\leq k}\tau^{\geq j-d-1}X_i)\\
&\simeq& \bigoplus_iH^kR\rho(X_i).
\eneqn
The first and last isomorphisms follow from~\eqref{eq:CHktruncated}.\\
The second isomorphism follows from Lemma~\ref{le:oplustau}.\\
The third isomorphism follows from Lemma~\ref{le:comp20}. 

\vspace{0.3ex}\noindent
(ii) follows from (i) and  the Brown representability theorem (see for example~\cite[Th~14.3.1]{KS06}).

\vspace{1ex}\noindent
(iii) This follows from hypothesis~(a) and (the well-known) Lemma~\ref{le:adjfcohdim} below.
\end{proof}

\begin{lemma}\label{le:adjfcohdim}
Let $\rho\cl\shc\to\shc'$ be a left exact functor between two Grothendieck categories.
Assume that $\rho\cl \RD(\shc)\to \RD(\shc')$ admits a right adjoint 
$\epb{\rho}\cl  \RD(\shc')\to \RD(\shc)$
and assume moreover that $\rho$ has finite cohomological dimension. 
Then the functor $\epb{\rho}$ sends $\RD^+(\shc')$ to $\RD^+(\shc)$.
\end{lemma}
\begin{proof}
By the hypothesis, we have for $X\in\RD(\shc)$ and $Y\in\RD(\shc')$
\eqn
&&\Hom[\RD(\shc')](\rho(X),Y)\simeq \Hom[\RD(\shc)](X,\epb{\rho}(Y)).
\eneqn
Assume that the cohomological dimension of the functor $\rho$ is $\leq r$.
Let $Y\in\RD^{\geq0}(\shc')$.  Then $\Hom[\RD(\shc)](X,\epb{\rho}(Y))\simeq 0$
for all $X\in\RD^{<-r}(\shc)$. This means that $\epb{\rho}(Y)$ belongs to the
right orthogonal to $\RD^{<-r}(\shc)$ and this implies that
$\epb{\rho}(Y)\in\RD^{\geq-r}(\shc')$.
\end{proof}

\subsubsection*{The functor $\rsect(U;\scbul)$}

\begin{lemma}\label{le:HjsectUindlim}
Let $\sht$ be either the site $\Msa$ or the site $\Msal$ and let
 $U\in\Op_\Msa$. Let $I$ be a small filtrant category and $\alpha\cl I \to\md[\cor_\sht]$ 
a functor. Set for short $F_i=\alpha(i)$.  Then for any $j\in \Z$
\eq\label{eq:Hjsectoplus}
&& \indlim[i]H^j\rsect(U;F_i)\isoto  H^j\rsect(U;\indlim[i]F_i).
\eneq
\end{lemma}
\begin{proof}
By Lemma~\ref{le:sectUindlim}, the functor $\sect(U;\scbul)$  commutes with small filtrant inductive limits and such limits of injective objects are $\sect(U;\scbul)$-acyclic by Lemma~\ref{le:flabby}. 
Hence, we may apply  Lemma~\ref{le:comp1}.
\end{proof}

\begin{proposition}\label{pro:sectUfd}
Let  $U\in\Op_\Msa$.
The functor $\sect(U;\scbul)\cl\md[\cor_\Msa]\to\md[\cor]$ has cohomological dimension $\leq\dim M$.
\end{proposition}
\begin{proof}
We know that if $F\in\mdrc[\cor_M]$, then $H^j\rsect(U;F)\simeq 0$ for $j>\dim M$. 
Since any $F\in\md[\cor_\Msa]$ is a small filtrant inductive limit 
of constructible sheaves, the result follows from Lemma~\ref{le:HjsectUindlim}.
\end{proof}

\begin{corollary}\label{cor:rhodimf}
Let $\shj$ be the subcategory of $\md[\cor_\Msa]$ consisting of sheaves which are $\sect$-acyclic. 
For any $F\in\md[\cor_\Msa]$, there exists an exact sequence
$0\to F\to F^0\to\cdots\to F^{n}\to0$
where $n=\dim M$ and the $F^j$'s belong to $\shj$.
\end{corollary}
\begin{proof}
Consider a resolution $0\to F\to I^0\to[d^0] I^1\to\cdots$ with the $I^j$'s injective and define  
$F^j=I^j$ for $j\leq n-1$, $F^j=0$ for $j>n$ and $F^{n}=\ker d^{n}$. 
It follows from Proposition~\ref{pro:sectUfd} that $F^{n}$ is $\sect$-acyclic. 
\end{proof}

\begin{proposition}\label{pro:sectUoplus}
Let $I$ be a small set and let $F_i\in\RD(\cor_\Msa)$ \lp$i\in I$\rp.
For  $U\in\Op_\Msa$, we have the natural isomorphism
\eq\label{eq:brown2}
&&\bigoplus_{i\in I}\rsect(U;F_i)\isoto \rsect(U;\bigoplus_{i\in I}F_i)\mbox{ in }\RD(\cor).
\eneq
\end{proposition}

\begin{proof}
The functor $\sect(U;\scbul)$ has finite cohomological dimension by Proposition~\ref{pro:sectUfd}, it commutes with small direct sums by Lemma~\ref{le:sectUindlim} and inductive limits of injective objects are $\sect(U;\scbul)$-acyclic
by Lemma~\ref{le:flabby}.
Hence, we may apply Proposition~\ref{pro:rhooplus}.
\end{proof}

\subsubsection*{The functor $\roim{\rhosal}$}
\begin{lemma}\label{le:rhodimfB}
Let $\shj$ be the subcategory of $\md[\cor_\Msa]$ consisting of sheaves which are $\sect$-acyclic. 
The category $\shj$ is $\oim{\rhosal}$-injective \lp see {\rm \cite[Cor.~13.3.8]{KS06}}\rp.
\end{lemma}
\begin{proof}
Let $0\to F'\to F\to F''\to 0$ be an exact sequence in $\md[\cor_\Msa]$.

\noindent
(i) We see easily that if both $F'$ and $F$ belong to $\shj$, then $F''$
belongs to $\shj$. 

\noindent
(ii) It remains to prove that if $F'\in\shj$, then the sequence 
$0\to \oim{\rhosal}F'\to \oim{\rhosal}F \to \oim{\rhosal}F''\to 0$ is exact. Let
$U\in\Op_\Msa$.  By Proposition~\ref{pro:sectrhoUF} and the hypothesis, the
sequence $0\to \oim{\rhosal}F'(U)\to \oim{\rhosal}F(U)\to \oim{\rhosal}F''(U)\to 0$ is
exact.
\end{proof}

Applying Corollary~\ref{cor:rhodimf}, we get: 

\begin{proposition}\label{pro:cohdimrho}
The functor $\oim{\rhosal}$ has cohomological dimension $\leq \dim M$.
\end{proposition}

\begin{proposition}\label{pro:sectUoplusB}
Let $I$ be a small set and let $F_i\in\RD(\cor_\Msa)$ \lp$i\in I$\rp.
We have the natural isomorphism
\eq\label{eq:brown3}
&&\bigoplus_{i\in I}\roim{\rhosal}F_i\isoto \roim{\rhosal}(\bigoplus_{i\in I}F_i)\mbox{ in }\RD(\cor_\Msal).
\eneq
\end{proposition}

\begin{proof}
By Proposition~\ref{pro:cohdimrho}, the functor $\oim{\rhosal}$ has finite cohomological dimension and  by Lemma~\ref{le:sectUindlim}
it commutes with small direct sums. Moreover,  inductive limits of injective objects are $\oim{\rhosal}$-acyclic
by Lemmas~\ref{le:rhodimfB} and~\ref{le:flabby}.
Hence, we may apply Proposition~\ref{pro:rhooplus}~(i).
\end{proof}

\begin{theorem}\label{th:rightadj}
\bnum
\item
The functor $\roim{\rhosal}\cl\RD(\cor_\Msa)\to\RD(\cor_\Msal)$ admits a right adjoint
 $\epb{\rhosal}\cl\RD(\cor_\Msal)\to\RD(\cor_\Msa)$.
\item
The functor $\epb{\rhosal}$ induces a functor $\epb{\rhosal}\cl\RD^+(\cor_\Msal)\to\RD^+(\cor_\Msa)$.
\enum
\end{theorem}
\begin{proof}
These results follow from Propositions~\ref{pro:sectUoplusB}  and~\ref{pro:cohdimrho},  as in 
Proposition~\ref{pro:rhooplus}.
\end{proof}
\begin{corollary}\label{cor:rightadj}
One has an isomorphism of functors on $\RD^+(\cor_\Msa)$:
\eq
&&\id\isoto\epb{\rhosal}\roim{\rhosal}.
\eneq
\end{corollary}
\begin{proof}
This follows from the fact that  $(\roim{\rhosal},\epb{\rhosal})$ is  a pair of adjoint functors
and that $\roim{\rhosal}$ is fully faithful by Lemma~\ref{le:rhoff}.
\end{proof}

\begin{remark}
(i) We don't know if the category $\Msal$ has finite flabby dimension. 
We don't even know if for any $F\in\Derb(\cor_\Msal)$ and any $U\in\Op_\Msa$, we have $\rsect(U;F)\in\Derb(\cor)$.

\vspace{0.3ex}\noindent
(ii) We don't know if the functor  $\epb{\rhosal}\cl\RD^+(\cor_\Msal)\to\RD^+(\cor_\Msa)$ constructed in Theorem~\ref{th:rightadj} induces a functor 
$\epb{\rhosal}\cl\Derb(\cor_\Msal)\to\Derb(\cor_\Msa)$.
\end{remark}

\section{Open sets with Lipschitz boundaries}

\subsubsection*{Normal cones and Lipschitz boundaries}

In this paragraph $\R^n$ is equipped with coordinates $(x',x_n)$,
$x'\in\R^{n-1}$, $x_n\in\R$.

\begin{definition}\label{def:lipschitzbd}
We say that $U\in\Op_\Msa$ has Lipschitz boundary  or simply that $U$ is
Lipschitz if, for any $x\in \partial U$, there exist an open
neighborhood $V$ of $x$ and a bi-Lipschitz subanalytic homeomorphism
$\psi\cl V\isoto W$ with $W$ an open subset of $\R^n$ such that
$\psi(V\cap U)=W\cap\{x_n>0\}$.
\end{definition}

\begin{remark}\label{rem:localLip}
(i) The property of being Lipschitz is local and thus the preceding
definition extends to subanalytic but not necessarily relatively compact open
subsets of $M$.

\vspace{0.2ex}\noindent
(ii) If $U_i$  is Lipschitz in $M_i$ ($i=1,2$) then $U_1\times U_2$ is Lipschitz in $M_1\times M_2$. 

\vspace{0.2ex}\noindent 
(iii) If $U$ is Lipschitz and $x\in \partial U$, there exist a constant $C>0$
and a sequence $\{y_n\}_{n\in \N}$, $y_n\in U$, such that $d(y_n,x) \to 0$ and
$d(y_n,x) \leq C d(y_n, \partial U)$, for all $n \in \N$ (in the notations of
the definition, assume $\psi(x) = (x',0)$ and set $y_n = \opb{\psi}(x',1/n)$).
\end{remark}

\begin{example}
(i) Lemma~\ref{le:gopenlipschitz} below will provide many examples of Lipschitz open sets.

\vspace{0.2ex}\noindent
(ii) Let $(x,y)$ denotes the coordinates on $\R^2$.  Using~(iii) of
Remark~\ref{rem:localLip} we see that the open set $U=\{(x,y);0<y<x^2\}$ is not
Lipschitz.
\end{example}

\begin{lemma}\label{le:lipsimpliesvdef}
Let $U\in\Op_\Msa$.
We assume that, for any $x\in \partial U$, there exist an open neighborhood $V$
of $x$ and a bi-analytic isomorphism $\psi\cl V\isoto W$ with $W$ an open
subset of $\R^n$ such that $\psi(V\cap U)= W\cap\{(x',x_n);\; x_n>\phi(x')\}$
for a Lipschitz subanalytic function $\phi$. Then $U$ is Lipschitz.
\end{lemma}
\begin{proof}
We define $\psi_1 \cl \R^n \to \R^n$, $(x',x_n) \mapsto (x', x_n-\phi(x'))$.
Then $\psi_1$ is a bi-Lipschitz subanalytic homeomorphism and we have
$(\psi_1 \circ \psi) (V\cap U) = \psi_1(W) \cap \{x_n>0\}$. Hence $U$ is
Lipschitz.
\end{proof}

\begin{lemma}\label{le:gopenlipschitz}
Let $\BBV$ be a vector space and let $\gamma$ be a proper closed convex cone
with non empty interior. Let $U \in \Op_{\BBV_{\rm sa}}$.
Then the open set $U+\gamma$ has Lipschitz boundary.
\end{lemma}
\begin{proof}
Let $p \in \partial(U+\gamma)$.  We identify $\BBV$ with $\R^n$ so that $p$ is
the origin and $\gamma$ contains the cone
$\gamma_0 = \{ (x',x_n);\; x_n > \| x' \| \}$.
We have in particular
\begin{equation}\label{eq:gopenlip1}
\gamma_0 \subset (U+\gamma) \subset (\R^n \setminus (-\gamma_0)).
\end{equation}
For $x'\in \R^{n-1}$ we set $l_{x'} = (U+\gamma) \cap (\{x'\} \times \R)$.
Then $l_{x'} = l_{x'} + [0,+\infty[$.  By~\eqref{eq:gopenlip1} we also have
$l_{x'} \not=\emptyset$ and $l_{x'} \not= \R$.  Hence we can write
$l_{x'} = ]\phi(x'),+\infty[$, for a well-defined function
$\phi \cl \R^{n-1} \to \R$.

Let us prove that $\phi$ is Lipschitz.  Let $x'\in \R^{n-1}$ and let us set
$q = (x',\phi(x')) \in \partial(U+\gamma)$.  We have the similar inclusion
as~\eqref{eq:gopenlip1},
$(q+\gamma_0) \subset (U+\gamma) \subset (\R^n \setminus (q-\gamma_0))$.
Hence $\partial (U+\gamma)  \subset
(\R^n \setminus ((q+\gamma_0) \cup (q-\gamma_0)))$.
For any $y'\in \R^{n-1}$ we have $(y',\phi(y')) \in \partial(U+\gamma)$
and the last inclusion translates into
$| \phi(y') - \phi(x') | \leq \| y' - x' \|$.  Hence $\phi$ is Lipschitz and
$U+\gamma$ is Lipschitz by Lemma~\ref{le:lipsimpliesvdef}.
\end{proof}

We refer to~\cite[Def~4.1.1]{KS90} for the
definition of the normal cone $C(A,B)$ associated with two subsets $A$ and $B$ of $M$.

\begin{definition}\lp See~{\rm\cite[\S~5.3]{KS90}.}\rp\,
Let $S$ be a subset of $M$. 
The strict normal cone $N_x(S)$  and the conormal cone $N^*_x(S)$ of $S$ at
$x\in M$ as well as the strict normal cone $N(S)$  and the conormal cone
$N^*(S)$ of $S$ are given by
\eqn
&&N_x(S)= T_xM\setminus C(M\setminus S,S),\mbox{ an open cone in $T_xM$},\\
&&N^*_x(S)=N_x(S)^\circ \mbox{ (where ${}^\circ$ denotes the polar cone), }\\
&&N(S)=\bigcup_{x\in M}N_x(S),\mbox{ an open convex cone in $TM$},\\
&&N^*(S)=\bigcup_{x\in M}N^*_x(S).
\eneqn
\end{definition}
By loc.\ cit. Prop.~5.3.7, we have: 

\begin{lemma}\label{le:conecond}
Let $U$ be an open subset of $M$ and let $x\in\partial U$. Then the conditions below are equivalent:
\bnum
\item
$N_x(U)$ is non empty,
\item
$N_y(U)$ is non empty for all $y$ in a neighborhood of $x$,
\item 
$N^*_x(U)$ is contained in a closed convex proper cone with non empty interior in $T^*_xM$, 
\item
there exists a local chart in a neighborhood of $x$ 
such that identifying $M$ with an open subset of $\BBV$, there exists 
a  closed convex proper cone with non empty interior $\gamma$ in $\BBV$  such that 
$U$ is $\gamma$-open in an open  neighborhood $W$ of $x$, that is, 
\eqn
&&W\cap((U\cap W)+\gamma)\subset U.
\eneqn
\enum
\end{lemma}

\begin{definition}\label{def:conecond}
We shall say that an open subset $U$ of $M$ satisfies a cone condition 
if for any $x\in\partial U$, $N_x(U)$ is non empty.
\end{definition}

By Lemmas~\ref{le:gopenlipschitz} and~\ref{le:conecond} we have:
\begin{proposition}
Let $U\in\Op_\Msa$. If $U$ satisfies a cone condition, then $U$ is Lipschitz.
\end{proposition}

\begin{remark}
One shall be aware that our definition of being Lipschitz differs from that of
Lebeau in~\cite{Le14}. By Lemma~\ref{le:lipsimpliesvdef}, if $U$ is Lipschitz in
Lebeau's sense, then it is Lipschitz in our sense.
\end{remark}

\subsubsection*{A vanishing theorem}
The next theorem is a key result  for this paper and its proof is due to A.~Parusinski~\cite{Pa14}. 

\begin{theorem}\label{th:ParuTh}{\rm (A.~Parusinski)}
Let $V\in\Op_{\Msa}$. Then there exists a finite covering $V=\bigcup_{j\in J}V_j$ with $V_j\in\Op_\Msa$ such that 
the family $\{V_j\}_{j\in J}$ is a covering of $V$ in $\Msal$ and moreover
$H^k(V_j;\cor_M)\simeq 0$ for all $k>0$ and all $j\in J$.
\end{theorem}
Recall that one denotes by $\rhosal\cl\Msa\to\Msal$ the natural morphism of sites.

\begin{lemma}\label{le:cormsal0}
We have $\roim{\rhosal}\cor_\Msa\simeq\cor_\Msal$. 
\end{lemma}
\begin{proof}
The sheaf $H^k(\roim{\rhosal}\cor_\Msa)$ is the sheaf associated with the presheaf
$U\mapsto H^k(U;\cor_\Msa)$. This sheaf if zero for $k>0$ by
Theorem~\ref{th:ParuTh}.
\end{proof}

\begin{lemma}\label{le:cormsal1}
Let $M=\R^n$ and set $U = {}]0,+\infty[ \times \R^{n-1}$.
Then we have $\roim{\rhosal}\cor_U \simeq\cor_U$. 
\end{lemma}
\begin{proof}
(i) The sheaf $H^k(\roim{\rhosal}\cor_U)$ is the sheaf associated with the
presheaf $V\mapsto H^k(V;\cor_U)$.
Hence it is enough to show that any $V\in\Op_{\Msa}$ admits a finite covering
$V=\bigcup_{j\in J}V_j$ in $\Msal$ such that $H^k(V_j;\cor_U)\simeq 0$ for all
$k>0$.
 We assume that the distance $d$ is a subanalytic function. 
Let us set $V_{<0} = V \cap (]-\infty,0[ \times \R^{n-1})$ and
$V' = V_{<0}^{1,V}$, where we use the notation~\eqref{eq:enlarge_to_covA}
with $\varepsilon=1$. In our case we can write~\eqref{eq:enlarge_to_covA}
as follows
$$
V' = \{x\in V; \; d(x, V \setminus U) < d(x, M\setminus V)\} .
$$
This is a subanalytic open subset of $V$.  By Lemma~\ref{lem:enlarge_to_cov}
we have
\eq\label{eq:cov_of_V}
\text{$\{V',V\cap U\}$ is a covering of $V$ in $\Msal$.}
\eneq

\medskip\noindent
(ii) Let us prove that $\rsect(V'; \cor_U) \simeq 0$.
We denote by $(x_1,x')$ the coordinates   on $M=\R^n$. For $x=(x_1,x')$ with 
$x_1\geq 0$, we have $d(x, V \setminus U) \geq d(x, M \setminus U) = x_1$.
If $(x_1,x') \in V'$ we obtain $d(x, M \setminus V) > x_1$, hence 
$\ol{B(x,x_1)} \subset V$,
where $B(x,x_1)$ is the ball with center $x$ and
radius $x_1$. This proves that $V'\cap \ol U$ is contained in the right
hand side of the following equality 
\eq\label{eq:descrVprime}
& V'\cap \ol U = \{ x= (x_1,x') \in V; \, x_1\geq 0 \text{ and }
\ol{B(x,x_1)} \subset V\} 
\eneq
and the reverse inclusion is easily checked.
It follows that, if $(x_1,x') \in V'\cap \ol U$, then
$(y_1,x') \in V'\cap \ol U$, for all $y_1\in [0,x_1]$.
Let $q\cl \R^n \to \{0\}\times\R^{n-1}$ be the projection. We deduce:

\smallskip
(a) $q$ maps $V'\cap \ol U$ onto $V\cap \partial U$,

(b) $\opb{q}(x) \cap V'\cap U$ is an open interval,
for any $x=(0,x') \in V\cap \partial U$.

\smallskip\noindent
For any $c<0<d$ we have $\rsect(]c,d[; \cor_{]0,d[}) \simeq 0$.
Hence (a) and (b) give $\roim{q}\rsect_{V'}\cor_U \simeq 0$,
by the base change formula, and we obtain
$\rsect(V'; \cor_U) \simeq \rsect(\R^{n-1}; \roim{q}\rsect_{V'}\cor_U)\simeq 0$.

\medskip\noindent
(iii) By Theorem~\ref{th:ParuTh} we can choose a finite covering of
$V\cap U$ in $\Msal$, say $\{W_j\}_{j\in J}$, such that
$H^k(W_j;\cor_U)\simeq 0$ for all $k>0$.
By~\eqref{eq:cov_of_V} the family $\{V', \{W_j\}_{j\in J}\}$ is a covering of
$V$ in $\Msal$.
By~(ii) this covering satisfies the required condition in~(i), which proves
the result. 
\end{proof}

We need to extend Definition~\ref{def:lipschitzbd}. 
\begin{definition}\label{def:lipschitzbd2}
We say that $U\in\Op_\Msa$ is weakly Lipschitz if for each $x\in M$ there exists a neighborhood 
$V\in\Op_\Msa$ of $x$, a finite set $I$ and $U_i\in\Op_\Msa$, $i\in I$, such that $U\cap V=\bigcup_{i\in I}U_i$ and
\eq\label{hyp:wlip}
&&\left\{\parbox{60ex}{
for all $\emptyset\not=J\subset I$, the set $U_J=\bigcap_{j\in J}U_j$ is a disjoint union of Lipschitz open sets.
}\right.\eneq
\end{definition}
By its definition,  the property of being weakly Lipschitz is local on $M$. 

\begin{example}
The open subset $U = \R^2 \setminus \{0\}$ of $\R^2$ is not
Lipschitz but it is weakly Lipschitz: setting $U_\pm = \{(x,y)\in \R^2$; $\pm y>-|x|\}$ we have $U=U_+ \cup U_-$ and $U_+$, $U_-$, $U_+\cap U_-$ are disjoint
unions of Lipschitz open subsets.
\end{example}

\begin{proposition}\label{exa:Uminusa}
Let $U\in\Op_\Msa$ and consider a finite family of smooth submanifolds $\{Z_i\}_{i\in I}$, closed in a 
neighborhood of $\ol U$. Set $Z=\bigcup_{i\in I}Z_i$. 
Assume that 
\banum
\item
$U$ is Lipschitz, 
\item
$Z_i\cap Z_j\cap\partial U=\emptyset$ for $i\not= j$, 
$\partial U$ is smooth in a neighborhood of 
$Z\cap\partial U$ and the intersection is transversal,
\item
in a neighborhood of each point of $Z\cap U$ there exist a local coordinate system $(x_1,\dots,x_n)$ and for each $i\in I$,
a subset $I_i$ of $\{1,\dots,n\}$ such that  $Z_i=\bigcap_{j\in I_i}\{x;x_j=0\}$.
\eanum
Then $U\setminus Z$ is  weakly Lipschitz. 
\end{proposition}
\begin{proof}
Since the property of being weakly Lipschitz is local on $M$, it is enough to prove the result in a neighborhood of each point $p\in\ol U$. 

\vspace{0.3ex}\noindent
(i) Assume $p\in\partial U$. We choose a local coordinate system  $(x_1,\dots,x_n)$ centered at $p$ such that $U=\{x;x_n>0\}$ and $Z=\{x;x_1=\dots=x_r=0\}$ (with $r<n$). 
For $1\leq i\leq r$, define $U_i=\{x;x_n>0, x_i\not=0\}$. Then the family $\{U_i\}_{i=1,\dots,r}$ satisfies~\eqref{hyp:wlip}.

\vspace{0.3ex}\noindent
(ii) Assume $p\in U$. We choose a local coordinate system $(x_1,\dots,x_n)$ such
that $Z_i=\bigcap_{j_i\in I_i}\{x;x_{j_i}=0\}$. For each $j_i\in I_i$ and
$\varepsilon_i=\pm 1$, define $U_i^{\varepsilon_i,j_i}=\{x;\; \varepsilon_i\,
x_{j_i}>0\}$.  Set $A = \prod_{i\in I} (\{\pm 1\} \times I_i)$ and for $\alpha
\in A$ set $U_\alpha=\bigcap_{i\in I}U_i^{\alpha_i}$. Then the family
$\{U_\alpha\}_{\alpha \in A}$ satisfies~\eqref{hyp:wlip}.  
\end{proof}

\begin{theorem}\label{th:Lipbnd}
Let $U\in\Op_\Msa$ and assume that $U$ is  weakly  Lipschitz. Then
\bnum
\item
$\roim{\rhosal}\cor_{U\Msa} \simeq \oim{\rhosal}\cor_{U\Msa} \simeq \cor_{U\Msal}$
is concentrated in degree zero.
\item
For  $F\in\Derb(\cor_\Msal)$, one has $\rsect(U;\epb{\rhosal}F)\simeq\rsect(U;F)$.
\item
Let $F\in\md[\cor_\Msal]$ and assume that $F$ is $\sect$-acyclic. 
Then $\rsect(U;\epb{\rhosal}F)$ is concentrated in degree $0$ and is isomorphic to $F(U)$.
\enum
\end{theorem}
Note that the result in (i) is local and it is not necessary to assume here that $U$ is relatively compact. 
\begin{proof}
(i)--(a) First we assume that $U$ is Lipschitz. 
The first isomorphism  is a local problem. Hence, by Remark~\ref{rem:bilip_homeo} and by the
definition of ``Lipschitz boundary'' the first isomorphism follows from
Lemma~\ref{le:cormsal1}. The second isomorphism is given in
Proposition~\ref{pro:corUsasal}.

\vspace{0.2ex}\noindent 
(i)--(b) The first isomorphism  is a local problem and we may assume that
$U$ has a covering by open sets $U_i$ as in 
Definition~\ref{def:lipschitzbd2}.
By using the \v{C}ech resolution associated with this covering, we find an exact
sequence of sheaves in $\md[\cor_\Msa]$:
\eqn
&&0\to L_r\to\cdots\to L_0\to\cor_U\to 0
\eneqn
where each $L_i$ is a finite sum of sheaves isomorphic to $\cor_V$ for some $V\in\Op_\Msa$ with $V$ Lipschitz.
Therefore, $\roim{\rhosal}L_i$ is concentrated in degree $0$ by (i)--(a) and the result follows.

\medskip\noindent
(ii) follows from (i) and the adjunction between $\roim{\rhosal}$ and $\epb{\rhosal}$.

\medskip\noindent
(iii) follows from (ii).
\end{proof}

\begin{example}
\label{exa:U12aa}
Let $M=\R^2$ endowed with coordinates $x=(x_1,x_2)$. Let $R>0$ and denote by
$B_R$ the open Euclidian ball with center $0$ and radius $R$. Consider the
subanalytic sets:
\eqn
&U_1=\{x\in B_R; x_1>0,x_2<x_1^2\},\quad U_2=\{x\in B_R; x_1>0,x_2>-x_1^2\},\\
&U_{12}=U_1\cap U_2,\quad U=U_1\cup U_2=\{x \in B_R;x_1>0\}.
\eneqn
Note that $\{U_1,U_2\}$ is a covering of $U$ in $\Msa$ but not in $\Msal$.
Denote for short by $\rho\cl\Msa\to\Msal$ the morphism $\rhosal$. 
We have the distinguished triangle in $\Derb(\cor_\Msal)$:
\eq\label{eq:dtU12a}
&\roim{\rho}\cor_{U_{12}}\to\roim{\rho} \cor_{U_1}\oplus\roim{\rho}\cor_{U_2}
\to\roim{\rho}\cor_{U}\to[+1].
\eneq
Since $U_1,U_2$ and $U$ are Lipschitz, $\roim{\rho}\cor_V$ is concentrated in
degree $0$ for $V=U_1,U_2,U$. It follows that $\roim{\rho}\cor_{U_{12}}$ is
concentrated in degrees $0$ and $1$.
Hence, we have the  distinguished triangle
\eq\label{eq:dtU12b}
&& \oim{\rho}\cor_{U_{12}}\to \roim{\rho}\cor_{U_{12}}\to R^1\oim{\rho}\cor_{U_{12}}[-1]\to[+1].
\eneq
Let us prove that $R^1\oim{\rho}\cor_{U_{12}}$ is isomorphic to the sheaf $N$
introduced in~\eqref{eq:defN_germcusp}.  We easily see that there exists a
natural morphism $\cor_U \to N$ which is surjective. Hence we have to prove
that the sequence
$$
\cor_{U_1} \oplus \cor_{U_2} \to \cor_U \to N 
$$
is exact. This reduces to the following assertion: if $V \in\Op_\Msa$ satisfies
$V\subset U$ and $N(V) = 0$, then $\{V\cap U_1, V\cap U_2\}$ is a linear
covering of $V$. We prove this claim now.

\smallskip\noindent
Let $V\subset U$ be such that $N(V) = 0$. By the definition of $N$, there
exists $A>0$ such that $U_{A,\varepsilon} \not\subset V$ for all
$\varepsilon>0$, where $U_{A,\varepsilon}$ is defined in~\eqref{eq:defUAeps}.
Hence there exists a sequence $\{(x_{1,n},x_{2,n})\}_{n\in \N}$ such that
$x_{1,n}>0$, $x_{1,n}\to 0$ when $n\to \infty$, $|x_{2,n}| < A x_{1,n}^2$ and
$(x_{1,n},x_{2,n}) \not\in V$, for all $n\in\N$.  We define
$f(x) = d((x,0), M\setminus V)$, for $x\in\R$.  Then $f$ is a continuous
subanalytic function and $f(x_{1,n}) < A x_{1,n}^2$, for all $n\in\N$.
The set $\{x \in ]0,1[; f (x) < Ax^2 \}$ is subanalytic and relatively compact,
hence it is a finite disjoint union of points (but it is open) and intervals.
Since it contains a sequence converging to $0$, it must contain some interval
$]0,x_0[$.  We have then $f(x) \leq A x^2$ for all $x\in ]0,x_0[$.  We deduce,
for any $(x_1,x_2)\in \R^2$ with $x_1\in ]0,x_0[$,
\begin{equation}\label{eq:U12aa1}
d((x_1,x_2), M\setminus V)
\leq |x_2| + d((x_1,0), M\setminus V) \leq |x_2| + Ax_1^2.  
\end{equation}
On the other hand we can find  $B>0$ such that, for any $(x_1,x_2) \in U$,
\begin{equation}\label{eq:U12aa2}
\max \{d((x_1,x_2) , M\setminus U_1),  d((x_1,x_2) , M\setminus U_2) \}
\geq |x_2| + B x_1^2.
\end{equation}
We deduce easily from~\eqref{eq:U12aa1} and~\eqref{eq:U12aa2} that
$\{V\cap U_1, V\cap U_2\}$ is a linear covering of $V$.
\end{example}

\chapter{Operations on sheaves}\label{ShvSa3}

All along this chapter, we follow Convention~\eqref{eq:distonM}.

In this chapter we study the natural operations on sheaves for the linear subanalytic topology.
 In particular, given a morphism of real analytic manifolds, 
 our aim is to define inverse and direct images for sheaves on the linear subanalytic topology. We are not able to do it in general (see Remark~\ref{rem:embedsubm}) and we shall distinguish the case of a closed embedding and the case of a submersion.

\section{Tensor product  and internal hom}
Since $\Msal$ is a site, the category $\md[\cor_\Msal]$ admits a tensor product, denoted
$\scbul\tens\scbul$ 
and an internal hom, denoted $\hom$, and these functors admit left and right derived
functors, 
respectively. For more details, we refer to~\cite[\S~18.2]{KS06}.

\section{Operations for closed embeddings}
\subsubsection*{$f$-regular open sets}
In this section, $f\cl M\into N$ will be a closed embedding.  We identify $M$
with a subset of $N$. We assume for simplicity that $d_M$ is the restriction of
$d_N$ to $M$ and we write $d$ for $d_M$ or $d_N$. We also keep the
convention~\eqref{eq:dist-au-vide} for $d(x,\emptyset)$.

\begin{definition}\label{def:fregular}
Let $V\in\Op_\Nsa$. We say that $V$ is $f$-regular if 
there exists $C>0$ such that
\eq\label{eq:fregular}
&& d(x,M\setminus M\cap V) \leq C \, d(x,N\setminus V)
\quad\mbox{  for all $x\in M$.}
\eneq
\end{definition}

\begin{itemize}
\item The property of being $f$-regular is local on $M$. More precisely, if
  $M=\bigcup_{i\in I}U_i$ is an open covering and $V\in\Op_\Nsa$ is
  $f\vert_{U_i}$-regular for each $i\in I$, then $V$ is $f$-regular.
\item If $V$ and $W$ belong to $\Op_\Nsa$ with $\opb{f}(V)=\opb{f}(W)$,
  $V\subset W$ and $V$ is $f$-regular, then $W$ is $f$ regular.
\end{itemize}

\begin{lemma}\label{le:interfregular}
Let  $f\cl M\into N$ be a closed embedding.
The family $\{V\in\Op_\Nsa;$ $V$ is $f$-regular$\}$
is stable by finite intersections.
\end{lemma}
\begin{proof}
We shall use the obvious fact which asserts that for two closed sets $F_1$ and
$F_2$ in a metric space, 
\eqn
&&d(x,F_1\cup F_2)=\inf(d(x,F_1),d(x,F_2)). 
\eneqn
Let $V_1$ and $V_2$ be two $f$-regular objects of $\Op_\Nsa$ and let $C_1$ and
$C_2$ be the corresponding constants as in~\eqref{eq:fregular}.
Let $x\in M$. We have
\eqn
d(x,M\setminus (M\cap V_1\cap V_2))
& = & \inf_i \,d(x,M\setminus (M\cap V_i) ) \\
&\leq& \inf_i (C_i \cdot d(x,N\setminus V_i)) \\
&\leq& (\max_i \, C_i) \cdot ( \inf_i \, d(x,N\setminus V_i) ) \\
&=& (\max_i \, C_i) \cdot d(x,N\setminus (V_1\cap V_2)).
\eneqn 
\end{proof}

\begin{lemma}\label{le:coveroim}
Let  $f\cl M\into N$ be a closed embedding and let
 $U \in \Op_\Msa$.
Then there exists $V \in \Op_\Nsa$ such that $V$ is $f$-regular
and $M\cap V = U$.
\end{lemma}
\begin{proof}
We choose $V_0 \in \Op_\Nsa$ such that $\ol U \subset V_0$.
We set
$$
\delta = \inf\{ d(x,N\setminus V_0);\; x\in U\}
$$
and $V = (V_0 \setminus (V_0\cap M) ) \cup U$. We have $\delta>0$.
Let $x\in M$ and $y\in N$ be such that $d(x,N\setminus V) = d(x,y)$.
If $y\in M$, then $d(x,N\setminus V) = d(x,M\setminus U)$.
If $y\not\in M$, then $d(x,N\setminus V) = d(x,N\setminus V_0) \geq \delta$.
In any case we have $d(x,N\setminus V) \geq  \min\{d(x,M\setminus U),\delta\}$.
Hence~\eqref{eq:fregular} is satisfied with $C = \max\{1, D/\delta\}$,
where $D = \max\{d(x,M\setminus U);\; x\in M\} < \infty$.
\end{proof}

\begin{lemma}\label{lem:freg-subcov}
Let  $f\cl M\into N$ be a closed embedding.
Let $V\in \Op_\Nsal$ be an $f$-regular open set and let $\{V_i\}_{i\in I}$ be a
linear covering of $V$, that is, a covering in $ \Op_\Nsal$.  Then there exists
a refinement $\{W_j\}_{j\in J}$ of $\{V_i\}_{i\in I}$ such that 
 $\{W_j\}_{j\in J}$ is a linear covering of $V$ and $W_j$ is
$f$-regular for all $j\in J$. We can even choose $J=I$ and $W_i\subset V_i$,
for all $i\in I$.
\end{lemma}
\begin{proof}
Let $C$ be a constant as in~\eqref{eq:fregular}.
Let $I_0\subset I$ be a finite subset and let $C'>0$ be such that
\eq\label{eq:freg-subcov3}
d(x,N\setminus V)\leq C'\cdot\max_{i\in I_0} d(x,N \setminus V_i)
\quad\text{ for all $x\in N$}.
\eneq
Then, for any $x\in M$ we have
\eq\label{eq:freg-subcov1}
&\begin{aligned}
d(x,M\setminus (M\cap V)) &\leq C\cdot d(x,N\setminus V) \\
&\leq CC'\cdot \max_{i\in I_0}  d(x,N \setminus V_i) .
\end{aligned}
\eneq
We set $D= 2CC'$. For $i\in I_0$ we define $W_i \in \Op_\Nsal$ by
$$
W_i = (V_i \setminus M) \cup
\{x\in M\cap V_i ;\; d(x,M\setminus (M\cap V)) <  D\, d(x,N \setminus V_i) \}
$$
and for $i\in I\setminus I_0$ we set $W_i = \emptyset$.

\medskip\noindent
(i) Since $D\geq CC'$, the inequality~\eqref{eq:freg-subcov1} gives
$V = \bigcup_{i\in I_0} W_i$. Let us prove that $\{W_i\}_{i\in I_0}$ is a
linear covering of $V$.  We first prove the following claim, for given
$\varepsilon>0$, $i\in I_0$ and $x\in N$:
\eq\label{eq:freg-subcov11} &
\begin{minipage}{9cm}
if $d(x,N\setminus W_i) \leq \varepsilon d(x,N\setminus V)$, \\
then $d(x,N\setminus V_i) \leq (\varepsilon(1+\frac{C}{D}) +\frac{C}{D})
d(x,N\setminus V)$.  
\end{minipage}
\eneq
If $d(x,N\setminus W_i) = d(x,N\setminus V_i)$, the claim is obvious.
In the other case we choose $y\in N$ such that $d(x,N\setminus W_i) = d(x,y)$.
Then we have $y\in V_i \setminus W_i$. Hence $y\in M$ and the definition of
$W_i$ gives $d(y,N \setminus V_i) \leq D^{-1} d(y,M\setminus (M\cap V))$.
We deduce
\begin{align*}
d(x,N\setminus V_i) &\leq d(x,y) + d(y,N \setminus V_i) \\
&\leq d(x,y) + D^{-1} d(y,M\setminus (M\cap V)) \\
&\leq d(x,y) + CD^{-1} d(y,N\setminus V) \\
&\leq (1+CD^{-1}) d(x,y) + CD^{-1} d(x,N\setminus V) \\
&\leq (\varepsilon(1+CD^{-1}) +CD^{-1})d(x,N\setminus V) ,
\end{align*}
which proves~\eqref{eq:freg-subcov11}.

Now we prove that $\{W_i\}_{i\in I_0}$ is a linear covering of $V$.
We choose $\varepsilon$ small enough so that
$(\varepsilon(1+\frac{C}{D}) +\frac{C}{D}) < \frac{1}{C'}$
(recall that $D=2CC'$) and we prove, for all $x\in N$,
\eq\label{eq:freg-subcov2}
d(x,N\setminus V)
\leq \varepsilon^{-1}\cdot\max_{i\in I_0} d(x,N \setminus W_i).
\eneq
Indeed, if~\eqref{eq:freg-subcov2} is false, then~\eqref{eq:freg-subcov11}
implies $d(x,N\setminus V_i) < \frac{1}{C'} d(x,N\setminus V)$ for some
$x\in V$ and all $i\in I_0$. But this contradicts~\eqref{eq:freg-subcov3}.

\medskip\noindent
(ii) Let us prove that $W_i$ is $f$-regular, for any $i\in I_0$.
We remark that $W_i\setminus M = V_i \setminus M$. Hence
$d(x,N \setminus W_i) = d(x,N \setminus V_i)$ or
$d(x,N \setminus W_i) =  d(x,M\setminus(M\cap W_i))$, for all $x\in M$.
In the first case we have, assuming $x\in M\cap W_i$
(the case $x\not\in M\cap W_i$ being trivial),
\begin{multline*}
d(x,M\setminus (M\cap W_i)) \leq d(x,M\setminus (M\cap V)) \\ 
\leq  D\, d(x,N \setminus V_i) = D\, d(x,N \setminus W_i)  .
\end{multline*}
In the second case we have obviously
$d(x,M\setminus (M\cap W_i)) \leq  d(x,N \setminus W_i)$.
Hence~\eqref{eq:fregular} holds for $W_i$ with the constant $\max\{D,1\}$.
\end{proof}

Thanks to Lemma~\ref{le:interfregular}, to $f$ we can associate a new site.
\begin{definition}
Let  $f\cl M\to N$ be a closed embedding.
\bnum
\item
The presite $\Nf$ is given by $\Op_\Nf=\{V\in\Nsa; V\mbox{ is $f$-regular}\}$.
\item
The site $\Nsalf$ is the presite $\Nf$ endowed with the topology such that 
a family $\{V_i\}_{i\in I}$ of objects $\Op_{\Nf}$  is a covering of $V$ in $\Nf$ 
if it is a covering in $\Nsal$.
\enum
One denotes by $\rho_f\cl \Nsal\to \Nsalf$ the natural morphism of sites. 
\end{definition}

\begin{proposition}\label{pro:morphismNsalf}
The functor  $f^t\cl  \Op_\Nsalf\to \Op_\Msa$, $V\mapsto \opb{f}(V)$, induces a 
morphism of sites $\tw f\cl \Msal\to\Nsalf$. Moreover, this functor of sites
is left exact in the sense of~{\rm\cite[Def.~17.2.4]{KS06}}. 
\end{proposition}
\begin{proof}
(i) Let $C$ be a constant as in~\eqref{eq:fregular}.
Let $\{V_i\}_{i\in I}$ be a covering of $V$ in $\Nsal$ and let $I_0\subset I$
be a finite subset and $C'>0$ be such that $d(y,N\setminus V)\leq
C'\cdot\max_{i\in I_0} d(y,N \setminus V_i)$ for all $y\in N$.  We deduce, for
$x\in M$,
\eqn
d(x,M\setminus M\cap V) &\leq& C\cdot d(x,N\setminus V) \\
&\leq& CC'\cdot \max_{i\in I_0}  d(x,N \setminus V_i)\\
 &\leq& CC'\cdot  \max_{i\in I_0}  d(x,M \setminus M\cap V_i).
\eneqn

\vspace{0.3ex}\noindent
(ii) We have to prove that the functor $f^t\cl \Op_\Nsalf\to\Op_\Msa$ is left
exact in the sense of~\cite[Def.~3.3.1]{KS06}, that is, for each
$U\in \Op_\Msa$, the category whose objects are the inclusions
$U\to \opb{f}(V)$ ($V\in\Op_\Nsalf$) is cofiltrant.

This category is non empty by Lemma~\ref{le:coveroim} and then it is cofiltrant
by Lemma~\ref{le:interfregular}.
\end{proof}

Hence, we have the morphisms of sites
\eq\label{diag:siteNf}
&&\ba{c}\xymatrix{
&\Nsal\ar[d]^-{\rho_f}\\
\Msal\ar[r]_-{\tw f}&\Nsalf.
}\ea\eneq

Now we consider two closed embeddings $f\cl M\to N$ and $g\cl N\to L$ of real
analytic manifolds and we set $h\eqdot g\circ f$.  We get the diagram of
presites:
\eq\label{diag:sitesNfLgLh}
&&\vcenter{\xymatrix{
\Nsa\ar[r]^-{\tw g}\ar[d]_-{\rho_f}&\Lg\ar[rdd]^-{\lambda_h}&\Lsa\ar[l]_-{\rho_g}\ar[dd]^-{\rho_h}\\
\Nf\ar[rrd]^-{\ol g}&&\\
\Msa\ar[u]^-{\tw f}\ar[rr]^-{\tw h}&& \Lg \cap \Lh ,
}}
\eneq

\noindent
where the objects of $ \Lg \cap \Lh$ are the open sets $U\in\Op_\Lsa$ which are both $g$-regular and $h$-regular,
$\ol g$ is induced by $\tw g$ and $\lambda_h$ is the obvious inclusion.
We will use the following lemma to prove that the direct images defined in the
next section are compatible with the composition.

\begin{lemma}\label{le:twgrhof1}
\bnum
\item
Let $W\in\Op_\Lh$. Then $W\cap N \in\Op_\Nf$.
\item Let $W\in \Op_\Lg$ be such that $N\cap W \in \Op_\Nf$.
Then $W\in \Op_\Lh$.
\item 
Let $W\in \Op_\Lg$ and $V\in\Op_\Nf$ be such that $V\subset N\cap W$.
Then there exists $U\in \Op_\Lg\cap \Op_\Lh$ such that $U\subset W$ and
$V \subset N\cap U$.
\enum
\end{lemma}
\begin{proof}
(i) By hypothesis there exists $C>0$ such that
$d(x,M\setminus M\cap W) \leq C\, d(x, L \setminus W)$, for any $x\in M$.
Since $d(x, L \setminus W) \leq d(x, N \setminus N\cap W)$ we deduce~(i).

\medskip\noindent
(ii) By hypothesis there exist $C_1,C_2 >0$ such that, for any $x\in M$,
$$
d(x,M\setminus M\cap W) \leq C_1 d(x, N \setminus N\cap W)
\leq C_1 C_2 d(x, L \setminus W),
$$
which proves the result.

\medskip\noindent
(iii) By Lemma~\ref{le:coveroim} there exists $U_0 \in \Op_\Lg$ such that
$N\cap U_0= V$.  Then $U = U_0 \cap W$ is $g$-regular by
Lemma~\ref{le:interfregular} and $N\cap U = V$. Hence $U$ is also $h$-regular
by~(ii).
\end{proof}

\subsubsection*{Inverse and direct images by closed embeddings}

Let us first recall the inverse and direct images of presheaves.

\begin{notation}
(i) For a morphism $f\cl\sht_1\to\sht_2$ of presites, we denote by $\oim{f}$
and $\opbdag{f}$ the direct and inverse image functors for presheaves.

\vspace{0.3ex}\noindent
(ii) We recall that the direct image functor
$\oim{f}$ has a right  adjoint 
$f^{\ddag} \cl \Psh(\cor_{\sht_2}) \to \Psh(\cor_{\sht_1})$ 
defined as follows (see~\cite[(17.1.4)]{KS06}).
For $P\in \Psh(\cor_{\sht_2})$ and $U\in \Op_{\sht_1}$ we have
$(f^{\ddag}P)(U) = \prolim[ f^t(V)\to U] P(V)$. 
\end{notation}

\begin{lemma}\label{lem:rhofddag=faisc}
Let  $f\cl M\to N$ be a closed embedding and let $G \in \md[\cor_\Nsalf]$.
Then, using the notations of~\eqref{diag:siteNf}, we have
$\rho_f^{\ddag} G \in \md[\cor_\Nsal]$.
\end{lemma}
\begin{proof}
We have to prove that, for any $V\in \Op_\Nsa$ and any covering of $V$ in
$\Nsal$, say $\{V_i\}_{i\in I}$, the following sequence is exact
\eq\label{eq:rhofddag=faisc1}
&& 0 \to \prolim[W\subset V]G(W)
\to \prod_{i\in I} \prolim[W_i\subset V_i]G(W_i)
\to \prod_{i,j\in I}\prolim[W_{ij}\subset V_i\cap V_j]G(W_{ij}),
\eneq
where $W$, $W_i$, $W_{ij}$ run respectively over the $f$-regular open subsets
of $V$, $V_i$, $V_i\cap V_j$.
The limit in the second term of~\eqref{eq:rhofddag=faisc1} can be replaced by
the limit over the pairs $(W,W_i)$ of $f$-regular open subsets with
$W\subset V$, $W_i\subset W\cap V_i$.
Then the family $\{W\cap V_i\}_{i\in I}$ is a covering of $W$ in $\Nsal$.
By Lemma~\ref{lem:freg-subcov} it admits a refinement $\{W'_i\}_{i\in I}$
where the $W'_i$'s are $f$-regular and $W'_i \subset V_i$.
We may as well assume that $W_i$ contains $W'_i$, for any $i\in I$.
Then $\{W_i\}_{i\in I}$ is a covering of $W$ in $\Nsalf$.
Hence the second term of~\eqref{eq:rhofddag=faisc1} can be replaced by
$$
\prolim[W\subset V] \; \prolim[\{W_i\}_{i\in I}] \; \prod_{i\in I} G(W_i),
$$
where $W$ runs over the $f$-regular open subsets of $V$ and the family
$\{W_i\}_{i\in I}$ runs over the coverings of $W$ in $\Nsalf$ such that
$W_i \subset W\cap V_i$.

Now in the third term of~\eqref{eq:rhofddag=faisc1} we may assume that $W_{ij}$
contains $W_i\cap W_j$ and the exactness of the sequence follows from the
hypothesis that $G \in \md[\cor_\Nsalf]$.
\end{proof}

\begin{definition}\label{def:opbgf}
Let $f\cl M\to N$ be a closed embedding.
We use the notations of~\eqref{diag:siteNf}.
\bnum
\item We denote by $\oimsal{f}\cl\md[\Msal]\to\md[\Nsal]$ the functor
  $\opbddag{\rho_f} \circ \oim{\tw f}$ and we call $\oimsal{f}$ the direct
  image functor.
\item We denote by $\opbsal{f}\cl\md[\Nsal]\to\md[\Msal]$ the functor
  $\opb{{\tw f}} \circ \oim{\rho_f}$ and we call $\opbsal{f}$ the inverse image
  functor.
\enum
\end{definition}
For $F\in \md[\Msal]$, $G\in \md[\Nsal]$, $U\in \Op_\Msal$ and $V\in \Op_\Nsal$,
we obtain
\eq
\label{eq:section-oimsal}
\sect(V;\oimsal{f} F) &\simeq &
\prolim[ W\in\Op_{N^f}, \, W \subset V] F(M\cap W) , \\
\label{eq:section-opbsal}
\sect(U;\opbsal{f} G) &\simeq &
\indlim[W\in\Op_{N^f},\, W\cap M = U] G(W) .
\eneq

\begin{lemma}\label{le:twgrhof2}
Let $f\cl M\to N$ and $g\cl N\to L$ be closed embeddings and let $h=g\circ f$.
We use the notations of the diagram~\eqref{diag:sitesNfLgLh}.
There is a natural isomorphism of functors
\eq\label{eq:twgrhof}
&&\oim{\tw g}\circ\opbddag{\rho_f}\isoto\opbddag{\lambda_h}\circ\oim{\ol g}.
\eneq
\end{lemma}
\begin{proof}
The morphisms of functors
$\oim{\lambda_h}\circ\oim{\tw g}\circ\opbddag{\rho_f}
\simeq \oim{\ol g}\circ\oim{\rho_f}\circ\opbddag{\rho_f} \to \oim{\ol g}$
gives by adjunction the morphism in~\eqref{eq:twgrhof}. To prove that this
morphism is an isomorphism, let us choose $G\in\Psh(\cor_\Nf)$ and
$W\in\Op_\Lg$. We get the morphism
\eq\label{eq:twgrhofW}
&& \sect(W; (\oim{\tw g}\circ\opbddag{\rho_f})(G))
\to \sect(W; (\opbddag{\lambda_h}\circ\oim{\ol g})(G)) ,
\eneq
where $\sect(W; (\oim{\tw g}\circ\opbddag{\rho_f})(G))
\simeq \prolim[V\in\Op_\Nf ,\, V\subset N\cap W]G(V)$ and 
$\sect(W; (\opbddag{\lambda_h}\circ\oim{\ol g})(G)) 
\simeq\prolim[U\in\Op_\Lh,\, U\subset W]G(N\cap U)$.
Then the result follows from Lemma~\ref{le:twgrhof1}.
\end{proof}

\begin{proposition}\label{prop:compo-embed}
Let $f\cl M\to N$ and $g\cl N\to L$ be closed embeddings and let $h=g\circ f$.
There is a natural isomorphism of functors
$\oimsal{g}\circ\oimsal{f} \isoto \oimsal{h}$.
\end{proposition}
\begin{proof}
Applying Lemma~\ref{le:twgrhof2}, we define the isomorphism as
the composition
$ \opbddag{\rho_g}\circ\oim{\tw g}\circ\opbddag{\rho_f}\circ\oim{\tw f}
\isoto
\opbddag{\rho_g}\circ\opbddag{\lambda_h}\circ\oim{\ol g}\circ\oim{\tw f}
\simeq \opbddag{\rho_h}\circ\oim{\tw h}$.
\end{proof}

\begin{theorem}\label{th:oimopbfsal}
Let $f\cl M\to N$ be a closed embedding.
\bnum
\item
The functor $\oimsal{f}$ is right adjoint to the functor 
$\opbsal{f}$.
\item
The functor  $\oimsal{f}$ is left exact and the functor $\opbsal{f}$ is exact. 
\item
If $g\cl N\to L$ is another closed embedding, we have 
$\oimsal{(g\circ f)}\simeq \oimsal{g}\circ\oimsal{f}$ and $\opbsal{(g\circ f)}\simeq \opbsal{f}\circ\opbsal{g}$.
\enum
\end{theorem}
\begin{proof}
(i) We have $\oimsal{f}= \opbddag{\rho_f}\circ\oim{\tw f}$ and 
$\opbsal{f}=\opb{\tw f}\circ\oim{\rho_f}$.
Since $(\opb{\tw f},\oim{\tw f})$ and $(\oim{\rho_f},\opbddag{\rho_f})$
are pairs of adjoint functors between categories of presheaves and since the
category of sheaves is a fully faithful subcategory of the category of
presheaves, the result follows.

\vspace{0.3ex}\noindent
(ii) By the adjunction property, it remains to show that  functor $\opbsal{f}$ is left exact, hence that the functor 
$\opb{\tw f}$ is exact. By Proposition~\ref{pro:morphismNsalf}, 
the morphism of sites $\tw f\cl \Msal\to\Nsalf$
is left exact in the sense of~\cite[Def.~17.2.4]{KS06}. Then the result follows from~\cite[Th.~17.5.2]{KS06}.

\vspace{0.3ex}\noindent
(iii) The functoriality of direct images follows from
Proposition~\ref{prop:compo-embed} and that of inverse images results by
adjunction.
\end{proof}

\section{Operations for  submersions}

Let $f\cl M\to N$ denote a morphism of real analytic manifolds.  In this
section we assume that $f$ is a submersion.  If $f$ is proper, it induces a
morphism of sites $\Msal \to \Nsal$, but otherwise, it does not even give a
morphism of presites. Following~\cite{KS01} we  shall introduce other sites $\Msb$
(denoted $\Msa$ in loc.\ cit.), similar to $\Msa$ but containing all open
subanalytic subsets of $M$, and $\Msbl$, similar to $\Msal$.  Then $\Msbl$ has
the same category of sheaves as $\Msal$ and any submersion $f\cl M\to N$
induces a morphism of sites $\fsbl\cl\Msbl\to\Nsbl$.

\subsubsection*{Another subanalytic topology}

One denotes by $\Op_{\Msb}$ the category of open subanalytic subsets of $M$ and
says that a family $\{U_i\}_{i\in I}$ of objects of $\Op_{\Msb}$ is a covering
of $U\in\Op_{\Msb}$ if $U_i\subset U$ for all $i\in I$ and, for each compact
subset $K$ of $M$, there exists a finite subset $J\subset I$ such that
$\bigcup_{j\in J}U_j\cap K=U\cap K$.  We denote by $\Msb$ the site
so-defined. The next result is obvious (and already mentioned in~\cite{KS01}).
\begin{proposition}
The morphism of sites $\Msb\to\Msa$ induces an equivalence of categories 
$\md[\cor_{\Msb}]\simeq\md[\cor_\Msa]$. 
\end{proposition}

Similarly, we introduce another linear subanalytic topology $\Msbl$ as
follows. The objects of the presite $\Msbl$ are those of $\Msb$, namely the open
subanalytic subsets of $M$. In order to define the topology, we have to
generalize Definitions~\ref{def:regsit1} and~\ref{def:linearcov}.

\begin{definition}\label{def:regsit1b}
Let $\{U_i\}_{i\in I}$ be a finite family in $\Op_{\Msb}$.
We say that this family is $1$-regularly situated 
if for any compact subset $K\subset M$, there is a constant $C$ such that for any $x\in K$
\eq\label{eq:resitb}
&&d(x,M\setminus\bigcup_{i\in I}U_i)\leq C\cdot
\max_{i\in I}  d(x,M \setminus U_i).
\eneq
\end{definition}

\begin{definition}\label{def:linearcovb}
A linear covering of $U\in\Op_\Msb$ is a small family $\{U_i\}_{i\in I}$ of
objects of $\Op_{\Msb}$ such that $U_i\subset U$ for all $i\in I$ and 
\eq\label{eq:lcoveringsb} 
&&\left\{\parbox{64ex}{ for each relatively compact subanalytic open subset  $W\subset M$ 
there exists   a finite subset $I_0\subset I$ such that the  family $\{W\cap U_i\}_{i\in I_0}$ 
is $1$-regularly situated in $W$ and $\bigcup_{i\in I_0}(U_i\cap W)=U\cap W$.
}
\right.
\eneq
\end{definition}

\begin{proposition}\label{pro:eqvsblsal}
\bnum
\item
The family of  linear coverings satisfies the axioms
of Grothendieck topologies.
\item 
The functor $\oim{\rho}$ associated with the morphism of sites
$\rho\cl\Msbl\to\Msal$ defines an equivalence of categories
$\md[\cor_{\Msbl}]\simeq\md[\cor_\Msal]$.
\enum
\end{proposition}
The verification is left to the reader.

\subsubsection*{Inverse and direct images}

\begin{proposition}\label{prop:fsbl-submersion}
Let $f\cl M\to N$ be a morphism of real analytic manifolds.
We assume that $f$ is a submersion.
Then $f$ induces a morphism of sites $\fsbl \cl \Msbl \to \Nsbl$.
\end{proposition}
\begin{proof}
Let $V\in \Op_\Nsb$ and let $\{V_i\}_{i\in I}$ be a linear covering of $V$.  We
have to prove that $\{\opb{f}V_i\}_{i\in I}$ is a linear covering of $\opb{f}V$.
As in the case of $\Msa$, the definition of the linear coverings is local (see
Corollary~\ref{cor:1covlocal}). Hence we can assume that $M=N\times L$.
We can also assume that $d_M((x,y),(x',y')) = \max \{ d_N(x,x'),d_L(y,y')\}$,
for $x,x' \in N$ and $y,y'\in L$. Then for any $(x,y)\in M$ we have
$d_M((x,y), N\setminus \opb{f}V) = d_N(x, N\setminus V)$
and the result follows easily.
\end{proof}

By Propositions~\ref{pro:eqvsblsal} and~\ref{prop:fsbl-submersion} any
submersion $f\cl M\to N$ between real analytic manifolds induces a pair of
adjoint functors $(\opbsal{f}, \oimsal{f})$ between $\md[\Msal]$ and
$\md[\Nsal]$ and we get the analogue of Theorem~\ref{th:oimopbfsal}:

\begin{theorem}\label{th:oimopbfsalsub}
Let $f\cl M\to N$ be a submersion.
\bnum
\item
The functor $\oimsal{f}$ is right adjoint to the functor 
$\opbsal{f}$.
\item
The functor  $\oimsal{f}$ is left exact and the functor $\opbsal{f}$ is exact. 
\item
If $g\cl N\to L$ is another submersion, we have 
$\oimsal{(g\circ f)}\simeq \oimsal{g}\circ\oimsal{f}$ and $\opbsal{(g\circ f)}\simeq \opbsal{f}\circ\opbsal{g}$.
\enum
\end{theorem}

\begin{remark}\label{rem:embedsubm}
Our two definitions of $\oimsal{f}$ for closed embeddings and submersions
do not give a definition for a general $f$ by composition.
For example let us consider the following commutative diagram
$$
\xymatrix@C=1.5cm{
M = \R^2   \ar[r]^i \ar[d]_p \ar[dr]^f  & \R^3 \ar[d]^q \\
\R   \ar[r]^j   & N = \R^2 ,
}
$$
where $i(x,y) = (x,y,0)$, $p(x,y) = x$, $q(x,y,z) = (x,z)$ and $j(x) = (x,0)$.
For $V\in \Op_\Nsb$ we define two families of open subsets of $\opb{f}(V)$:
\begin{align*}
I_1 &= \{M\cap W;\, W\in \Op_{\R^3_\sb},\, W \subset\opb{q} V, \, 
 \text{$W$ is $i$-regular} \} , \\
I_2 &= \{\opb{p}(\R\cap V');\; V' \in \Op_\Nsb,\, V' \subset V, \,
 \text{$V'$ is $j$-regular} \}.
\end{align*}
Then, for any $F\in \md[\Msbl]$ we have
\begin{align}
\label{eq:rem2defa}
\sect(V; \oimsal{q} \oimsal{i} F)
&\simeq \sect( \opb{q} V; \oimsal{i} F) \simeq \prolim[U\in I_1] F(U) ,\\
\label{eq:rem2defb}
\sect(V; \oimsal{j} \oimsal{p} F)
& \simeq \prolim[V' \subset V, \, \text{$V'$ $j$-regular}] 
 \sect(\R\cap V'; \oimsal{p}F)  \simeq \prolim[U\in I_2] F(U).
\end{align}
Let us take for $V$ the open set $V = \{(x,z);\, x^3 > z^2\}$.  Then the two
families $I_1$ and $I_2$ of open subsets of $\opb{f}(V) = \{(x,y); \; x>0\}$
are not cofinal. Indeed the set $W_0 \subset \R^3$ given by
$W_0 = \{(x,y,z);\, x^3 > y^2+z^2 \}$ is $i$-regular.
Hence $M\cap W_0 = \{(x,y);\, x^3 > y^2\}$ belongs to $I_1$.
On the other hand we see easily that, if $V'$ is $j$-regular and $V'\subset V$,
then $\R\cap V' \subset \mathopen{]}\varepsilon,+\infty[$, for some
$\varepsilon>0$.  Hence $M\cap W_0$ is not contained in any set of the family
$I_2$.

Let us define
$F = \indlim[\varepsilon>0] \cor_{[0,\varepsilon] \times \{0\}} \in \md[\Msbl]$.
Taking $U=M\cap W_0$ in~\eqref{eq:rem2defa} we can see that
$\sect(V; \oimsal{q} \oimsal{i} F) \simeq \cor$.
On the other hand~\eqref{eq:rem2defb} implies
$\sect(V; \oimsal{j} \oimsal{p} F) \simeq 0$.
Hence $\oimsal{q} \oimsal{i} \not\simeq \oimsal{j} \oimsal{p}$.
\end{remark}

\chapter{Construction of sheaves}\label{ShvSa4}

On the site $\Msa$, the sheaves $\Cinft_\Msa$ and $\Dbt_\Msa$ below have been
constructed in~\cite{KS96,KS01}.  By using the linear topology we shall
construct sheaves on $\Msal$ associated with more precise growth conditions.

All along this chapter, we follow Convention~\ref{eq:distonM}.

\section{Sheaves on the subanalytic site}

\subsubsection*{Temperate growth}
For the reader's convenience, let us recall first some definitions of~\cite{KS96,KS01}.
As usual, we denote by 
$\shhc_M^\infty$    (resp.\ $\sha_M$) 
the sheaf of complex valued functions of class $\Cinf$ (resp.\ real analytic), 
by $\Db_M$  (resp.\ $\shhb_M$) 
the sheaf of Schwartz's distributions (resp.\ Sato's hyperfunctions)
and by $\shhd_M$ the sheaf of  finite-order differential operators with coefficients in $\sha_M$.

\begin{definition}\label{def:cinftyt}
Let $U\in\Op_\Msa$ and let $f\in \Cinf_M(U)$. One says that $f$ has
{\it  polynomial growth} at $p\in M\setminus U$
if it satisfies the following condition.
For a local coordinate system
$(x_1,\dots,x_n)$ around $p$, there exist
a sufficiently small compact neighborhood $K$ of $p$
and a positive integer $N$
such that
\eq\label{eq:cinftyt0}
&&\sup_{x\in K\cap U}\big(d(x,K\setminus U)\big)^N\vert f(x)\vert
<\infty.
\eneq
We say that $f$ is  temperate at $p$ if all its derivatives
have polynomial growth at $p$. We say that $f$ is temperate 
if it is temperate at any point $p\in M\setminus U$.
\end{definition}

For $U\in\Op_\Msa$, 
we shall denote by $\Cinft_M(U)$ the subspace
of $\Cinf_M(U)$ consisting of temperate functions.

For $U\in\Op_\Msa$, we shall denote by $\Dbt_M(U)$ the space
of temperate distributions on $U$, defined by the exact sequence
\eqn
&&0\to\sect_{M\setminus U}(M;\Db_M)\to\sect(M;\Db_M)\to\Dbt_M(U)\to 0.
\eneqn    

It follows from~\eqref{eq:resitsuba} that  $U\mapsto \Cinft_M(U)$ is a sheaf and it follows from
the work of Lojasiewicz~\cite{Lo59}
that $U\mapsto \Dbt_M(U)$ is also a sheaf. 
We denote by $\Cinft_\Msa$ and $\Dbt_\Msa$ these sheaves on $\Msa$. 
The first one is called the sheaf of $\Cinf$-functions with temperate growth and the second the sheaf of temperate 
distributions.
Note that both sheaves are $\sect$-acyclic (see~\cite[Lem~7.2.4]{KS01} or Proposition~\ref{pro:Csoft0} below) and 
the sheaf $\Dbt_\Msa$ is flabby (see Definition~\ref{def:flabby}).

We also introduce the sheaf $\Cinf_\Msa$ of 
$\Cinf$-functions on $\Msa$ 
as
\eqn
&&\Cinf_\Msa\eqdot\oim{\rhosa}\Cinf_M. 
\eneqn
We denote as usual by $\shd_M$ the sheaf of rings of finite order differential operators on the real analytic manifold $M$. If $\iota_M\cl M\into X$ is a complexification of $M$, then 
$\shd_M\simeq\opb{\iota_M}\shd_X$. 
We set, following~\cite{KS01}:
\eq\label{def:DMsa}
&&\shhd_\Msa\eqdot\eim{\rhosa}\shhd_M.
\eneq
The sheaves  $\Cinft_\Msa$, $\Cinf_\Msa$ and  $\Dbt_\Msa$  
are $\shd_\Msa$-modules. 

\begin{remark}
The sheaves $\Cinft_\Msa$ and $\Dbt_\Msa$ are  respectively denoted by 
$\Cin[t]_M$ and $\Db_M^t$ in~\cite{KS01}.
\end{remark}

\subsubsection*{A cutoff lemma on $\Msa$}
Lemma~\ref{le:Ho00}  below  is an 
immediate corollary of a result of
H\"ormander~\cite[Cor.1.4.11]{Ho83} and was already used  in~\cite[Prop.~10.2]{KS96}.

\begin{lemma}\label{le:Ho00}
Let $Z_1$ and $Z_2$ be two closed subanalytic subsets of $M$.
Then there exists $\psi\in\Cinft_M(M\setminus(Z_1\cap Z_2))$ such that
$\psi=0$ on a neighborhood of $Z_1\setminus Z_2$ and $\psi=1$ on a neighborhood
of $Z_2\setminus Z_1$.
\end{lemma}

\begin{proposition}\label{pro:Csoft0} 
Let $\shf$ be a sheaf of  $\Cinft_\Msa$-modules on $\Msa$.
Then $\shf$ is $\sect$-acyclic. 
\end{proposition}
\begin{proof}
By Proposition~\ref{pro:MV1}, it is enough to prove that 
for $U_1,U_2$ in $\Op_\Msa$,  the sequence
$0\to \shf(U_1\cup U_2)\to \shf(U_1)\oplus \shf(U_2)\to \shf(U_1\cap U_2)\to 0$
is exact.  
This follows from Lemma~\ref{le:Ho00} (see~\cite[Prop.~10.2]{KS96} or Proposition~\ref{pro:Csoft} below).  
\end{proof}

\subsubsection*{Gevrey growth}
The definition below of the sheaf  $\GG_\Msa$ is inspired by the definition of the sheaves of
$\Cinf$-functions of Gevrey classes, but is completely different from the classical one. Here we are interested in the growth of functions at the boundary contrarily to the classical setting where one is interested in the Taylor expansion of the function. As usual, there are two kinds of regularity which can be interesting: regularity at the interior or at the boundary. Since we shall soon consider the Dolbeault complexes of our new sheaves, the interior regularity is irrelevant and we are only interested in the growth at the boundary.

We refer to~\cite{Ko73b} for an exposition on classical Gevrey functions or distributions  and their link with Sato's theory of boundary values of holomorphic functions. Note that there is also a recent study by~\cite{HM11} of these sheaves using the tools of subanalytic geometry. 

In \S~\ref{section:shvonMsal} we shall define more refined sheaves by using the linear subanalytic topology.

\begin{definition}\label{def:GG}
Let $U\in\Op_\Msa$ and let $f\in \Cinf_M(U)$. We
say that $f$ has {\it $0$-Gevrey growth} at $p\in M\setminus U$ if
it satisfies the following condition.  For a local coordinate system
$(x_1,\dots,x_n)$ around $p$, there exist a sufficiently small compact
neighborhood $K$ of $p$, $h>0$ and $s>1$ such that
\eq\label{eq:gevrey}
&&\sup_{x\in K\cap U}\big(\exp(-h\cdot d(x,K\setminus U)^{1-s})\big)\vert f(x)\vert<\infty.
\eneq
We say that $f$ has Gevrey  growth at $p$   if all its derivatives
have  $0$-Gevrey growth at $p$. We say that $f$  has Gevrey growth
if it has such a growth at any point $p\in M\setminus U$.

We denote by $G_M(U)$ the subspace of $\Cinf_M(U)$ consisting of functions with 
Gevrey growth and by $\GG_\Msa$ the presheaf $U\mapsto G_M(U)$ on $\Msa$. 
\end{definition}

The next result is clear in view of~\eqref{eq:resitsuba} and Proposition~\ref{pro:Csoft0}.
\begin{proposition}
\banum
\item
The presheaf  $\GG_\Msa$ is a sheaf on $\Msa$,
\item
the sheaf $\GG_\Msa$ is a $\shd_\Msa$-module,
\item
the sheaf $\GG_\Msa$ is a $\Cinft_\Msa$-module,
\item
the sheaf $\GG_\Msa$ is $\sect$-acyclic.
\eanum
\end{proposition}

\section{Sheaves on the linear subanalytic site}\label{section:shvonMsal}

By Lemma~\ref{le:rhodimfB}, 
 if a sheaf $\shf$ on $\Msa$ is $\sect$-acyclic, then $\roim{\rhosal}\shf$ is concentrated in degree $0$. This applies in particular to the sheaves  
 $\Cinft_\Msa$, $\Dbt_\Msa$ and $\GG_\Msa$.
 
 In the sequel, we shall  use the following notations. We set
\eqn
&&\Cinft_\Msal\eqdot \oim{\rhosal}\Cinft_\Msa,\quad \Dbt_\Msal\eqdot \oim{\rhosal}\Dbt_\Msa,\quad
 \GG_\Msal\eqdot \oim{\rhosal}\GG_\Msa.
\eneqn

\subsubsection*{Temperate growth of a given order}
\begin{definition}\label{def:cinftyt2}
Let $U\in\Op_\Msa$, let $f\in \Cinf_M(U)$ and let $t\in\R_{\geq0}$. We say that $f$ has
{\it  polynomial growth of order $\leq t$} at $p\in M\setminus U$
if it satisfies the following condition.
For a local coordinate system
$(x_1,\dots,x_n)$ around $p$, there exists
a sufficiently small compact neighborhood $K$ of $p$
such that
\eq\label{eq:cinftyt}
&&\sup_{x\in K\cap U}\big(d(x,K\setminus U)\big)^t\vert f(x)\vert
<\infty.
\eneq
We say that $f$ is  temperate of order $t$ at $p$ if, for each $m\in\N$, 
all its derivatives of order $\leq m$ have polynomial growth of order $\leq t+m$ at $p$. We say that $f$ is temperate of order $t$ 
if it is temperate of order $t$  at any point $p\in M\setminus U$.
\end{definition}

For $U\in\Op_\Msa$,
we denote by $\Cin[t]_M(U)$ the subspace
of $\Cinf_M(U)$ consisting of  functions temperate of order $t$ and we denote by $\Cin[t]_\Msal$
the presheaf on $\Msal$ so obtained. 

The next result is clear by Proposition~\ref{pro:MV0}.
\begin{proposition}\label{pro:Cinftfilt}
\bnum
\item
The presheaves  $\Cin[t]_\Msal$ \lp$t\geq0$\rp\,  are sheaves on $\Msal$,
\item
the sheaf $\Cin[0]_\Msal$ is a sheaf of rings,
\item
for $t\geq0$, $\Cin[t]_\Msal$ is a $\Cin[0]_\Msal$-module and there are natural morphisms 
$\Cin[t]_\Msal\tens[{\Cin[0]_\Msal}]\Cin[t']_\Msal\to\Cin[t+t']_\Msal$.
\enum
\end{proposition}
We also introduce the sheaf
\eqn
&&\Cinftt_\Msal\eqdot\indlim[t]\Cin[t]_\Msal.
\eneqn
(Of course, the limit is taken in the category of sheaves on $\Msal$.)
Then, for $0\leq t\leq t'$, there are natural monomorphisms of sheaves on $\Msal$ :
\eq\label{eq:Cinftfilt}
&&  \Cin[0]_\Msal\into\Cin[t]_\Msal\into\Cin[t']_\Msal\into \Cinftt_\Msal \into \Cinft_\Msal.
\eneq
Note that the inclusion $\Cinftt_\Msal \into \Cinft_\Msal$ is strict since
there exists a function $f$ (say on an open subset $U$ of $\R$) with polynomial
growth of order $\leq t$ and such that its derivative does not have polynomial
growth of order $\leq t +1$.

\subsubsection*{Gevrey growth of a given order}

\begin{definition}\label{def:gevrey}
Let $U\in\Op_\Msa$, let  $s\in ]1,+\infty[$ and let $f\in \Cinf_M(U)$. We
say that $f$ has {\it $0$-Gevrey growth of type $(s)$} at $p\in M\setminus U$ if
it satisfies the following condition.  For a local coordinate system
$(x_1,\dots,x_n)$ around $p$, there exists a sufficiently small compact
neighborhood $K$ of $p$ such that
\eq\label{eq:gevrey2}
&&\sup_{x\in K\cap U}\big(\exp(-h\cdot d(x,K\setminus U)^{1-s})\big)\vert f(x)\vert<\infty 
\eneq
for all $h\in ]0,+\infty[$.
We say that $f$ has Gevrey  growth of type $(s)$   if all its derivatives
have $0$-Gevrey growth of type $(s)$ 
at $p$. We say that $f$  has Gevrey growth of type $(s)$ 
if it has such a growth at any point $p\in M\setminus U$.

Similarly, one defines $f$ of Gevrey  growth of type $\{s\}$ when replacing ~\eqref{eq:gevrey2} for all 
 $h\in ]0,+\infty[$ with the same condition for some  $h\in ]0,+\infty[$.
\end{definition}

\begin{definition}\label{def:gevrey2}
For $U\in\Op_\Msa$ and $s\in ]1,+\infty[$, we denote by 
$G_M^{(s)}(U)$ and $G_M^{\{s\}}(U)$ the spaces of functions of Gevrey growth of type $(s)$ and $\{s\}$, respectively.

We denote by $\Ga[s]_\Msal$ and $\Gb[s]_\Msal$ the presheaves on $\Msal$ so obtained. 
\end{definition}
Clearly, the presheaves  $\Ga[s]_\Msal$ and $\Gb[s]_\Msal$ do not depend on the choice of the distance.

\begin{proposition}\label{pro:gevrey}
\bnum
\item
The presheaves $\Ga[s]_\Msal$ and  $\Gb[s]_\Msal$ are sheaves on $\Msal$,
\item
the sheaves $\Ga[s]_\Msal$ and  $\Gb[s]_\Msal$ are  $\Cinft_\Msal$-modules,
\item
the presheaves $\Ga[s]_\Msal$ and  $\Gb[s]_\Msal$ are $\sect$-acyclic,
\item

we have natural monomorphisms of sheaves on $\Msal$ for $1< s< s'$
\eqn
&& \Ga[s]_\Msal\into \Gb[s]_\Msal\into  \Ga[s']_\Msal\into \Gb[s']_\Msal.
\eneqn
\enum
\end{proposition}
\begin{proof}
(i), (ii) and (iv) are obvious and (iii) will follow from (ii) and Proposition~\ref{pro:Csoft} below 
(see Corollary~\ref{co:s-growth}). 
\end{proof}
We set
\eqn
&&\Gc_\Msal\eqdot\indlim[s>1]\Gb[s]_\Msal.
\eneqn
Hence, we have monomorphisms of sheaves on $\Msal$ for $0\leq t$ and $1< s$
\eqn
&&\Cin[0]_\Msal\into\Cin[t]_\Msal\into \Cinftt_\Msal\into \Cinft_\Msal\\
&&\hspace{15ex}\into \Ga[s]_\Msal\into \Gb[s]_\Msal \into 
\Gc_\Msal\into\GG_\Msal\into\Cinf_\Msal.
\eneqn

\begin{definition}\label{def:s-growth3}
If $\shf_\Msal$ is one of the sheaves
$\Cin[t]_\Msal$,  $\Cinftt_\Msal$, $\Ga[s]_\Msal$, $\Gb[s]_\Msal$ or
$\Gc_\Msal$, we set $\shf_\Msa\eqdot \epb{\rhosal}\,\shf$. 
\end{definition}

Let us apply Theorem~\ref{th:Lipbnd} and Corollary~\ref{co:s-growth}. We get
that if $U\in\Op_\Msa$ is weakly Lipschitz and if $\shf_\Msal$ denotes one of the sheaves
above,  then 
\eqn
&&\rsect(U;\shf_\Msa)\simeq \sect(U;\shf_\Msal).
\eneqn

We call $\Cin[t]_\Msa$, $\Cinftt_\Msa$, $\Ga[s]_\Msa$, $\Gb[s]_\Msa$ and $\Gc_\Msa$ the
sheaves on $\Msa$ of  $\Cinf$-functions of growth $t$, strictly temperate
growth,  Gevrey growth of type $(s)$ and  $\{s\}$ and strictly Gevrey growth,
respectively. Recall that on $\Msa$, we also have the sheaf $\Cinft_\Msa$ of $\Cinf$-functions of  temperate
growth, the sheaf $\Dbt_\Msa$ of temperate distributions and the sheaf $\GG_\Msa$ of $\Cinf$-functions of  Gevrey growth.

\subsubsection*{Rings of differential operators}
Let $M$ be a real analytic manifold. 
Recall that $\shhd_M$ denotes the sheaf of finite order analytic differential operators
on $M$ and that we have set in~\eqref{def:DMsa}
\eq\label{def:DMsa2}
&&\shhd_{\Msa}\eqdot\eim{\rhosa}\shhd_M.
\eneq
 Now we set
\eq\label{def:DMsal}
&&\shhd_\Msal\eqdot\oim{\rhosal}\shhd_\Msa.
\eneq
Hence, $\shhd_{\Msa}$ is the sheaf on $\Msa$ associated with the presheaf $U\mapsto\shd_M(\ol U)$
and  $\Msal$ is its direct image on $\Msal$.
We define similarly the sheaves $\shhd_\sht(m)$ of differential operators of order $\leq m$ on the site $\sht=M,\Msa, \Msal$.

\begin{lemma}\label{le:shdmcinft}
There are natural morphisms $\shd_\Msal(m)\tens\Cin[t]_\Msal\to\Cin[t+m]_\Msal$ making 
 $\Cinftt_\Msal$ and  $\Cinft_\Msal$ left $\shd_\Msal$-modules.\\
 The sheaves  $\Ga[s]_\Msal$ and  $\Gb[s]_\Msal$ are naturally 
 left $\shd_\Msal$-modules.
\end{lemma}
\begin{proof}
This follows immediately from Definitions~\ref{def:cinftyt2} and~\ref{def:gevrey}. 
\end{proof}

By using the functor $\epb{\rhosal}$, we  will construct new sheaves (in the
derived sense) 
on $\Msa$ associated with the  sheaves previously constructed on $\Msal$. 

\begin{theorem}\label{th:rightadj2}
\bnum
\item
The functor $\oim{\rhosal}\cl\md[\shhd_\Msa]\to\md[\shhd_\Msal]$ has finite cohomological dimension.
\item
The functor $\roim{\rhosal}\cl\RD(\shhd_\Msa)\to\RD(\shhd_\Msal)$ commutes with small direct sums.
\item
The functor $\roim{\rhosal}$ in {\rm (ii)} admits a right adjoint
 $\epb{\rhosal}\cl\RD(\shhd_\Msal)\to\RD(\shhd_\Msa)$.
\item
The functor $\epb{\rhosal}$ induces a functor $\epb{\rhosal}\cl\RD^+(\shhd_\Msal)\to\RD^+(\shhd_\Msa)$.
\enum
\end{theorem}
\begin{proof}
Consider the quasi-commutative diagram of categories
\eqn
&&\xymatrix{
\md[\shhd_\Msa]\ar[rr]^{\oim{\rhosal}}\ar[d]_{\for}&&\md[\shhd_\Msal]\ar[d]_{\for}\\
\md[\C_\Msa]\ar[rr]^{\oim{\rhosal}}&&\md[\C_\Msal].
}\eneqn
The functor $\for\cl \md[\shhd_\Msa]\to\md[\C_\Msa]$ is exact and sends injective objects to injective objects, 
and similarly with $\Msal$ instead of $\Msa$. It follows that 
 the diagram below commutes:
\eqn
&&\xymatrix{
\RD(\shhd_\Msa)\ar[rr]^{\roim{\rhosal}}\ar[d]_{\for}&&\RD(\shhd_\Msal)\ar[d]_{\for}\\
\RD(\C_\Msa)\ar[rr]^{\roim{\rhosal}}&&\RD(\C_\Msal).
}\eneqn
Moreover, the two functors $\for$ in the last diagram above are conservative. Then

\vspace{0.2ex}\noindent
(i) and (ii) follow from the corresponding result for $\C_\Msa$-modules.

\vspace{0.2ex}\noindent
(iii) and (iv) follow from the Brown representability theorem, (see
Proposition~\ref{pro:rhooplus}).
\end{proof}

\section{A refined cutoff lemma}

Lemma~\ref{le:Ho0}  below will play an important role in this paper and  is an 
immediate corollary of a result of
H\"ormander~\cite[Cor.1.4.11]{Ho83}. Note that H\"ormander's result was 
already used  in~\cite[Prop.~10.2]{KS96} (see Lemma~\ref{le:Ho00} above).

H\"ormander's  result is stated for $M=\R^n$ but we check in Lemma~\ref{le:Ho0-bis}
that it can be extended to an arbitrary manifold. 

\begin{lemma}\label{le:Ho0}
Let $Z_1$ and $Z_2$ be two closed subsets of $M\eqdot\R^n$.
Assume that there exists $C>0$ such that
\eq\label{eq:Ho1}
&& d(x, Z_1\cap Z_2)\leq C(d(x,Z_1)+d(x,Z_2))\mbox{ for any $x\in M$}.
\eneq
Then there exists $\psi\in\Cin[0]_M(M\setminus(Z_1\cap Z_2))$ such that
$\psi=0$ on a neighborhood of $Z_1\setminus Z_2$ and $\psi=1$ on a neighborhood
of $Z_2\setminus Z_1$.
\end{lemma}

\begin{lemma}\label{le:Ho0-bis}
Let $M$ be a manifold.  Let $Z_1$ and $Z_2$ be two closed subsets
of $M$ such that $M\setminus (Z_1\cap Z_2)$ is relatively compact and such
that~\eqref{eq:Ho1} holds for some $C>0$.
Then the conclusion of Lemma~\ref{le:Ho0} holds true.
\end{lemma}
\begin{proof}
We consider an embedding of $M$ in some $\R^N$ and we denote by $d_M$,
$d_{\R^N}$ the distance on $M$ or $\R^N$.  We have a constant $D\geq 1$
such that $D^{-1}d_{\R^N}(x,y) \leq d_M(x,y) \leq D \, d_{\R^N}(x,y)$, for all
$x,y \in M\setminus (Z_1\cap Z_2)$.

Let $x\in \R^N$ and let $x'\in M$ such that $d_{\R^N}(x,x') = d_{\R^N}(x,M)$.
In particular $d_{\R^N}(x,x') \leq d_{\R^N}(x,Z_1)$.
Then we have, assuming $x'\not\in Z_1\cap Z_2$,
\begin{align*}
d_{\R^N}(x, Z_1 \cap Z_2)
& \leq d_{\R^N}(x,x') + D\, d_M(x', Z_1 \cap Z_2) \\ 
& \leq d_{\R^N}(x,x') + DC ( d_M(x',Z_1) + d_M(x',Z_2) ) \\
& \leq d_{\R^N}(x,x') + D^2 C ( d_{\R^N}(x',Z_1) + d_{\R^N}(x',Z_2) ) \\
& \leq (1+ 2D^2C)d_{\R^N}(x,x') + D^2 C (d_{\R^N}(x,Z_1) + d_{\R^N}(x,Z_2)) \\
& \leq (1+3D^2C) ( d_{\R^N}(x,Z_1) + d_{\R^N}(x,Z_2) ) .
\end{align*}
If $x' \in  Z_1\cap Z_2$, then $d_{\R^N}(x, Z_1 \cap Z_2) = d_{\R^N}(x,M)
\leq d_{\R^N}(x,Z_1)$ and the same inequality holds trivially.
Hence we can apply Lemma~\ref{le:Ho0} to $Z_1,Z_2 \subset \R^N$ and obtain
a function $\psi\in\Cin[0]_{\R^N}(\R^N\setminus(Z_1\cap Z_2))$.
Then $\psi|_{M\setminus(Z_1\cap Z_2)}$ belongs to
$\Cin[0]_M(M\setminus(Z_1\cap Z_2))$ and satisfies the required properties.
\end{proof}

\begin{lemma}\label{lem:Ho1}
Let $U_1, U_2 \in \Op_{\Msa}$ and set $U=U_1 \cup U_2$.
We assume that $\{U_1,U_2\}$ is a linear covering of $U$.
Then there exist $U'_i \subset U_i$, $i=1,2$,
and $\psi\in\Cin[0]_M(U)$ such that
\bnum
\item $\{U'_i, U_1\cap U_2\}$ is a linear covering of $U_i$,
\item $\psi|_{U'_1} =0$ and $\psi|_{U'_2} =1$.
\enum
\end{lemma}
\begin{proof}
We choose $U'_i \subset U_i$, $i=1,2$, as in Lemma~\ref{lem:prelim_Horm2}
and we set $Z_i = (M\setminus U) \cup \ol{U'_i}$. Then the result follows
from Lemmas~\ref{lem:prelim_Horm2} and~\ref{le:Ho0-bis}.
\end{proof}

\begin{proposition}\label{pro:Csoft} 
Let $\shf$ be a sheaf of  $\Cin[0]_\Msal$-modules on $\Msal$.
Then $\shf$ is $\sect$-acyclic. 
\end{proposition}
\begin{proof}
By Proposition~\ref{pro:MV1}, it is enough to prove that 
for any $\{U_1,U_2\}$ which is a covering of $U_1\cup U_2$,  the sequence
$0\to \shf(U_1\cup U_2)\to \shf(U_1)\oplus \shf(U_2)\to \shf(U_1\cap U_2)\to 0$
is exact.  
This follows from Lemma~\ref{lem:Ho1}, similarly as in the proof
of~\cite[Prop.~10.2]{KS96}.  The only non trivial fact is the surjectivity at
the last term, which we check now.

We choose $U'_i \subset U_i$, $i=1,2$,
and $\psi\in\Cin[0]_M(U)$ as in Lemma~\ref{lem:Ho1}.
Let $s\in\sect(U_1\cap U_2;\shf)$.
Since $\{U'_i, U_1\cap U_2\}$ is a linear covering of $U_i$, $i=1,2$,
we can define $s_1\in\sect(U_1;\shf)$ and $s_2\in\sect(U_2;\shf)$ by
\eqn
& s_1\vert_{U_1\cap U_2} = \psi\cdot s,\; s_1|_{U'_1} = 0
\quad\text{and}\quad
s_2\vert_{U_1\cap U_2}=(1-\psi)\cdot s,\; s_2|_{U'_2} = 0.
\eneqn
Then $s_1|_{U_1\cap U_2} + s_2|_{U_1\cap U_2} = s$, as required. 
\end{proof}

\begin{corollary}\label{co:s-growth}
The sheaves $\Cinftt_\Msal$,  $\Cinft_\Msal$, $\Dbt_\Msal$,  
$\Cin[t]_\Msal$ \lp $t\in\R_{\geq0}$\rp, $\Ga[s]_\Msal$ and $\Gb[s]_\Msal$  \lp $s>1$\rp,
$\Gc_\Msal$ and $\GG_\Msal$ are $\sect$-acyclic.
\end{corollary}
Let $\shf_\Msal$ denote one of the sheaves appearing in
Corollary~\ref{co:s-growth} and let
$\shf_\Msa\eqdot\epb{\rhosal}\shf_\Msal\in\RD^+(\shd_\Msa)$. Then, if $U$ is
weakly Lipschitz, $\rsect(U;\shf_\Msa)$ is concentrated in degree $0$ and
coincides with $\shf_\Msal(U)$.

\section{A comparison result}

In the next lemma, we  set $M\eqdot \R^n$ and we denote by $dx$ the Lebesgue measure. 
As usual, for $\alpha\in\N^n$ we denote by $D_x^\alpha$ the differential operator 
 $(\partial/\partial_{x_1})^{\alpha_1}\dots(\partial/\partial_{x_n})^{\alpha_n}$ and we denote by 
 $\Delta=\sum_{i=1}^n\partial^2/\partial x_i^2$ the Laplace operator on $M$.

In all this section, we consider an open set $U\in\Op_\Msa$. We set for short
\eqn
&&d(x)=d(x,M\setminus U).
\eneqn
For a locally integrable function $\phi$ on $U$ and $s\in\R_{\geq0}$, we set 
\eq\label{eq:linftynorm}
&&\vvert\phi\vvert_\infty=\sup_{x\in U}\vert \phi(x)\vert,\quad
\vvert\phi\vvert^s_\infty=\vvert d(x)^s\phi(x)\vvert_\infty.
\eneq

\begin{proposition}\label{pro:harmon}
There exists a constant $C_\alpha$ such that for any locally integrable function $\phi$ on $U$, 
one has the estimate for $s\geq0$:
\eq\label{eq:linftynorm1}
&&\vvert D_x^\alpha\phi\vvert^{s+\vert\alpha\vert}_\infty\leq 
C_\alpha\bl\vvert \phi\vvert^s_\infty+\vvert \Delta D_x^\alpha\phi\vvert^{s+\vert\alpha\vert+2}_\infty\br.\label{eq:estim1}
\eneq
 \end{proposition}
\begin{proof}
We shall adapt the proof of~\cite[Prop.~10.1]{KS96}.

\vspace{0.2ex}\noindent
(i) Let us take a distribution $K(x)$ and a $\Cinf$ function
$R(x)$ such that
\eqn
&&\delta(x)=\Delta K(x)+R(x)
\eneqn
(where $\delta(x)$ is the Dirac distribution at the origin)
and the support of $K(x)$ and
the support of $R(x)$ are contained in $\{x\in M; |x|\le 1\}$.
Then $K(x)$ is integrable.
For $c>0$ and for a function $\psi$ set:
\eqn
&&\psi_c(x)=\psi(c^{-1}x), \,\tw K_c=c^{2-n}K_c\mbox{ and }\tw R_c=c^{-n}R_c.
\eneqn
Then we have again
\eqn
&&\delta(x)=\Delta \tw K_c(x)+\tw R_c(x)\,.
\eneqn
Hence we have for any distribution $\psi$
\eq\label{eq:estim0}
&&\psi(x)=\int \tw K_c(x-y)(\Delta \psi)(y)dy+\int \tw R_c(x-y)\psi(y)dy\,.
\eneq
Now for $x\in U$, set $c(x)=d(x)/2$. We set
\eqn
&&A_\alpha(x)=\vert\int \tw K_{c(x)}(x-y)(\Delta D_y^\alpha \phi)(y)dy\vert,\\
&&B_\alpha(x)=\vert\int \tw R_{c(x)}(x-y)D_y^\alpha\phi(y)dy\vert.
\eneqn
Since $\int\vert \tw K_{c(x)}(x-y)\vert dy=c(x)^2\int\vert K(\frac{x}{c(x)}-y)\vert dy$, we get
\eqn
&&\int\vert\tw K_{c(x)}(x-y)\vert dy\leq C_1 d(x)^2
\eneqn
for some constant $C_1$.

\vspace{0.3ex}\noindent
(ii) We have
\eqn
A_\alpha(x)&\leq&\Big(\sup_{|x-y|\le c(x)}|(D_y^\alpha\Delta \phi)(y)|\Big)
\int\vert \tw K_{c(x)}(x-y)|dy\\
&\le& C_1\,\Big(\sup_{|x-y|\le c(x)}|(D_y^\alpha\Delta \phi)(y)|\Big) \cdot d(x)^{2}.
\eneqn
Hence, 
\eq
d(x)^{s+\vert\alpha\vert}A_\alpha(x)
&\le& C_1\,\Big(\sup_{|x-y|\le c(x)}|(D_y^\alpha\Delta \phi)(y)|\Big) \cdot d(x)^{s+\vert\alpha\vert+2}\nonumber\\
\label{eq:estim2}
&\le& 2^{s+\vert\alpha\vert+2} C_1\,\Big(\sup_{|x-y|\le c(x)}|d(y)^{s+\vert\alpha\vert+2}(D_y^\alpha\Delta \phi)(y)|\Big) \\
&\le& 2^{s+\vert\alpha\vert+2} C_1\vvert \Delta D_x^\alpha\phi\vvert^{s+\vert\alpha\vert+2}_\infty.\nonumber
\eneq
Here we have used the fact that on the ball centered at $x$ and radius $c(x)$, we have $d(x)\leq 2 d(y)$. 

\vspace{0.2ex}\noindent
(iii) Since $\tw R_c(x-y)$ is supported by the ball of center $x$ and radius $c(x)$, we have
\eqn
 B_\alpha(x)&=&|\int_{B(x,c(x))} D_y^\alpha \tw R_{c(x)}(x-y)\phi(y)dy|\\
&=&c(x)^{-\vert\alpha\vert}|\int_{B(x,c(x))} c(x)^{-n}(D_y^\alpha R)_{c(x)}(x-y)\phi(y)dy|\\
&\le&c(x)^{-\vert\alpha\vert}\sup_{\vert x-y\vert\le c(x)}\vert \phi(y)\vert\cdot \int\vert D_y^\alpha R(y) \vert dy.
\eneqn
Here we have used the fact that 
$D_y^\alpha R_{c(x)}(y)=c(x)^{-|\alpha|}(D_y^\alpha R)_{c(x)}(y)$.

As in (ii), we deduce that
\eq\label{eq:estim3}
\nonumber d(x)^{s+\vert\alpha\vert}B_\alpha(x)&\leq &C_2 \sup_{\vert x-y\vert\le c(x)}\vert d(y)^s\phi(y)\vert\\
&\leq&C_2\vvert \phi\vvert^{s}_\infty.
\eneq
for some constant $C_2$. 

\vspace{0.2ex}\noindent
(iv) By choosing $\psi=D^\alpha_x\phi$ in~\eqref{eq:estim0} 
the estimate~\eqref{eq:estim1}  follows from~\eqref{eq:estim2} and~\eqref{eq:estim3}.
\end{proof}

\section{Sheaves on complex manifolds}

Let $X$ be a complex manifold of complex dimension $d_X$ and denote by $X_\R$
the real analytic 
underlying manifold. Denote by $\ol X$ the complex manifold conjugate to
$X$. (The holomorphic 
functions on $\ol X$ are the anti-holomorphic functions on $X$.)
Then $X\times \ol X$ is a complexification of $X_\R$ and $\sho_{\ol X}$ is a 
$\shhd_{X\times\ol X}$-module which plays the role of the Dolbeault complex. In the sequel, when there is no risk of confusion, we write for short $X$ instead of $X_\R$.

\begin{notation}
In the sequel, we will often have to consider the composition $\roim{\rhosal}\circ\eim{\rhosa}$. For convenience, we introduce a notation. We set
\eq\label{not:roeim}
&&\oeim{\rhosl}\eqdot\oim{\rhosal}\circ\eim{\rhosa}.
\eneq
\end{notation}

\subsubsection*{Sheaves on complex manifolds}
By applying the Dolbeault functor $\rhom[\shhd_\olXsal](\oeim{\rhosl}\sho_{\ol X},\scbul)$ to one of the sheaves 
$$
\setlength{\arraycolsep}{1.8mm}
\begin{array}{lllllll}
\Cinftt_\Xsal,&  \Cinft_\Xsal,& \Ga[s]_\Xsal,& \Gb[s]_\Xsal,&  \Gc_\Xsal,& \GG_\Xsal, &\Cinf_\Xsal,  \end{array}
$$
we obtain respectively the sheaves
$$
\setlength{\arraycolsep}{3.2mm}
\begin{array}{lllllll}
\Ott_\Xsal,& \Ot_\Xsal, & \OGa[s]_\Xsal, &\OGb[s]_\Xsal,& \OGc_\Xsal,&\OG_\Xsal,& \sho_\Xsal.
\end{array}
$$
All these objects belong to $\RD^+(\shd_\Xsal)$.
Then we can apply the functor $\epb{\rhosal}$ and we obtain the sheaves 
$$
\setlength{\arraycolsep}{3.2mm}
\begin{array}{lllllll}
\Ott_\Xsa, &\Ot_\Xsa, & \OGa[s]_\Xsa,& \OGb[s]_\Xsa,& \OGc_\Xsa,&\OG_\Xsa,& \sho_\Xsa. 
\end{array}
$$
Note that the functor $\epb{\rhosal}$ commutes with the Dolbeault functor. More precisely: 

\begin{lemma}\label{le:epbrhodolbeault}
Let $\shc$ be an object of $\RD^+(\shd_\XRsal)$. There is a natural isomorphism
\eq\label{eq:epbrhodolbeault}
&&\epb{\rhosal}\rhom[\shhd_\olXsal](\oeim{\rhosl}\sho_{\ol X}, \shc_\Xsal)\simeq
\rhom[\shhd_\olXsa](\eim{\rhosa}\sho_{\ol X}, \epb{\rhosal}\shc_\Xsal).
\eneq
\end{lemma}
\begin{proof}
This follows from the fact that the  $\shhd_X$-module $\sho_{\ol X}$ admits a global locally finite free resolution.
\end{proof}

Recall the natural isomorphism~\cite[Th.~10.5]{KS96}
\eqn
\Ot_\Xsa&\isoto&\rhom[\shhd_\olXsa](\eim{\rhosa}\sho_{\ol X}, \Dbt_\Xsa).
\eneqn

\begin{proposition}\label{pro:ot=ott2}
The natural morphism
\eqn
&&\Ott_\Xsal\to\Ot_\Xsal
\eneqn
is an  isomorphism in $\RD^+(\shhd_\Xsal)$.
\end{proposition}
\begin{proof}
Let $U\in\Op_\Msa$. Consider the diagram (in which $M=\R^{2n}$)
\eqn
&&\xymatrix{
0\ar[r]&\sect(U;\Cinftt_\Msal)\ar[r]^-{\Delta}\ar[d]&\sect(U;\Cinftt_\Msal)\ar[r]\ar[d]&0\\
0\ar[r]&\sect(U;\Cinft_\Msal)\ar[r]^-{\Delta}&\sect(U;\Cinft_\Msal)\ar[r]&0.
}\eneqn
As in the proof of~\cite[Th.~10.5]{KS96}, we are reduced to prove that the vertical arrows
induce a qis from the top line to the bottom line. We shall apply Proposition~\ref{pro:harmon}.

\vspace{0.3ex}\noindent
(i) Let $\phi\in \sect(U;\Cinft_\Msal)$ with $\Delta\phi=0$. 
There exists some $s\geq0$ such that $\vvert d(x)^s\phi\vvert_\infty<\infty$.
Then $\vvert d(x)^{s+\vert\alpha\vert}D^\alpha_x\phi\vvert_\infty<\infty$ by~\eqref{eq:estim1}.

\vspace{0.3ex}\noindent
(ii) It follows from~\cite[Prop.10.1]{KS96} that the arrow in the bottom is surjective. Now let 
$\psi\in\sect(U;\Cinftt_\Msal)$. There exists $\phi\in\sect(U;\Cinft_\Msal)$ with $\Delta\phi=\psi$. Then it follows 
from~\eqref{eq:estim1} that $\phi\in\sect(U;\Cinftt_\Msal)$.
\end{proof}

\begin{remark}\label{pro:ot=ott}
It is natural to expect that the morphism
\eqn
&&\OGc_\Xsal\to\OG_\Xsal
\eneqn
is an  isomorphism in $\RD^+(\shhd_\Xsal)$. The proof of Proposition~\ref{pro:ot=ott2} can be adapted with the exception that one does not know if the map $\Delta\cl  \GG_\Xsa(U) \to \GG_\Xsa(U)$ is surjective. 
\end{remark}

\subsubsection*{Solutions of holonomic $\shd$-modules}
The next result is a reformulation of a theorem of Kashiwara~\cite{Ka84}.

\begin{theorem}
Let $\shm$ be a regular holonomic $\shd_X$-module. Then the natural morphism
\eqn
&&\rhom[\shd_\Xsa](\eim{\rhosa}\shm, \Ot_\Xsa)\to \rhom[\shd_\Xsa](\eim{\rhosa}\shm, \sho_\Xsa)
\eneqn
is an isomorphism.
\end{theorem}
The next result was a conjecture of~\cite{KS03} and has recently been proved by Morando~\cite{Mr13}
(see also~\cite{KS15} for a rather different proof)
by using  the deep results of Mochizuki~\cite{Mo09} (completed by those of Kedlaya~\cite{Ke10, Ke11} for the analytic case). 

\begin{theorem}\label{conj:construct1}
Let $\shm$ be a  holonomic $\shd_X$-module. Then for any $G\in\Derb_\Rc(\C_X)$, 
\eqn
&&\opb{\rhosa}\rhom(G,\rhom[\shd_\Xsa](\eim{\rhosa}\shm, \Ot_\Xsa))\in \Derb_\Rc(\C_X).
\eneqn
\end{theorem}
It is natural to  conjecture that this theorem still holds when replacing the sheaf  $\Ot_\Xsa$ with one of the sheaves $\OGa[s]_\Xsa$ or 
 $\OGb[s]_\Xsa$.

In~\cite{KS03}, the object $\hom[\shd_\Xsa](\eim{\rhosa}\shm, \Ot_\Xsa)$ is explicitly calculated when $X=\C$ and, denoting by $t$ a holomorphic coordinate on $X$, $\shm$ is associated with the operator $t^2\partial_t+1$, that is, 
$\shm=\shd_X\exp(1/t)$. 

It is well-known, after~\cite{Ra78} (see also~\cite{Ko73a}), that the 
holomorphic solutions of an ordinary linear differential equation singular at the origin have Gevrey growth, the 
growth being related to the slopes of the Newton polygon.

\begin{conjecture}
Let $\shm$ be a holonomic $\shd_X$-module. Then
the natural morphism
\eqn
&&\rhom[\shd_\Xsa](\eim{\rhosa}\shm, \OG_\Xsa)\to \rhom[\shd_\Xsa](\eim{\rhosa}\shm,\sho_\Xsa)
\eneqn
is an isomorphism, or, equivalently, 
\eqn
&&\rhom[\shd_\Xsa](\eim{\rhosa}\shm, \OG_\Xsa)\isoto\roim{\rhosa}\rhom[\shd_X](\shm,\sho_X).
\eneqn
Moreover,   there exists a discrete set $Z\subset \R_{>1}$ such that
 the morphisms $\rhom[\shd_\Xsa](\shm,\OGa[s]_\Xsa)\to \rhom[\shd_\Xsa](\shm,\OGa[t]_\Xsa)$
are isomorphisms for $s\leq t$ in the same components of $\R_{>1}\setminus Z$. 
\end{conjecture}

\chapter{Filtrations}\label{ShvSa5}

\section{Derived categories of filtered objects}
In this section, we shall recall results of~\cite{Sn99} completed  in~\cite{SSn13}.

\subsubsection*{Complements on abelian categories}
In this subsection we state and prove some elementary results (some of them being well-known)  on abelian
and derived categories that we shall need.

Let $\shc$ be an abelian category and let $\mon$ be a small category. As usual, one denotes by 
$\Fct(\mon,\shc)$ the abelian category of functors from $\mon$ to $\shc$. Recall that 
the kernel of a morphism $u\cl X\to Y$ is the functor $\lambda\mapsto\ker u(\lambda)$
and similarly with the cokernel or more generally with limits and colimits. 

\begin{lemma}\label{le:comp3}
Assume that $\shc$ is a Grothendieck category. Then
\banum
\item
the category $\Fct(\mon,\shc)$ is a Grothendieck category,
\item
if $F\in \Fct(\mon,\shc)$ is injective, then for $\lambda\in\mon$,  $F(\lambda)$ is injective in $\shc$.
\eanum
\end{lemma}

\begin{proof}
The category $\Fct(\mon,\shc)$ is equivalent to the category
$\Psh(\mon^\rop,\shc)$ of preshaves on $\mon^\rop$ with values in $\shc$.
It follows that, for any given $\lambda\in \Lambda$, the functor
$\Fct(\mon,\shc) \to \shc$, $F \mapsto F(\lambda)$ has a left adjoint.
We can define it as follows (see {\em e.g.}~\cite[Not. 17.6.13]{KS06}).  For
$G\in \shc$ we define $G_\lambda \in \Fct(\mon,\shc)$ by
$$
G_\lambda(\mu) =\bigoplus_{\Hom[\mon](\lambda,\mu)}G.
$$
Then we can check directly that
\eq
\label{eq:Glambda1}
&&\mbox{the functor $\shc\ni G\mapsto G_\lambda\in\Fct(\mon,\shc)$ is exact,}\\
\label{eq:Glambda2}
&&\Hom[\Fct(\mon,\shc)](G_\lambda,F) \simeq  \Hom[\shc](G,F(\lambda))
\quad\text{for any $F\in \Fct(\mon,\shc)$.}
\eneq

\vspace{0.3ex}\noindent
(a) Applying {\em e.g.} Th.~17.4.9 of loc.\ cit., it remains to show that
$\Fct(\mon,\shc)$ admits a small system of generators. Let $G$ be a generator of
$\shc$.  It follows from~\eqref{eq:Glambda2} that the family
$\{G_\lambda\}_{\lambda\in\mon}$ is a small system of generators in
$\Fct(\mon,\shc)$.

\vspace{0.3ex}\noindent
(b) Follows from~\eqref{eq:Glambda2} and~\eqref{eq:Glambda1}.
\end{proof}

We consider two abelian categories $\shc$
and $\shc'$ and a left exact functor
$\rho\cl\shc\to\shc'$.  The functor $\rho$ induces a functor 
\eq\label{eq:fctoimrho}
&&\tw\rho\cl\Fct(\mon,\shc)\to\Fct(\mon,\shc').
\eneq

\begin{lemma}\label{le:comp2}
Assume that $\shc$ is a Grothendieck category.
\banum
\item
The functor $\tw\rho$ is left exact. 
\item
Let $I$ be a small category and assume that $\rho$ commutes with colimits
indexed by $I$. Then the functor $\tw\rho$ in~\eqref{eq:fctoimrho}
commutes with colimits indexed by $I$.
\item
Assume that  $\rho$ has cohomological
dimension $\leq d$, that is, $R^j\rho=0$ for $j> d$. Then 
 $\tw\rho$ has cohomological dimension $\leq d$.
\item
Assume that $\rho$ commutes with small direct sums and that small direct sums of injective objects in $\shc$ are acyclic for the functor $\rho$. Then small direct sums of injective objects in $\Fct(\mon,\shc)$ are acyclic for
the functor $\tw\rho$.
\eanum
\end{lemma}
\begin{proof}
(a) is obvious.

\vspace{0.3ex}\noindent
(b) follows from the equivalence $\Fct(I,\Fct(\mon,\shc))\simeq\Fct(\mon,\Fct(I,\shc))$ and similarly with 
$\shc'$.

\vspace{0.3ex}\noindent
(c) By Lemma~\ref{le:comp3}~(a), the category $\Fct(\mon,\shc)$ admits enough injectives.  
Let $F\in \Fct(\mon,\shc)$ and let $F\to F^\scbul$ be an injective resolution of $F$, that is, 
$F^\scbul$ is a complex in degrees $\geq0$ of injective objects and $F\to F^\scbul$ is a qis. 
By Lemma~\ref{le:comp3}~(b), for $\lambda\in\mon$, 
$ F^\scbul(\lambda)$ is an injective resolution of $F(\lambda)$
and by the hypothesis, $H^j(\rho(F^\scbul(\lambda)))\simeq0$ for $j>d$ and $\lambda\in\mon$. This implies
that $R^j\rho(F)\simeq H^j(\rho(F^\scbul))$ is $0$ for $j>d$.

\vspace{0.3ex}\noindent
(d) For a given $\lambda\in\mon$ we denote by $i_\lambda^\shc$ the functor
$\Fct(\mon,\shc) \to \shc$, $F\mapsto F(\lambda)$. 
Then $i_\lambda^\shc$  is exact and, by Lemma~\ref{le:comp3}~(b), we have
$i_\lambda^{\shc'} \circ R \tw\rho \simeq R\rho \circ i_\lambda^\shc$.
Let $F\in \Fct(\mon,\shc)$ be a small direct sum of injective objects.
Since $i_\lambda^\shc$ commutes with direct sums, it follows from
Lemma~\ref{le:comp3}~(b) again that $i_\lambda^\shc(F)$ is a small direct sum
of injective objects in $\shc$. By the hypothesis we obtain
$R^j \rho \circ i_\lambda^\shc(F) \simeq 0$, for all $j>0$.  Hence
$i_\lambda^{\shc'} \circ R^j \tw\rho(F) \simeq 0$, for all $j>0$.  Since this
holds for all $\lambda\in\mon$ we deduce $R^j \tw\rho(F) \simeq 0$, for all
$j>0$, as required. 
\end{proof}

\subsubsection*{Abelian tensor categories}

Recall (see {\em e.g.}~\cite[Ch.~5]{KS06})   that a tensor Grothendieck category  $\shc$ is a Grothendieck category endowed with a biadditive functor 
$\tens\cl\shc\times\shc\to\shc$ satisfying functorial associativity isomorphisms. 
We do not recall here what is a tensor category with unit, a ring object $A$ in $\shc$,  a ring object with unit and an $A$-module $M$. In the sequel, all tensor categories will be with unit and a
ring object means a ring object with unit.

We shall consider 
\eq\label{hyp:mon2}
&&\left\{\parbox{60ex}{
a Grothendieck tensor category $\shc$ \lp with unit\rp\, in which small inductive limits   commute with $\tens$.
}\right.\eneq

\begin{lemma}\label{le:comp9}
Let  $\shc$ be as in~\eqref{hyp:mon2} and let $A$ be a ring object \lp with unit\rp\, in $\shc$. Then
\banum
\item
The category $\md[A]$ is a Grothendieck category,
\item 
the forgetful functor $\for\cl\md[A]\to\shc$ is exact and conservative,
\item 
the natural functor $\tw\for\cl \RD(A)\to\RD(\shc)$ is conservative.
\eanum
\end{lemma}
\begin{proof}
(a) and (b) are proved in~\cite[Prop.~4.4]{SSn13}.

\vspace{0.3ex}\noindent
(c) Since $\RD(A)$ and $\RD(\shc)$ are triangulated, it is enough to check that 
if $X\in\RD(A)$ verifies $\tw\for(X)\simeq0$, then $X\simeq0$. Let $X$ be such an object and let $j\in\Z$. 
Since $\for$ is exact, $\for H^j(X)\simeq H^j(\tw\for(X))\simeq0$. Since $\for$ is conservative,
we get $H^j(X)\simeq 0$. 
\end{proof}

\subsubsection*{Derived categories of filtered objects}

We shall consider  
\eq\label{hyp:mon1}
&&\left\{\parbox{60ex}{
a filtrant preordered additive monoid $\mon$ (viewed as a tensor category with unit),\\
 a category $\shc$ as in~\eqref{hyp:mon2}.
}\right.\eneq

Denote by 
$\Fct(\mon,\shhc)$ the abelian category of functors from $\mon$ to $\shhc$. It is naturally endowed with a structure of a tensor category with unit by setting for $M_1,M_2\in\Fct(\mon,\shc)$,
\eqn
&&(M_1\tens M_2)(\lambda) =\indlim[\lambda_1+\lambda_2\leq\lambda] M_1(\lambda_1)\tens M_2(\lambda_2).
\eneqn
A $\mon$-ring $A$ of $\shc$ is  a ring with unit of the tensor category $\Fct(\mon,\shc)$ and we denote by 
$\md[A]$ the abelian category of $A$-modules.

We denote by $\Fil_\mon\shc$ the full subcategory of $\Fct(\mon,\shhc)$ consisting of functors $M$ such that for each morphism 
$\lambda\to\lambda'$ in $\mon$, the morphism $M(\lambda)\to M(\lambda')$ is a monomorphism. This is a quasi-abelian category. Let 
\eqn
&&\iota\cl \Fil_\mon\shc\to \Fct(\mon,\shhc)
\eneqn
 denote the inclusion functor. This functor admits a left adjoint $\kappa$ 
and the category $\Fil_\mon\shc$  is again a tensor category by setting
\eqn
&&M_1\tens_F M_2 =\kappa(\iota(M_1)\tens\iota(M_2)).
\eneqn
A ring object in the tensor category $\Fil_\mon\shc$ will be called a $\mon$-filtered ring in $\shc$ and usually denoted $FA$. 
An $FA$-module $FM$ is then simply a module over $FA$ in $\Fil_\mon\shc$ and we denote by 
$\md[FA]$ the quasi-abelian category of $FA$-modules. 

It follows from Lemmas~\ref{le:comp3} and~\ref{le:comp9} that $\md[\iota FA]$ is a Grothendieck category.

\begin{notation}
In the sequel, for a ring object $B$ in a tensor category, we shall write $\RD^*(B)$ instead of $\RD^*(\md[B])$, $*=+,-,\rb,\ub$.
\end{notation}

The next theorem is due to~\cite{SSn13} and generalizes previous results of~\cite{Sn99}.
\begin{theorem} \label{th:eqv1}
Assume~\eqref{hyp:mon1}.
Let $FA$ be a $\mon$-filtered ring in $\shc$.
Then the category $\md[FA]$ is quasi-abelian, the functor $\iota\cl\md[FA]\to\md[\iota FA]$ is strictly exact and 
induces  an equivalence of categories for $*=\ub,+,-,\rb$:
\eq\label{eq:eqvfilt1}
&&
\iota\cl\RD^*(FA)\isoto\RD^*(\iota FA).
\eneq
\end{theorem}

\begin{notation}\label{not:colimFA}
Let $\mon$ and $\shc$ be as in~\eqref{hyp:mon1}. The functor 
$\sindlim\cl\Fct(\mon,\shc)\to\shc$ is exact. Let $FA$ be a $\mon$-filtered ring in $F_\mon\shc)$ and set 
\eq\label{eq:colimFA1}
&&A\eqdot\indlim[\lambda]A(\lambda).
\eneq
(For short, we write $A(\lambda)$ instead of $FA(\lambda)$.)
 The functor $\sindlim$  induces an exact functor 
\eq\label{eq:colimFA2}
&&\sindlim\cl\md[FA]\to\md[A],
\eneq
thus, using Theorem~\ref{th:eqv1}, for $*=\ub,+,-,\rb$, a functor 
\eq\label{eq:colimFA3}
&&\sindlim\cl\RD^*(FA)\to\RD^*(A).
\eneq
Since one often considers $FA$ as a filtration on the ring $A$, we shall denote by $\for$ (forgetful) the functor
$\sindlim$:
\eq\label{eq:colimFA4}
&&\for\cl\RD^*(FA)\to\RD^*(A),\quad \for\eqdot\sindlim.
\eneq
\end{notation}

\subsubsection*{Complements on filtered objects}
\begin{lemma}\label{le:comp7}
Let $\mon$ and $\shc$ be as in~\eqref{hyp:mon1}  and let  $\shc'$ be another
Grothendi\-eck tensor category satisfying the same hypotheses as $\shc$.  Let
$FB$ be a $\mon$-filtered ring in $\shc'$.
\banum
\item
Let $\sigma\cl\shc'\to\shc$ be an exact functor of tensor categories
\lp see Definition~4.2.2 in~{\rm \cite{KS06}}\rp.  Denote by
$\tw\sigma\cl \Fct(\mon,\shc')\to\Fct(\mon,\shc)$ the natural functor
associated with $\sigma$.  Then
\bnum
\item
$FA\eqdot\tw\sigma(FB)$ has a natural structure of a $\mon$-filtered ring with
values in $\shc$,
\item
the functor $\tw\sigma$ induces an exact functor
$\tw\sigma_\mon\cl\md[\iota FB]\to\md[\iota FA]$
hence a functor $\sigma_\mon\cl\md[FB]\to\md[FA]$.
\enum
\item
Assume moreover that the functor $\sigma$ has a right adjoint $\rho$ which is
fully faithful \lp hence $\rho$ is left exact and $\sigma\rho \simeq \id_{\shc}$\rp.
Denote by $\tw\rho\cl \Fct(\mon,\shc)\to\Fct(\mon,\shc')$ the natural functor
associated with $\rho$.
Then
\bnum
\item
$\tw\rho$ is fully faithful and right adjoint to $\tw\sigma$,
\item
$\tw\rho$ induces a left exact fully faithful functor
$\tw\rho_\mon\cl\md[\iota FA]\to\md[\iota FB]$ right adjoint to
$\tw\sigma_\mon$ and a fully faithful functor $\rho_\mon\cl\md[FA]\to\md[FB]$
right adjoint to $\sigma_\mon$.
\enum
\item
The diagram below, in which the horizontal arrows are the forgetful functors,
is commutative when composing horizontal and down vertical arrows, or when
composing horizontal and up vertical arrows
$$
\xymatrix{
\md[FA] \ar@<.5ex>[d]^-{\rho_\mon} \ar[r] & \md[\iota FA] \ar@<.5ex>[d]^{\tw\rho_\mon} \ar[r]
& \Fct(\mon,\shhc) \ar@<.5ex>[d]^{\tw\rho}  \\
\md[FB] \ar@<.5ex>[u]^{\sigma_\mon} \ar[r] & \md[\iota FB] \ar@<.5ex>[u]^{\tw\sigma_\mon} \ar[r] 
        &  \Fct(\mon,\shc') \ar@<.5ex>[u]^{\tw\sigma}.
}
$$
\eanum
\end{lemma}
\begin{proof}
(a) We first recall that a $\mon$-ring $A$ of a tensor category $\shc$ is the
data of $A(\lambda) \in \shc$, for each $\lambda\in\mon$,  morphisms
$\mu_A^{\lambda,\lambda'} \cl A(\lambda) \tens A(\lambda') \to
A(\lambda+\lambda')$,
for all $\lambda,\lambda' \in \Lambda$, and
$\varepsilon_A \cl \mathbf{1}_\shc \to A(0)$, where $\mathbf{1}_\shc$ is the
unit of $\shc$ and $0$ the unit of $\Lambda$.  These morphisms satisfy three
commutative diagrams (which we do not recall here) expressing the associativity
of $\mu_A$ and the fact that $\varepsilon_A$ is a unit.  Similarly a module $M$
over $A$ is the data of $M(\lambda) \in \shc$, for each $\lambda\in\mon$, and
morphisms
$\mu_M^{\lambda,\lambda'} \cl A(\lambda) \tens M(\lambda') \to
M(\lambda+\lambda')$,
for all $\lambda,\lambda' \in \Lambda$, satisfying two commutative diagrams
left to the reader.

Let us go back to the situation of the lemma.  For a $\mon$-filtered ring $FB$ of $\shc'$
and an $FB$-module $N$,  setting
$FA = \tw\sigma(FB)$,  the morphisms $\mu_N^{\lambda,\lambda'}$ induce
$$
\mu_{ \tw\sigma(N)}^{\lambda,\lambda'} \cl
A(\lambda) \tens \sigma(N(\lambda'))
 \simeq \sigma(B(\lambda) \tens N(\lambda')) 
 \to \sigma(N(\lambda+\lambda')) .
$$
For $N=\iota FB$ we obtain $\mu_A^{\lambda,\lambda'}$. We define
$\varepsilon_A = \sigma(\varepsilon_B)$. We leave to the reader the
verification that $\varepsilon_A$, $\mu_A^{\cdot,\cdot}$ and
$\mu_{ \tw\sigma(N)}^{\cdot,\cdot}$ satisfy the required commutative diagrams.
This defines the functor $\tw\sigma_\mon$. We see easily that $\tw\sigma_\mon$
is exact. Since $FB$ is $\mon$-filtered, the exactness of $\sigma$ implies
that $FA$ is $\mon$-filtered and that $\tw\sigma_\mon$ induces the functor
$\sigma_\mon$ of the lemma.

\spa
(b) The statement~(i) is straightforward. Let us define $\tw\rho_\mon$.
For a $\iota FA$-module $M$ the data of
$$
\mu_{M}^{\lambda,\lambda'} \cl \sigma(B(\lambda) \tens \rho(M(\lambda')))
\simeq A(\lambda) \tens M(\lambda')  \to M(\lambda+\lambda')
$$
give by adjunction
$\mu_{ \tw\rho(M)}^{\lambda,\lambda'} \cl B(\lambda) \tens \rho(M(\lambda'))
\to \rho(M(\lambda+\lambda'))$
and define a structure of $\iota FB$-module on $\rho(M)$.  Since $\rho$ is
left exact $\tw\rho_\mon$ induces $\rho_\mon$.  The adjunction properties are
clear, as well as $\tw\sigma_\mon \tw\rho_\mon \simeq \id$ and
$\sigma_\mon \rho_\mon \simeq \id$.  Hence $\tw\rho_\mon$ and $\rho_\mon$ are
fully faithful.

\spa
(c) is clear.
\end{proof}

\begin{theorem}\label{th:abstractBrown}
{\rm(1)} We make the assumptions of   {\rm Lemma~\ref{le:comp7}~(a)-(b)}
 and assume moreover that
\bnum 
\item
$\rho$ has  cohomological dimension $\leq d$,
\item
for any $M\in \md[\iota FA]$, there exists a monomorphism $M \to I$ in
$\md[\iota FA]$ such that $I(\lambda)$ is $\rho$-acyclic,
for all $\lambda\in \Lambda$.
\enum
Then the derived functor $R\rho_\mon\cl\RD^*(FA)\to\RD^*(FB)$
\lp $*=\ub,+$\rp\, exists.  It is fully faithful and admits a left adjoint
$\opb\rho_\mon\cl \RD^*(FB)\to\RD^*(FA)$ \lp$*=\ub,+$\rp.

\spaa
{\rm(2)} Assume moreover that 
\bnum
\item[{\rm(iii)}]
$\rho$  commutes with small direct sums,
\item[\rm(iv)]
 small direct sums of injective objects in $\shc$ are acyclic for the functor $\rho$.
\enum
Then 
the derived functor $R\rho_\mon\cl\RD(FA)\to\RD(FB)$ commutes
with small direct sums and 
 admits a right adjoint $\epb\rho_\mon\cl \RD(FB)\to\RD(FA)$. Moreover, 
 $\epb\rho_\mon$ induces a functor $\RD^+(FB)\to\RD^+(FA)$.

\spaa
{\rm(3)} We make the assumptions of {\rm Lemma~\ref{le:comp7}~(a)} and assume
moreover that $\sigma$ is fully faithful and has a right adjoint $\rho$ which
is exact. Then the derived functor $\sigma_\mon\cl\RD^*(FA)\to\RD^*(FB)$
\lp$*=\ub,+,\rb$\rp\, is well defined, is fully faithful and admits a right adjoint
$\rho_\mon\cl \RD^*(FB)\to\RD^*(FA)$  \lp$*=\ub,+,\rb$\rp.
\end{theorem}

\begin{proof}
By Theorem~\ref{th:eqv1}, it is enough to prove the statements when replacing 
$FA$ and $FB$ with $\iota FA$ and $\iota FB$, respectively and $\rho_\mon$ with $\tw\rho_\mon$.

\medskip\noindent
(1) Let us first prove that $\tw\rho_\mon\cl\md[\iota FA]\to\md[\iota FB]$
admits a derived functor and has cohomological dimension $\leq d$.

We let $\shi$ be the subcategory of $\md[\iota FA]$ which consists of the 
$I\in \md[\iota FA]$ such that $I(\lambda)$ is $\rho$-acyclic,
for all $\lambda\in \Lambda$.  Using the hypothesis~(iv) and the relation
$\for \circ \tw\rho_\mon \simeq \tw\rho \circ\for$ we see that the subcategory
$\shi$ is $\tw\rho_\mon$-injective.  Hence $R\tw\rho_\mon$ exists.
We also see that $\for(\shi)$ is a $\tw\rho$-injective family.
Hence $\for \circ R\tw\rho_\mon \simeq R\tw\rho \circ\for$.
Now the assertion on the cohomological dimension follows from
Lemma~\ref{le:comp2}-(c).

By Lemma~\ref{le:comp7}, the functor $\tw\rho_\mon$ is right adjoint to
$\tw\sigma_\mon$. This functor $\tw\sigma_\mon$ induces $\opb{\tw\rho_\mon}$ on
the derived category which is left adjoint to $R\tw\rho_\mon$.  The relation
$\tw\sigma_\mon \tw\rho_\mon \simeq \id$ gives
$\opb{\tw\rho_\mon} R\tw\rho_\mon \simeq \id$. Hence $R\tw\rho_\mon$ is fully
faithful.

\spaa
(2) By the Brown representability theorem, 
it is enough to prove that  
\eq\label{eq:commsmall}
&& \mbox{$R\tw\rho_\mon$ commutes with small direct sums. }
\eneq
We consider the functor $\tw\rho\cl \Fct(\mon,\shc)\to\Fct(\mon,\shc')$.  The
hypotheses of Proposition~\ref{pro:rhooplus} are satisfied by
Lemma~\ref{le:comp2}.  Therefore the functor $\tw\rho$ has cohomological
dimension $\leq d$ and the functor
$R\tw\rho\cl \RD(\Fct(\mon,\shc))\to\RD(\Fct(\mon,\shc'))$ commutes with small
direct sums.

Now we prove~\eqref{eq:commsmall}. Let $\{X_i\}_{i\in I}$ be a family of objects
of $\RD(\iota FA)$.  There is a natural morphism
$\bigoplus_{i\in I}R\tw\rho_\mon (X_i)\to R\tw\rho_\mon(\bigoplus_{i\in I}X_i)$
in $\RD(\iota FB)$ and it follows from Lemma~\ref{le:comp9} that this morphism
is an isomorphism.

\spaa
(3) is obvious.
\end{proof}

\section{Filtrations on $\sho_\Xsal$}

In the sequel, if $FM$ is a filtered object in $\shc$ over the ordered additive
monoid $\R$, we shall write $F^sM$ instead of $(FM)(s)$ to denote the image of
the functor $FM$ at $s\in\R$.
This induces a functor $\RD(\Fil_\R\shc) \to \RD(\shc)$
denoted in the same way $FM \mapsto F^sM$.

\subsubsection*{The filtered ring of differential operators}

Recall that the sheaf $\shd_M$ of finite order differential operators on $M$ has a natural $\N$-filtration 
given by the order.

Recall that the rings $\shd_\Msa$ and $\shd_\Msal$ as well as the sheaves
$\shd_\Msa(m)$ and $\shd_\Msal(m)$ are defined in~\eqref{def:DMsa2}
and~\eqref{def:DMsal}. 
We remark that $\opb{\rhosa}(\shd_\Msa(m)) \simeq \shd_M(m)$ and
$\opb{\rhosal}(\shd_\Msal(m)) \simeq \shd_\Msa(m)$.

\begin{definition}\label{not:filtD}
Let $\sht$ be the site $M$ or $\Msa$ or $ \Msal$.
We define the filtered sheaf $\Fi\shd_\sht$  over $\mon= \R$ by setting:
\eqn
&&\Fi^s\shd_\sht=\shd_\sht([s])
\eneqn
where $[s]$ is the integral part of $s$ and $\shd_\sht([s])$ is the sheaf of differential operators of order $\leq[s]$. 
In particular, $\Fi^s\shd_\sht=0$ for $s<0$. We denote by 
$\md[\Fi\shd_\sht]$ the category of filtered modules over $\shd_\sht$. 
\end{definition}

Let $M_\sht$ be either $M$, $\Msa$ or $\Msal$. 
In the sequel, we look at $\md[\C_{M_\sht}]$ as an abelian Grothendieck tensor category with unit and 
at $\Fi\shd_{M_\sht}$ as a $\mon$-ring object in $F_\mon\shc$ ( with $\mon=\R$) and $\shc=\md[\C_{M_\sht}]$.
Note that Definition~\ref{not:filtD} is in accordance with Lemma~\ref{le:comp7}~(a)~(i).

Since $\opb{\rhosa}(\shd_\Msa(m)) \simeq \shd_M(m)$ and
$\opb{\rhosal}(\shd_\Msal(m)) \simeq \shd_\Msa(m)$ we can apply
Lemma~\ref{le:comp7}~(a) with the exact functors 
$\sigma = \opb{\rhosa}$ or $\sigma = \opb{\rhosal}$.
We obtain the functors
\begin{equation}\label{eq:opbrhosalfilt}
  \begin{split}
\opb{\rhosa}&\cl\md[\Fi\shhd_\Msa] \to \md[\Fi\shhd_M],\\
\opb{\rhosal}&\cl\md[\Fi\shhd_\Msal] \to \md[\Fi\shhd_\Msa].    
  \end{split}
\end{equation}
We will also use the fully faithful right adjoint of $\opb{\rhosal}$ given by
Lemma~\ref{le:comp7}~(b)
\begin{equation}\label{eq:oimrhosalfilt}
\oim{\rhosal}\cl\md[\Fi\shhd_\Msa] \to \md[\Fi\shhd_\Msal].
\end{equation}

\begin{theorem}\label{th:epbrhoFD}
\bnum
\item
The functor $\oim{\rhosal}$ in~\eqref{eq:oimrhosalfilt}  admits a right derived functor
$\roim{\rhosal}\cl\RD^*(\Fi\shhd_\Msa)\to\RD^*(\Fi\shhd_\Msal)$ \lp$*=\ub,+$\rp\,
which is fully faithful and admits a left adjoint functor
$\opb{\rhosal}\cl\RD^*(\Fi\shhd_\Msal)\to\RD^*(\Fi\shhd_\Msa)$ \lp$*=\ub,+$\rp.
\item 
The functor $\roim{\rhosal}$ \lp$*=\ub,+$\rp\, commutes with small direct sums
and admits a right adjoint
$\epb{\rhosal}\cl\RD^*(\Fi\shhd_\Msal)\to\RD^*(\Fi\shhd_\Msa)$
\lp$*=\ub,+$\rp.
\item
The functor $\opb{\rhosa}\cl\RD^+(\Fi\shhd_\Msa)\to\RD^+(\Fi\shhd_M)$ has a
fully faithful right adjoint
$\reim{\rhosa}\cl\RD^+(\Fi\shhd_M)\to\RD^+(\Fi\shhd_\Msa)$.
\enum
\end{theorem}

\begin{proof}
(i)--(ii) We shall apply Theorem~\ref{th:abstractBrown}~(1)-(2) with
$\shc=\md[\C_\Msa]$, $\shc'=\md[\C_\Msal]$, $\rho=\oim{\rhosal}$,
$\sigma=\opb{\rhosal}$, $\mon=\R$, $FA=\Fi\shd_\Msa$, $FB=\Fi\shd_\Msal$.  Let
us check hypotheses (i)--(iv) of Theorem~\ref{th:abstractBrown}.
Hypothesis~(i) follows from Proposition~\ref{pro:cohdimrho}. The
hypotheses~(iii) and~(iv) follow from Lemma~\ref{le:sectUindlim}. By
Lemma~\ref{le:comp9} we know that $\md[\iota \Fi\shd_\Msa]$ has enough
injectives.  Hence to check the hypothesis (ii) it is enough to prove that if
$I\in \md[\iota \Fi\shd_\Msa]$ is injective, then $I(\lambda)$ is
$\oim{\rhosal}$-acyclic for any $\lambda\in \Lambda$.

By Lemmas~\ref{le:rhodimfB} and~\ref{le:flabby} it is enough to prove that
$I(\lambda)$ is flabby. For any $U\in\Op_\Msa$ we have
\eq\label{eq:sectUlambda}
\sect(U; I(\lambda))
 \simeq 
\Hom[{\md[\iota \Fi\shd_\Msa]}]( (\shd_\Msa^{[-\lambda]})_U , I) ,
\eneq
where $\shd_\Msa^{[-\lambda]}$ denotes the object $\iota \Fi\shd_\Msa$ with the
filtration shifted by $\lambda$, that is,
$F^s\shd_\Msa^{[-\lambda]} = F^{s-\lambda}\shd_\Msa$;
this isomorphism sends a section $s$ of $I(\lambda)$ to the morphism
$1\mapsto s$ (which is filtered because $1 \in F^\lambda\shd_\Msa^{[-\lambda]}$). 
Hence the flabbiness of $I(\lambda)$
follows from the injectivity of $I$ and the exact sequence
$0 \to (\shd_\Msa^{[\lambda]})_U \to (\shd_\Msa^{[\lambda]})_V$, for any inclusion
$U\subset V$. 
This completes the proof of (i)--(ii). 

\spaa 
(iii) We apply Theorem~\ref{th:abstractBrown}~(3) with $\rho=\opb{\rhosa}$,
$\sigma=\eim{\rhosa}$.
\end{proof}

We define a functor
\eqn
&&\Fi\hom\cl \mdrc[\C_M]\times \md[\Fi\shd_\Msa]\to\md[\Fi\shd_\Msa]
\eneqn
by setting for  $G\in\mdrc[\C_M]$ and $\Fi \shm\in\md[\Fi\shd_\Msa]$
\eqn
&& \hom(G,\Fi \shm)(\lambda)=\hom(G,\shm(\lambda)).
\eneqn
Using Theorem~\ref{th:eqv1},  this functor admits a derived functor
\eqn
&& \Fi\rhom\cl \Derb_\Rc(\C_M)\times \RD^+(\Fi\shd_\Msa)\to\RD^+(\Fi\shd_\Msa).
\eneqn

Recall the functor $\for$ in~\eqref{eq:colimFA4}.
\begin{lemma}\label{le:rhomcomfor}
Let $G\in\Derb_\Rc(\C_M)$ and let $\Fi \shm\in\RD^+(\Fi\shd_\Msa)$. Then 
\eqn
&\Fi^\lambda \rhom(G, \Fi \shm)\simeq \rhom(G,\Fi^\lambda \shm), \\
&\for\Fi\rhom(G, \Fi \shm)\simeq \rhom(G,\for\Fi \shm).
\eneqn
 \end{lemma}
\begin{proof}
The first isomorphism follows directly from Lemma~\ref{le:comp3}~(b) and
we only prove the second one.

\spa
(i) Since the problem is local on $M$, we may assume that $G$ has compact support. 

\spa
(ii) By standard arguments, we may then reduce to the case where $G=\C_U$, $U\in\Op_\Msa$. 

\spa
(iii) Using Theorem~\ref{th:eqv1}, we may replace $\Fi \shm\in\RD^+(\Fi\shd_\Msa)$ with 
an object $\tw \shm\in \RD^+(\Fct(\R,\md[\C_\Msa]))$. Let us represent $\tw \shm$ by a complex of injective objects 
$I^\scbul\in \RC^+(\Fct(\R,\md[\C_\Msa]))$. Then,
\eqn
\for\Fi\rhom(\C_U, \Fi \shm)&\simeq&\sindlim\rsect(\C_U, \tw \shm)\\
&\simeq&\sindlim\sect(U;I^\scbul)\\
&\underset{\mathrm{(a)}}\simeq& \sect(U;\sindlim I^\scbul)
\underset{\mathrm{(b)}}\simeq \rsect(U;\sindlim I^\scbul)\\
&\simeq& \rsect(U;\sindlim\tw \shm)\simeq \rsect(U;\for\Fi \shm). 
\eneqn
Isomorphism (a) follows from Lemma~\ref{le:sectUindlim} and  isomorphism (b) follows from Lemma~\ref{le:comp3}~(b) and Corollary~\ref{cor:limit-Gam-acyc}.
\end{proof}

On a complex manifold $X$, we endow the $\shd_X$-module $\sho_X$ with the
filtration $\Fi\sho_X$ given by
\eq\label{eq:FO}
&&\Fi^s\sho_X=\begin{cases}
0&\mbox{ if }s<0,\\
\sho_X&\mbox{ if }s\geq0.
\end{cases}
\eneq
By applying the  functors $\eim{\rhosa}$ and $\oim{\rhosal}$, we get the objects $\eim{\rhosa}\sho_X$ and $\oeim{\rhosl}\sho_X$
of $\md[F\shd_\Xsa]$ and $\md[F\shd_\Xsal]$, respectively. One shall be aware that these objects are in degree $0$ 
contrarily to the sheaf $\sho_\Xsa$ (when $d_X>1$). 

\subsubsection*{The $L^\infty$-filtration on $\Cinft_\Msal$}

Recall that on the site $\Msal$, the sheaf $\Cinftt_\Msal$ is 
endowed with a filtration, given by the sheaves 
$\Cin[t]_\Msal$ ($t\in\R_{\geq0}$). We also set 
\eqn
&&\Cin[t]_\Msal=0\mbox{ for }t<0.
\eneqn
Using Lemma~\ref{le:shdmcinft} and Theorem~\ref{th:epbrhoFD}, we set:
\begin{definition}\label{def:FtepCinf}
\banum
\item
We denote by $\Finf \Cinft_\Msal$ the object of $\md[F\shhd_\Msal]$
given by the sheaves $\Cin[t]_\Msal$ ($t\in\R$). 
\item
We set $\Finf\Cinft_\Msa\eqdot\epb{\rhosal}\,\Finf\Cinft_\Msal$, an object of $\RD^+(\Fi\shhd_\Msa)$. 
\eanum
 We call these filtrations the $L^\infty$-filtration on 
$\Cinft_\Msal$ and $\Cinft_\Msa$, respectively.
\end{definition}
Hence, 
\begin{itemize}
\item
$\Finf^s\Cinft_\Msal=\Cin[s]_\Msal$ for $s\in\R$,
\item
we have morphisms $\Fi^r\shd_\Msal\tens \Finf^s\Cinft_\Msal\to \Finf^{s+r}\Cinft_\Msal$,
\item
using Notation~\ref{not:colimFA},
$\for \Finf\Cinft_\Msal\simeq\Cinftt_\Msal$ and similarly with $\Msa$ instead of $\Msal$. 
\end{itemize}

If $U\in\Op_\Msa$ is weakly Lipschitz, we thus have for $s\geq0$:
\eq
&&\rsect(U;\Finf^s\Cinft_\Msa)\simeq\Cin[s]_M(U).
\eneq

\begin{remark}\label{rem:L2}
One could have also endowed $\Cinft_\Msal$ with the $L^2$-filtration constructed similarly as the $L^\infty$-filtration, when replacing the norm in~\eqref{eq:linftynorm} with the $L^2$-norm:
\eq\label{eq:ltwonorm}
&&\vvert\phi\vvert_2=(\int_U\vert \phi(x)\vert^2dx)^{1/2},\quad
\vvert\phi\vvert^s_2=\vvert d(x)^s\phi(x)\vvert_2.
\eneq
One gets the filtered sheaves  $\Fi_2\Cinft_\Msal$ and $\Fi_2\Cinft_\Msa$. 
\end{remark}

\subsubsection*{The $L^\infty$-filtration on $\Ot_\Xsal$}
On a complex manifold $X$, we set:
\eq
\hspace{2ex}\Finf\Ot_\Xsal&\eqdot&\rhom[\Fi\shhd_\olXsal](\oeim{\rhosl}\sho_{\ol  X}, \Finf\Cinft_\Xsal)\in \RD^+(\Fi\shhd_\Xsal),\label{eq:FOsal}\\
\hspace{2ex}\Finf\Ot_\Xsa
&\eqdot&\rhom[\Fi\shhd_\olXsa](\eim{\rhosa}\sho_{\ol  X}, \Finf\Cinft_\Xsa)
\label{eq:FOsa}\\
&\simeq&\epb{\rhosal}\,\Finf\Ot_\Xsal  \in \RD^+(\Fi\shhd_\Xsa).\nonumber
\eneq

\begin{proposition}\label{pro:spencerfiltration}
The object $\Finf^s\Ot_\Xsal$ is represented by the complex of sheaves on $\XRsal$:
\eq\label{eq:FOs}
&&0\to \Finf^s\Cin[(0,0)]_\Xsal\to[\ol\partial] \Finf^{s+1}\Cin[(0,1)]_\Xsal\to\cdots\to
\Finf^{s+d_X}\Cin[(0,d_X)]_\Xsal\to0.
\eneq
\end{proposition}
\begin{proof}
Recall that the Spencer complex $\SP_X(\shd_X)$ is the complex of left $\shd_X$-modules 
\eq\label{eq:SP1}
&&
\SP_X(\shd_X)\eqdot\quad 
0\to\shd_X\tens[\sho]\bigwedge^{d_x}\Theta_X\to[d]\cdots\to
\shd_X\tens[\sho]\Theta_X\to \shd_X\to 0.
\eneq
Moreover, there is an isomorphism of complexes, in any local chart, 
\eq\label{eq:SPkoszul}
&&\SP_X(\shd_X)\simeq K_\bullet(\shd_X;\cdot\partial_1,\dots,\cdot\partial_{d_X})
\eneq
where the right hand side is the co-Koszul complex 
of the  sequence $\cdot\partial_1,\dots,\cdot\partial_{d_X}$ 
acting on the right on $\shd_X$. This implies 
that the left $\shd$-linear morphism $\shd_X\to \sho_X$ induces an
isomorphism 
$\SP_X(\shd_X)\isoto \sho_X$ in $\RD^\rb(\shd_X)$.

If we endow $\shd_X\tens[\sho]\bigwedge^k\Theta_X$, $k=0,\ldots,d_X$, with the
filtration $\Fi^s(\shd_X\tens[\sho]\bigwedge^i\Theta_X) =
\Fi^{s-k}(\shd_X)\tens[\sho]\bigwedge^i\Theta_X$, then $\SP_X(\shd_X)$ gives a
complex in $\md[\Fi\shd_\sht]$ and we obtain $\SP_X(\shd_X)\isoto \sho_X$ in
$\RD^\rb(\Fi\shhd_X)$. Applying this to $\ol X$ and using the
definition~\eqref{eq:FOsa} we obtain the result. 
\end{proof}

\begin{corollary}
Let $U\subset X$ be an open relatively compact subanalytic subset. Assume that $U$ is weakly Lipschitz. Then 
the object $\rsect(U;\Finf^s\Ot_\Xsa)$ is represented by the complex 
\eq\label{eq:FOsU}
&&\\
&&0\to \Cin[s,(0,0)]_X(U)\to[\ol\partial]\Cin[s+1,(0,1)]_X(U)\to\cdots\to\Cin[s+d_X,(0,d_X)]_X(U)\to0.\nonumber
\eneq
\end{corollary}
Applying the functor $\opb{\rhosa}$, one recovers the filtration introduced in~\eqref{eq:FO}:
\eq
&&\opb{\rhosa}\Finf\Ot_\Xsa\simeq \Fi\sho_X.
\eneq

\section[A filtration on regular holonomic modules]{A functorial filtration on regular holonomic modules}

Good filtrations on  holonomic modules already exist in the literature, in the regular case (see~\cite{KK81,BK86,Sa88,Sa90}) and also in the irregular case (see~\cite{Ma96}). But these filtrations are constructed on each holonomic module and are by no means functorial. Here we directly construct objects of $\RD^+(\Fi\shd_X)$, the derived category of 
filtered $\shd$-modules. 

Denote by $\Derb_{\holreg}(\shd_X)$ the full triangulated subcategory of $\Derb(\shd_X)$ consisting of objects with 
regular holonomic cohomology. 
To  $\shm\in\Derb_{\holreg}(\shd_X)$, one associates 
\eqn
&&\Sol(\shm)\eqdot \rhom[\shd](\shm,\sho_X).
\eneqn
We know by~\cite{Ka75} that $\Sol(\shm)$ belongs to $\Derb_{\Cc}(\C_X)$, that is, $\Sol(\shm)$  has  $\C$-constructible cohomology. Moreover, one can recover $\shm$ from $\Sol(\shm)$ by the formula:
\eq\label{eq:RHcorr}
\shm\simeq \opb{\rhosa}\rhom(\Sol(\shm),\Ot_\Xsa).
\eneq
This is the Riemann-Hilbert correspondence obtained by Kashiwara in~\cite{Ka80,Ka84}.

Using the filtration $\Finf\Ot_\Xsa$ on $\sho_\Xsa$ we can set:
\begin{definition}\label{th:Filt1}
Let  $\shm$ be a regular holonomic module. We define the filtered Riemann-Hilbert functors $\RHFinfsa$ and 
$\RHFinf$ by the formulas
\eqn
\RHFinfsa\cl\RD_{\holreg}^+(\shd_X)&\to&\RD^+(\Fi\shd_\Xsa),\\
\hspace{5ex}\shm&\mapsto&\Fi\rhom(\Sol(\shm), \Finf\Ot_\Xsa),\\
\RHFinf= \opb{\rhosa}\RHFinfsa\cl \RD_{\holreg}^+(\shd_X)&\to&\RD^+(\Fi\shd_X).
\eneqn
\end{definition}
Note that $\RHFinfsa$ and $\RHFinf$ are  triangulated functors.

Recall Notation~\ref{not:colimFA} and the functor $\for$.

\begin{proposition}
In the diagram below
\eqn
&& \Derb_{\holreg}(\shd_X)\to[ \RHFinf] \RD^+(\Fi\shd_X)\to[\for] RD^+(\shd_X)
\eneqn 
 the composition is isomorphic to the identity functor.
 \end{proposition}
 \begin{proof}
 Since $\opb{\rhosa}$ commutes with inductive limits, the diagram below commutes: 
 \eqn
 &&\xymatrix{
  \Derb_{\holreg}(\shd_X)\ar[rr]^-{\RHFinfsa}
           &&\RD^+(\Fi\shd_\Xsa)\ar[r]^-\for\ar[d]_{\opb{\rhosa}}&\RD^+(\shd_\Xsa)\ar[d]_{\opb{\rhosa}}\\
            &&\RD^+(\Fi\shd_X)\ar[r]^-\for&\RD^+(\shd_X).
}\eneqn
Now let $\shm\in \Derb_{\holreg}(\shd_X)$ and set for short $G=\Sol_X(\shm)$. By using Lemma~\ref{le:rhomcomfor}
we get
\eqn
\for\Fi\rhom(G,\Finf\Ot_\Xsa)&\simeq &\rhom(G,\for\Finf\Ot_\Xsa)\\
&\simeq&\rhom(G,\Ot_\Xsa)
\eneqn
and we conclude with~\eqref{eq:RHcorr}.
 \end{proof}

 \begin{notation}
 The module $\shm$ endowed with the filtration obtained by applying the functor  $\RHFinfsa$ or $\RHFinf$, 
 will simply be denoted by  $\Finfsa\shm$ or $\Finf\shm$, respectively. 
  \end{notation}

\begin{example}\label{exa:normalform}
Let $D$ be a normal crossing divisor in $X$ and let $\shm$ be  a regular holonomic module such that 
$\Sol(\shm)\simeq\C_{X\setminus D}$. Let $W\in\Op_\Xsa$ with smooth boundary transversal to the strata of 
$D$ so that $W\setminus D$ is weakly Lipschitz. Set $U\eqdot W\setminus D$. 
Then, by Lemma~\ref{le:rhomcomfor},
$\rsect(W;\Finfsa^s\shm)\simeq\rsect(U;\Finf^s\Ot_\Xsa)$ and therefore 
the object 
$\rsect(W;\Finf^s\shm)$ is represented by the complex~\eqref{eq:FOsU}.
\end{example}

 \begin{remark}\label{rem:L2bis}
By using the filtration $\Fi_2$ on $\Cinft_\Xsal$ (see Remark~\ref{rem:L2}),
one can also endow $\Ot_\Xsal$ with an $L^2$-filtration and define similarly $\Fi_2\Ot_\sal$. 
Unfortunately, H\"ormander's theory does not apply  immediately to this situation. 
More precisely, for $U$ open in $\R^n$, denote by $L^2(U;loc)$ the space of functions $\phi$  which are locally 
in $L^2$ for the Lebesgue measure and define 
\eq\label{eq:ltwonorm2}
&&L^{2,s}(U)=\{\phi\in L^2(U;loc);\vvert\phi\vvert^s_2<\infty,\},
\eneq
where $\vvert\phi\vvert^s_2$ is defined in~\eqref{eq:ltwonorm2}.

For $U$ relatively compact and  open in $\C^n$, denote by $W^{2,s,(p,q)}(U)$ the space of $(p,q)$-forms 
with coefficients in $L^{2,s}(U)$  and set 
\eqn
&&W_0^{2,s,(p,q)}(U)=\{\phi\in W^{2,s,(p,q)}(U);\ol\partial\phi\in W^{2,s,(p,q+1)}(U)\}.
\eneqn
Now we define $\tw\Fi_2\Ot_\Xsal$ as the Dolbeault complex 
\eqn\label{eq:FO2s}
&&\tw\Fi^s_2\Ot_\Xsal(U)\eqdot
0\to W_0^{2,s,(0,0)}(U)\to[\ol\partial]\cdots\to[\ol\partial] W_0^{2,s,(0,n)}(U)\to0.
\eneqn
Then \cite[Th~2.2.3]{Ho65} asserts that if $U$ is pseudoconvex,
$\tw\Fi^s_2\Ot_\Xsal(U)$ is concentrated in degree $0$.
However $\Fi^m\shd_\Xsal$ does not send $W_0^{2,s,\cdot}$ in
$W_0^{2,s+m,\cdot}$ and $\tw\Fi_2\Ot_\Xsal$ is not defined as an object of
$\RD(\Fi\shd_\Xsal)$.  
\end{remark}

Given a regular holonomic $\shd_X$-module $\shm$,  natural questions arise.
\bnum  
\item\vspace{-1.ex}
Does there exist an integer $r$ such that $H^j(\Finf^s\shl)\to H^j(\Finf^{s+r}\shl)$ is the zero morphism for
$s\gg0$ and  $j\not=0$.
\item\vspace{-1.ex}
Is the filtration  $H^0(\Finf\shm)$ a good filtration?
\item\vspace{-1.ex}
Does there exist a discrete set $Z\subset \R_{\geq0}$ such that
the morphisms $\Finf^s\shm\to \Finf^t\shm$ ($s\leq t$) are isomorphisms for $[s,t]$ contained in a connected component
 of $\R_{\geq0}\setminus Z$?
\enum
Note that it may be convenient to use better the  $L^2$-filtration  
(see Remark~\ref{rem:L2bis}).

One can also ask the question of comparing these filtrations with other filtrations already existing in the literature.

\providecommand{\bysame}{\leavevmode\hbox to3em{\hrulefill}\thinspace}

\vspace*{1cm}
\noindent
\parbox[t]{21em}
{\scriptsize{
\noindent
St{\'e}phane Guillermou\\
Institut Fourier, Universit{\'e} de Grenoble I, \\
email: Stephane.Guillermou@ujf-grenoble.fr\\
%

\medskip\noindent
Pierre Schapira\\
Sorbonne Universit{\'e}s, UPMC Univ Paris 6\\
Institut de Math{\'e}matiques de Jussieu\\
e-mail: pierre.schapira@imj-prg.fr\\
http://www.math.jussieu.fr/\textasciitilde schapira/
}}


\begin{bibdiv}
\begin{biblist}


\bib{BK86} {article}{
author={Barlet, Daniel},
author={Kashiwara, Masaki},
title={Le r\'eseau $L^2$ d'un syst\`eme holonome r\'egulier},
journal={Invent. Math.},
volume= {86}, 
date={1986}, 
pages={35-62},
}

\bib{BM88}{article}{
author={Bierstone, Edward}
author={Milman, Pierre D.},
title={Semi-analytic and subanalytic sets},
journal={Publ. Math. IHES},
volume= {67}, 
date={1988},
pages={5-42},
}


\bib{EP10} {article}{
author={Edmundo, Mario J.}
author={Prelli, Luca},
title={Sheaves on $\cal T$-topology },
journal={Journ.  Math. Soc. Japan },
eprint={arXiv:1002.0690}, 
}


\bib{HM11} {article}{
author={Honda, Naofumi}
author={Morando, Giovani},
title={Stratified Whitney jets and tempered ultradistibutions on the subanalytic site},
journal={Bull. Soc. Math. France},
volume={139},
date={2011},
pages={ 923-943},
} 


\bib{Ho65}{article}{
author={ H\"ormander, Lars }
title={$L^2$-estimates and existence theorems for the $\ol\partial$ operator},
journal={Acta Mathematica},
volume={113},
date={1965 },
pages={ 89-152},
}


\bib{Ho83}{book}{
author={H\"ormander, Lars}
title={The analysis of linear partial differential operators I,II}
series={Grundlehren der Math. Wiss}
publisher={Springer-Verlag}
volume={256, 257}
date={1983}
}



\bib{Ka75}{article}{
author={Kashiwara, Masaki},
title={On the maximally overdetermined systems of linear differential equations I},
journal={Publ. Res. Inst. Math. Sci.},
volume={10},
date={1975},
pages={563-579},
}

\bib{Ka80}  {article}{
author={Kashiwara, Masaki},
title={Faisceaux constructibles et syst{\`e}mes holon{\^o}mes d'{\'e}quations aux d{\'e}riv{\'e}es partielles lin{\'e}aires {\`a} points singuliers r{\'e}guliers},
journal={S{\'e}minaire Goulaouic-Schwartz, exp 19},
date={1980},
publisher={{\'E}cole Polytech., Palaiseau},
}

\bib{Ka84}  {article}{
author={Kashiwara, Masaki},
title={The Riemann-Hilbert problem for holonomic systems},
journal={Publ.\ RIMS, Kyoto Univ. },
volume={ 20},
date={1984},
pages={319-365},
}

\bib{Ka03}{book}{
 author={Kashiwara, Masaki},
 title={$D$-modules and microlocal calculus},
 series={Translations of Mathematical Monographs},
 volume={217},
publisher={American Mathematical Society, Providence, RI},
 date={2003},
 pages={xvi+254},
}

\bib{KK81}{article}{
 author={Kashiwara, Masaki},
 author={Kawai, Takahiro},
title={On holonomic systems of microdifferential equations III, Systems with regular singularities},
journal={Publ. Rims, Kyoto Univ. },
volume={17}, 
date={1981},
pages={813-979},  
}

\bib{KS90}{book}{
 author={Kashiwara, Masaki}
 author={Schapira, Pierre},
 title={Sheaves on manifolds},
 series={Grundlehren der Mathematischen Wissenschaften},
 volume={292},
 publisher={Springer-Verlag, Berlin},
 date={1990},
 pages={x+512},
}

\bib{KS96} {book}{
 author={Kashiwara, Masaki}
 author={Schapira, Pierre}
title={Moderate and formal cohomology associated with constructible sheaves}
series={M{\'e}moires Soc. Math. France}
volume={64}
date={1996}
}

\bib{KS01} {book}{
author={Kashiwara, Masaki},
 author={Schapira, Pierre},
title={Ind-sheaves},
series={Ast\'erisque},
volume={271},
publisher={Soc. Math. France},
date={2001},
journal = {ArXiv e-prints},
note = {\tt arXiv:1003.3304},
}


\bib{KS03} {article}{
 author={Kashiwara, Masaki},
 author={Schapira, Pierre},
title= {Microlocal study of ind-sheaves I: micro-support and regularity},
journal={Ast\'erisque},
volume={284},
publisher={Soc. Math. France},  
 date={2003},
 pages={143-164},
 }
 
 
\bib{KS06}{book}{
   author={Kashiwara, Masaki},
   author={Schapira, Pierre},
   title={Categories and sheaves},
  series={Grundlehren der Mathematischen Wissenschaften}
   volume={332},
   publisher={Springer-Verlag, Berlin}
   date={2006},
   pages={x+497},
}


\bib{KS15}{book}{
   author={Kashiwara, Masaki},
   author={Schapira, Pierre},
   title={Regular and irregular holonomic D-modules},
   date={2015}
journal = {ArXiv e-prints},
     note = {\tt arXiv:1507.00118}
     }


\bib{Ke10} {article}{
author={Kedlaya, Kiran S.},
title={Good formal structures for flat meromorphic connections, I: Surfaces},
journal={Duke Math. J.},
volume={154},
date={2010},
pages={343-418},
}
 
 \bib{Ke11} {article}{
author={Kedlaya, Kiran S.},
title={Good formal structures for flat meromorphic connections,  II: Excellent schemes},
journal={J. Amer. Math. Soc.} ,
volume={24},
date={2011},
pages={183-229},
}

\bib{Ko73a} {article}{
author={Komatsu, Hikosaburo},
title={On the regularity of hyperfunction solutions of linear ordinary differential equations with real analytic coefficients},
journal={J. Fac. Sci. Univ. Tokyo},
series={ Sect Math},
volume={20}, 
date={1973},
pages={107-119},
}

\bib{Ko73b}  {article}{
author={Komatsu, Hikosaburo},
title={Ultradistributions I. Structure theorem and a characterisation,}
journal={J. Fac. Sci. Univ. Tokyo},
series={Sect Math},
volume={20},
date={1973},
 }
 
\bib{Le14} {article}{
author={Lebeau, Gilles},
title={Sobolev spaces and Sobolev sheaves}, 
journal={Ast\'erisque, Soc. Math. France},
publisher={Soc. Math. France,}
note={this volume}
date={2016}
}


\bib{Lo59} {article}{
author={Lojaciewicz, Stanislaw}
title={Sur le probl{\`e}me de la division,}
journal={Studia Math}
volume= {8} 
date={1959} 
pages={87-136}
}

\bib{Ma96} {article}{
author={Malgrange, Bernard}
title={Connexions m\'eromorphes II, le r\'eseau canonique,} 
journal={Inventiones Math }
volume={124} 
date={1996}
pages={ 367-387}
}

\bib{Mo09} {article}{
author={Mochizuki, Takuro}
title={ Good formal structure for meromorphic flat connections on smooth 
projective surfaces,} 
series={Algebraic Analysis and around, Advances Studies in Pure Math. }
volume={54}, 
publisher={ Math. Soc. Japan}
date={2009}
pages={223-253}
}


\bib{Mr13} {article}{
author={Morando, Giovani},
title={Constructibility of tempered solutions of holonomic D-modules},
date={2013},
journal = {ArXiv e-prints},
     note = {\tt arXiv:1311.6621}
}

\bib{Pa14} {article}{
author={Parusinski, Adam}
title={Regular covers of open relatively compact subanalytic sets},
journal={Ast\'erisque, Soc. Math. France},
note={This volume}
publisher={Soc. Math. France},
date={2016}
}

\bib{Pr08} {article}{
author={Prelli, Luca}
title={Sheaves on subanalytic sites}, 
journal={Rend. Sem. Mat. Univ. Padova,}
volume={120}  
date={2008}
pages={167-216}
}

\bib{Ra78}  {article}{
author={Ramis, Jean-Pierre},
title={D\'evissage Gevrey,}
journal={Ast\'erisque},
volume={59-60},
publisher={Soc. Math. France},
 date={1978},
 pages={73-204},
}

\bib{Sa88} {article}{
author={Saito, Morihiko}
title={Modules de Hodge polarisables,} 
journal={Publ. Res. Inst. Math. Sci.}
volume=  {24} 
date={1988}
pages={849-995}
}

\bib{Sa90}{article}{
author={Saito, Morihiko}
title={Mixed Hodge modules,} 
journal={Publ. Res. Inst. Math. Sci.}
volume=  { 29} 
date={1990}
pages={221-333}
}

\bib{SSn13}{article}{
   author = {Schapira, Pierre},
   author = {Schneiders, Jean-Pierre},
    title = {Derived category of filtered objects},
journal={Ast\'erisque, Soc. Math. France}
note={This volume}
date={2016}, 
 eprint={arXiv:1306.1359},
}

\bib{Sn99}{article}{
   author={Schneiders, Jean-Pierre},
   title={Quasi-abelian categories and sheaves},
   journal={M\'em. Soc. Math. Fr. (N.S.)},
   date={1999},
   number={76},
}



\end{biblist}
\end{bibdiv}
\end{document}